\documentclass[reqno,11pt]{amsart}

\usepackage[T1]{fontenc}
\usepackage[english]{babel}
\usepackage{times}
\usepackage{amsmath}
\usepackage{amsfonts}
\usepackage{amssymb}
\usepackage{amsmath}
\usepackage{amsthm}
\usepackage{graphicx}
\usepackage{array}
\usepackage{color}
\usepackage{mathrsfs}
\usepackage{hyperref}
\usepackage{esint} 
\usepackage{tikz}
\usepackage{ upgreek }
\usepackage{enumitem}

\definecolor{darkgreen}{rgb}{0.1,0.7,0.1}
\definecolor{darkred}{rgb}{0.7,0.1,0.1}

\newtheorem{theorem}{Theorem}
\newtheorem{lemma}{Lemma}[section]
\newtheorem{proposition}[lemma]{Proposition}
\newtheorem{corollary}[lemma]{Corollary}

\newtheorem{remark}[lemma]{Remark}

\newcommand{\eps}{\varepsilon}

\newcommand\symb[2][\bf]{{\mathchoice{\hbox{#1#2}}{\hbox{#1#2}}%
        {\hbox{\scriptsize#1#2}}{\hbox{\tiny#1#2}}}}

\def\R{{\symb R}}
\def\N{{\symb N}}
\def\Z{{\symb Z}}

\def\Q{{\symb Q}}
\def\P{{\symb P}}

\def\un{\mathbf{1}}

\renewcommand{\P}{\mathbb{P}}
\newcommand{\E}{\mathbb{E}}

\newcommand{\bE}{\mathbf{E}}

\newcommand{\bP}{\mathbf{P}}

\newcommand{\bbN}{\mathbb{N}}

\newcommand{\bbP}{\mathbb{P}}

\newcommand{\bbR}{\mathbb{R}}

\newcommand{\bbZ}{\mathbb{Z}}

\newcommand{\cA}{\mathcal{A}}
\newcommand{\cB}{\mathcal{B}}
\newcommand{\cC}{\mathcal{C}}
\newcommand{\cD}{\mathcal{D}}
\newcommand{\cE}{\mathcal{E}}
\newcommand{\cF}{\mathcal{F}}
\newcommand{\cG}{\mathcal{G}}
\newcommand{\cH}{\mathcal{H}}

\newcommand{\cO}{\mathcal{O}}
\newcommand{\cP}{\mathcal{P}}
\newcommand{\cQ}{\mathcal{Q}}

\newcommand*{\Chi}{\mbox{\Large$\upchi$}}
          
\hypersetup{
citecolor=blue,
colorlinks=true, 
breaklinks=true, 
urlcolor= blue, 
linkcolor= black, 
bookmarksopen=true, 
pdftitle={Anderson Hamiltonian}, 
}

\begin{document}

\title{ \textbf{\textsc{Localization of the continuous\\Anderson Hamiltonian in $1$-d}}}

\author{Laure Dumaz}
\address{
Universit\'e Paris-Dauphine, PSL Research University, UMR 7534, CNRS, CEREMADE, 75016 Paris, France}
\email{dumaz@ceremade.dauphine.fr}

\author{Cyril Labb\'e}
\address{
Universit\'e Paris-Dauphine, PSL Research University, UMR 7534, CNRS, CEREMADE, 75016 Paris, France}
\email{labbe@ceremade.dauphine.fr}

\vspace{2mm}

\date{\today}

\maketitle

\begin{abstract}

We study the bottom of the spectrum of the Anderson Hamiltonian $\cH_L := -\partial_x^2 + \xi$ on $[0,L]$ driven by a white noise $\xi$ and endowed with either Dirichlet or Neumann boundary conditions. We show that, as $L\rightarrow\infty$, the point process of the (appropriately shifted and rescaled) eigenvalues converges to a Poisson point process on $\R$ with intensity $e^x dx$, and that the (appropriately rescaled) eigenfunctions converge to Dirac masses located at independent and uniformly distributed points. Furthermore, we show that the shape of each eigenfunction, recentered around its maximum and properly rescaled, is given by the inverse of a hyperbolic cosine. We also show that the eigenfunctions decay exponentially from their localization centers at an explicit rate, and we obtain very precise information on the zeros and local maxima of these eigenfunctions. Finally, we show that the eigenvalues/eigenfunctions in the Dirichlet and Neumann cases are very close to each other and converge to the same limits.

\medskip

\noindent
{\bf AMS 2010 subject classifications}: Primary 60H25, 60J60; Secondary 35P20. \\
\noindent
{\bf Keywords}: {\it Anderson Hamiltonian; Hill's operator; localization; Riccati transform; Diffusion.}
\end{abstract}

\setcounter{tocdepth}{1}
\tableofcontents

\section{Introduction}

Consider the Anderson Hamiltonian
\begin{equation}\label{Eq:Hamiltonian}
\cH_L = -\partial^2_x + \xi\;,\quad x\in(0,L)\;,
\end{equation}
endowed with either Dirichlet or Neumann boundary conditions. Here, the potential $\xi$ is taken to be a real white noise on $(0,L)$, that is, a mean zero, delta-correlated Gaussian field on $(0,L)$. The operator $\cH_L$ is a random Schr\"odinger operator, sometimes called Hill's operator, that models disordered solids in physics.

\smallskip

There is a competition between the two terms appearing in the operator: while the eigenfunctions of the Laplacian are spread out over the whole box, the multiplication-by-$\xi$ operator tends to concentrate the mass of the eigenfunctions in very small regions. In his seminal article~\cite{Anderson58}, Anderson showed that for a discrete version of the present hamiltonian and in dimension $3$, the bottom of the spectrum consists of \emph{localized} eigenfunctions. This phenomenon, now referred to as \emph{Anderson localization}, has been the object of numerous studies, see for instance~\cite{Kirsch} for an extended survey.



In the present paper, we establish a localization phenomenon at the bottom of the spectrum of $\cH_L$ when $L\rightarrow\infty$.

\medskip

This operator was first studied by the physicist Halperin \cite{Halperin}, see also the work of Frisch and Lloyd \cite{FrischLloyd}.
The focus in these works was on the macroscopic picture of the eigenvalues in the large $L$ limit. If $N(\lambda)$ denotes the number of eigenvalues smaller than $\lambda$, the density of states is defined as the derivative of the large $L$ limit of $N(\lambda)$ divided by $L$. They found that the density of states of the operator $\cH_L$ admits an explicit integral formula, see Equation \eqref{eqDensityOfStates} or Figure \ref{fig:DensityOfStates}. 
\begin{figure}[!h]\label{fig:DensityOfStates}
\centering
\includegraphics[width = 6cm]{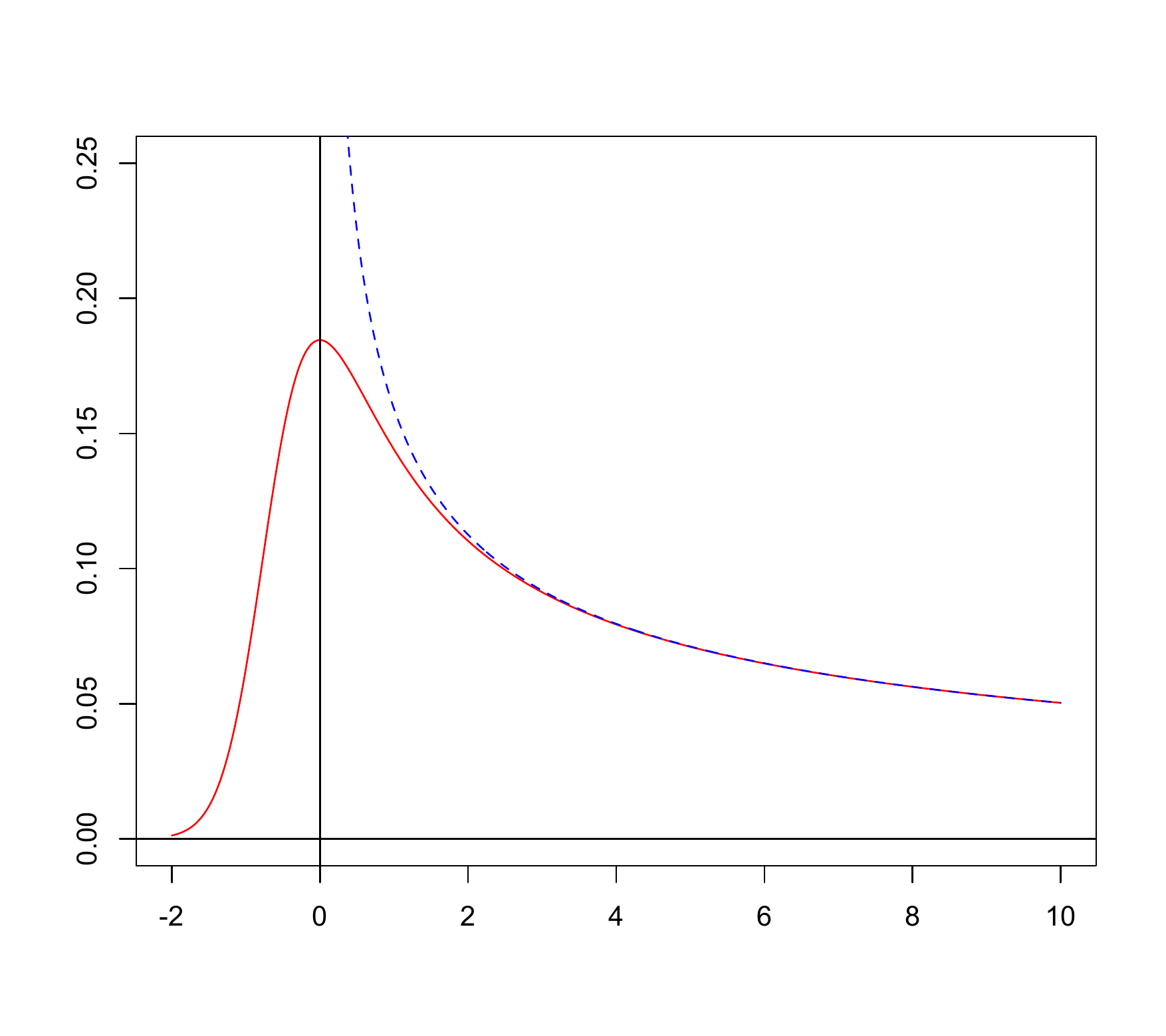}
\caption{Density of states of the operator $\mathcal{H}_L$ (plain line) and of $u \mapsto -\partial_x^2 u$ (dashed line).}
\end{figure}

Later on, Fukushima and Nakao~\cite{Fukushima} gave a precise formulation of the eigenvalue problem $\cH_L \varphi = \lambda \varphi$. They proved that for any fixed $L$, almost surely $\cH_L$ is a self-adjoint operator on $L^2(0,L)$, bounded below, that admits a pure point spectrum $(\lambda_{k})_{k\ge 1}$, with $\lambda_{1} < \lambda_{2} < \ldots$, and that the associated eigenfunctions $(\varphi_{k})_{k\ge 1}$ form an orthonormal basis of $L^2(0,L)$ and are H\"older $3/2^-$ (that is, H\"older $3/2 - \eps$ for all $\eps >0$). They also rigorously derived the density of states.

Subsequently, McKean~\cite{McKean} established that the first eigenvalue, appropriately shifted and rescaled:
\begin{align}
-4 \sqrt{a_L}(\lambda_{1} + a_L)\;, \label{CVfirsteigenvalue}
\end{align}
converges as $L\to\infty$ to a Gumbel distribution and therefore falls in one of the three famous extreme-value
distributions (see also \cite{Texier} for the aymptotics of the $k$-th eigenvalue). The precise definition of $a_L$ is given right above Equation \eqref{asymp_aL}, let us simply mention that as $L\rightarrow\infty$ we have
$$ a_L \sim \Big(\frac38 \ln L\Big)^{2/3}\;.$$

Let us also cite the works~\cite{CamMcK,CamRid} where the authors obtained an exact formula for the density distribution of the first eigenvalue in terms of an integral over the circular Brownian motion, and a precise asymptotic for its left tail. In these works, $L$ is fixed and the operator is endowed with periodic boundary conditions.

In all the aforementioned studies, the starting point is the Riccati transform which maps the second order linear differential equation $\cH_L u =  \lambda \,u$ into a first order non-linear one, see below. We will also use this tool in the present paper. 

\subsection*{Statement of our results:} We are interested in the large $L$ limit behavior of the smallest eigenvalues of $\cH_L$ and of their associated eigenfunctions. We will consider the operator $\cH_L$ endowed with either Dirichlet boundary conditions: $\varphi(0) = \varphi(L)=0$, or Neumann boundary conditions: $\varphi'(0)=\varphi'(L)=0$. In the sequel, if no mention is made of the boundary conditions, then they are taken to be Dirichlet.
\medskip

Let us introduce the rescaled eigenfunctions:
\begin{equation}\label{Eq:mL}
m_{k}(dt) = L \varphi_{k}^2(tL) dt\;,\quad t\in (0,1)\;,
\end{equation}
that belong to the space $\cP$ of probability measures on $[0,1]$ endowed with the topology of weak convergence. 
We also introduce $U_{k} \in [0,L]$ be the (first, if many) point at which the eigenfunction $|\varphi_k|$ achieves its maximum.

Our first result shows that asymptotically in $L$, the eigenvalues form a Poisson point process on $\R$ while the eigenfunctions converge to Dirac masses whose locations are uniformly distributed over $[0,L]$ and independent from the eigenvalues. 
\begin{theorem}\label{Th:Main}
The sequence of random variables $\big(4\sqrt{a_L}\,(\lambda_k + a_L), U_k/L, m_k\big)_{k \geq 1}$ converges in law to $\big(\lambda_k^{\infty}, U_k^\infty, \delta_{U_k^{\infty}}\big)_{k \geq 1}$ where $\big(\lambda_k^{\infty}, U_k^\infty\big)_{k \geq 1}$ is a Poisson point process on $\R \times [0,1]$ with intensity $e^x dx \otimes dt$ (where the convergence holds for the set of sequences of elements in $\R \times [0,1] \times \cP$ endowed with the product topology).
\end{theorem}

In the following theorem, we obtain a very precise description of the asymptotic behavior of the eigenfunctions. At a macroscopic level, we show that they decay exponentially fast from their localization centers at rate $\sqrt{a_L}$. At a microscopic level, we show that the asymptotic shape around the maximum is given by the inverse of a hyperbolic cosine at a space-scale of order $(\ln L)^{-1/3}$. 

Let us set for $i \ge 1$,
$$ h_{i}(t) := \frac{\sqrt 2}{a_L^{1/4}}\Big|\varphi_i\Big(U_{i} + \frac{t}{\sqrt{a_L}}\Big)\Big|\;,\quad t\in \R\;,$$
where we implicitly set the value of this function to $0$ whenever the argument of the function on the right does not belong to $[0,L]$. Similarly, we define
$$ b_i(t) := \frac1{\sqrt{a_L}} \Big(B\Big(U_{i} + \frac{t}{\sqrt{a_L}}\Big) - B(U_i)\Big)\;,\quad t\in \R\;,$$
where $B(t) := \langle \xi , \un_{[0,t]}\rangle$, $t\in [0,L]$. Note that $B$ is a Brownian motion. We also define
$$ h(t) := \frac1{\cosh(t)}\;,\quad b(t) := -2 \tanh(t)\;,\qquad t\in \R \;.$$
Finally, denote by $z_0:= 0 < z_1 <\cdots < z_{i-1} < z_i = L$ the zeros of the eigenfunction $\varphi_i$. 

\begin{theorem}[Exponential decay and shape of the eigenfunctions]\label{Th:Shape}
For every $i\ge 1$,
\begin{itemize}
\item \emph{Exponential decay:} 
There exist deterministic constants $C_2 > C_1 > 0$ such that with probability going to $1$, for all $t \in [0,L]$:
\begin{align}
C_1 \exp(- (\sqrt{a_L}+\kappa_L)|t - U_i|)1_{\{t \in D\}} \le \Big|\frac{\varphi_i(t)}{\varphi_i(U_i)}\Big| \le C_2 \exp(- (\sqrt{a_L}-\kappa_L)|t - U_i|), \label{ExponentialDecreaseEigenfunction}
\end{align}
where
\begin{align*}
\kappa_L := \frac{\ln^2 a_L}{a_L^{1/4}},\quad D = [0,L] \backslash \cup_{k=0}^{i} \big[(z_k - \frac38 \frac{\ln a_L}{\sqrt{a_L}})\vee 0, (z_k + \frac38 \frac{\ln a_L}{\sqrt{a_L}})\wedge L\big]\;.
\end{align*}
\smallskip
\item \emph{Shape around the maximum:} The processes $h_{i}$ and $b_i$ converge to $h$ and $b$ uniformly over compact subsets of $\R$ in probability as $L\rightarrow \infty$.
\end{itemize}
\end{theorem}

\begin{remark} The lower bound in \eqref{ExponentialDecreaseEigenfunction} does not hold near the zeros of the eigenfunctions and that is the reason for the restriction to $t\in D$. Actually our proof shows that near any zero $z_k$ we have
\begin{align*}
\forall t \in \big[(z_k - \frac38 \frac{\ln a_L}{\sqrt{a_L}})\vee 0, (z_k + \frac38 \frac{\ln a_L}{\sqrt{a_L}})\wedge L\big],\quad \varphi_i(t) = \varphi'_i(z_k) \frac{\sinh(\sqrt{a_L} |t-z_k|)}{\sqrt{a_L}}\big(1 + o(1)\big)\;,
\end{align*}
with a probability going to $1$.
\end{remark}

\begin{remark}
Note that $h$ is the main eigenfunction of the operator $-\Delta + b'$ with $b'$ the derivative of $b$.
\end{remark}

When $i > 1$, we obtain further information about the eigenfunctions (see Figure \ref{Fig:Eigenfunction}):
\begin{theorem}[Local maxima and zeros] \label{Th:LocalMaxZeros}
Fix $i >1$. Let us order the locations of the maxima of the $i$ first eigenfunctions $U_{\sigma(1)} < \cdots < U_{\sigma(i)}$.
\begin{itemize}
\item \emph{Position of the local maximum:}  There exists a constant $C>0$ such that for all $k\in \{1,\ldots,i\}$, all the points where the eigenfunction $|\varphi_i|$ reaches its maximum over the time interval $[z_{k-1}, z_{k}]$ lie at distance at most $C/(\sqrt{a_L}\ln a_L)$ from $U_{\sigma(k)}$ with probability going to $1$.
\smallskip
\item \emph{Deterministic shape around the local maxima:} For all $k \in \{1,\cdots,i\}$, $\varphi_i(U_{\sigma(k)} + \frac{t}{\sqrt{a_L}})/\varphi_i(U_{\sigma(k)})$ converges to $h$ uniformly over compact subsets of $\R$ in probability. 
\smallskip
\item \emph{Zeros of the eigenfunction:} For $k \in \{1,\cdots,i\}$,
\begin{align*}
&|z_k - U_{\sigma(k)} - \frac{3}{4} \frac{\ln a_L}{\sqrt{a_L}}| \leq \frac{(\ln\ln a_L)^2}{\sqrt{a_L}},\quad \mbox{if }U_{\sigma(k)} < U_i\;, \\
&|z_{k-1} - U_{\sigma(k)} + \frac34 \frac{\ln a_L}{\sqrt{a_L}}| \leq \frac{(\ln\ln a_L)^2}{\sqrt{a_L}},\quad \mbox{if }U_{\sigma(k)} > U_i\;,
\end{align*}
with probability going to $1$.
\end{itemize}
\end{theorem}
\begin{remark}
Thanks to \eqref{ExponentialDecreaseEigenfunction}, the amplitude of the local maxima on the time interval $[z_{k-1},z_{k}]$ is of order $|\varphi_i(U_{\sigma(k)})| = |\varphi_i(U_{i})|\exp\big((-\sqrt{a_L}+O(\kappa_L)) |U_{\sigma(k)}-U_i|\big)$ with probability going to $1$.
\end{remark}

\begin{figure}[!h]
\centering
\includegraphics[width = 13cm]{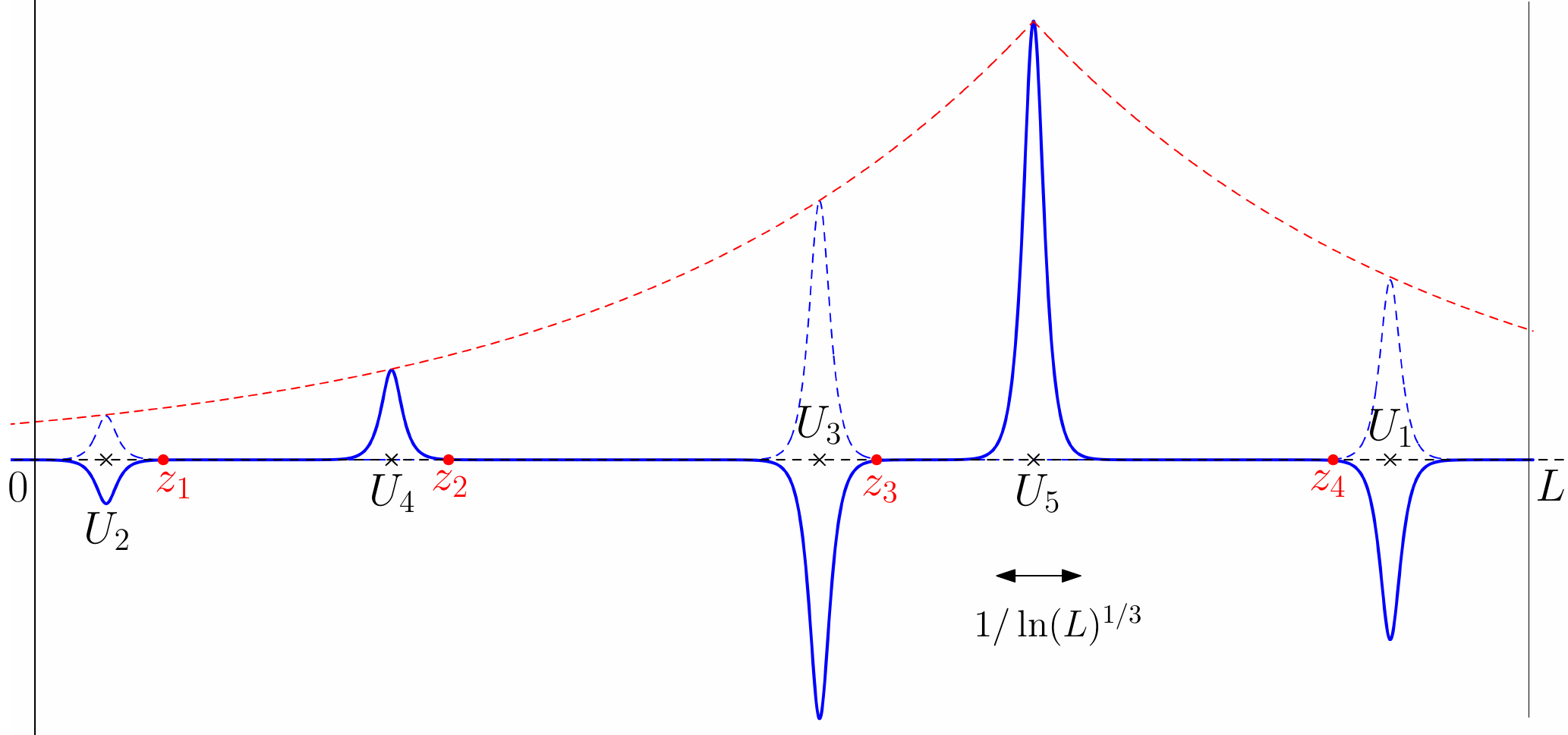}
\caption{A very schematic plot of the fifth eigenfunction $\varphi_5$ (not at scale). The main peak lies at $U_5$, while near $U_k$, for every $k<5$, the eigenfunction has peaks of smaller order. Observe that the height of the peak near $U_k$ decays exponentially with the distance $|U_k-U_5|$. Red dots correspond to the zeros of the eigenfunction.}\label{Fig:Eigenfunction}
\end{figure}

We now consider the case where $\cH_L$ is endowed with Neumann boundary conditions: we denote by $\varphi^{(N)}_{k}$ and $\lambda_{k}^{(N)}$ the corresponding eigenfunctions and eigenvalues. Notice that $\varphi^{(N)}_{k},\lambda_{k}^{(N)}$ and $\varphi_k,\lambda_{k}$ are all defined on a same probability space. Our next result shows that the limiting behavior of the eigenvalues/eigenfunctions in the Neumann case is the same as in the Dirichlet case. Furthermore, the convergence can be taken jointly for the two types of boundary conditions and the limits are the same, thus showing that the choice of Dirichlet/Neumann boundary conditions does not affect the asymptotic behavior of the eigenvalues/eigenfunctions.
\begin{theorem}\label{Th:Neumann}
For Neumann boundary conditions, the eigenvalues and eigenfunctions satisfy the same results as those stated in Theorems \ref{Th:Main}, \ref{Th:Shape} and \ref{Th:LocalMaxZeros}. Furthermore, for every $k\ge 1$,
$$(\lambda_{k}^{(N)} - \lambda_{k})\sqrt a_L\;,\quad\mbox{ and }\quad (U_{k}^{(N)}-U_{k}) \sqrt a_L\;,$$
 converge to $0$ in probability as $L\rightarrow\infty$ and consequently the limiting random variables satisfy a.s. 
 \begin{align*}
 \lambda_{k}^{(N), \infty}= \lambda_k^\infty, \quad \mbox{ and }\quad U^{(N),\infty}_{k} = U^{\infty}_k.
 \end{align*}
\end{theorem}

\medskip

\subsection*{Discussion:} Before presenting the outline of the proof, let us give some motivations for studying this operator. 
The parabolic Anderson model (PAM) is the following Cauchy problem:
\begin{align}
\partial_t u(t,x) = \partial_x^2 u(t,x) - \xi(x) u(t,x)\;, \quad u(0,x) = \delta_0(x)\;,\qquad x\in\R\;. \label{PAM}
\end{align}
One expects its solution at time $t$ to be well approximated by the solution of the same equation restricted to a segment $[-L/2,L/2]$, with $L=L(t)$ properly chosen, and endowed with Dirichlet boundary conditions. Thus, the spectral decomposition of $\cH_L$ would yield
\begin{align}
u(t,x) \approx \sum_{k = 1}^{\infty} \exp(- \lambda_k t) \varphi_k(x) \varphi_k(0)\;. \label{eq:solPAM}
\end{align}
The bottom eigenvalues and eigenfunctions of $\cH_L$ in the large $L$ limit should therefore give the asymptotic behavior of the solution \eqref{eq:solPAM} when $t \to \infty$: in particular, the solution should concentrate on a few islands corresponding to the ``supports'' of the first eigenfunctions. We will provide rigorous arguments on this heuristic discussion in a future work.

If the Anderson operator is multiplied by the imaginary unit on the right hand side of \eqref{PAM}, the equation becomes the famous Schr\"odinger equation which is of fundamental importance in quantum mechanics. In this case, one needs to study the whole spectrum (not only its bottom part). This will be the object of a forthcoming work.

\medskip
The discretization of $\cH_L$ and \eqref{PAM} has been investigated in many papers for a general dimension $d$. In this case, the Laplacian is discrete on a grid, for instance $\Z^d$, and the white noise is replaced by i.i.d random variables. The discrete operator is first defined on a finite box $B$, say $B = [0,L]^d \cap \Z^d$, and taken with Dirichlet boundary conditions. We refer to the book of K\"onig \cite{Konig} for a state of the art on the subject. Analogous results to our Theorem \ref{Th:Main} are known for some distributions of potentials, see for instance Biskup and K\"onig~\cite{BisKon}.

When $d=1$, much more precise results are known. Let us introduce the discrete Anderson Hamiltonian 
\begin{align*}
u \mapsto - \Delta_x u + \xi\, u,
\end{align*}
acting on $u \;:\; [0,L] \cap \Z \to \R$ with Dirichlet boundary conditions, and where $\Delta_x$ denotes the discrete Laplacian and $(\xi(x), \;x \in \Z)$ are i.i.d random variables.

If the variance of $\xi(0)$ does not depend on $L$, the eigenfunctions are localized \cite{carmona1987, kunz1980, GMP1977}, and the local statistics of the eigenvalues are Poissonian \cite{minami1996, molcanov1980, Killip2007}. On the other hand, if $\xi\equiv 0$ then the eigenfunctions are spread out and the eigenvalues are deterministic with locally regular spacings (clock-points).

The critical regime appears when the variance of $\xi$ is of order $1/L$. It has been considered by Kritchevski, Valk\'o and Vir\'ag in \cite{KriValVir}. They proved that delocalization holds in this case and that the eigenvalues near a fixed bulk energy $E$ have a point process limit depending on only one parameter $\tau$ (which is a simple function of the variance and energy $E$) they called \emph{Schr$_{\tau}$}. Moreover, they showed that this point process exhibits strong eigenvalue repulsion. Rifkind and Vir\'ag \cite{RifkindVirag} also studied the associated eigenfunctions. They found that the shape of the eigenfunctions near their maxima is given by the exponential of a Brownian motion plus a linear drift, and is independent of the eigenvalue. Note that heuristically, this regime corresponds to the high eigenvalues and eigenfunctions of $\mathcal{H}_L$. 

\medskip

Another famous family of one-dimensional discrete random Schr\"odinger operators is given by the tridiagonal matrices called $\beta$-ensembles. These operators seen from the edge converge, in the large dimensional limit, to the stochastic Airy operator $\mathcal{A}_{\beta}$ formally defined as:
\begin{align*}
\mathcal{A}_{\beta} : u \mapsto -\partial_x^2 u + (x + \frac{2}{\sqrt{\beta}} \xi)\; u\;.
\end{align*}
In the paper \cite{AllezDumazTW}, the analogue of the result of McKean was proved, namely that the first eigenvalue of the Airy operator properly rescaled converges to a Gumbel distribution in the small $\beta$ limit. The density of states was also derived.
Using similar techniques to the present paper, one should be able to prove the convergence of the point process of the first eigenvalues (properly rescaled) to a Poisson point process of intensity $e^x dx$. In the bulk, the eigenvalues of the $\beta$-ensembles converge towards the Sine$_{\beta}$ process \cite{ValkoViragCarousel}. In the small $\beta$ limit, the Sine$_{\beta}$ process was also shown to converge towards a (homogeneous) Poisson point process \cite{AllezDumazSine} using its characterization via coupled diffusions. Closely related discrete models are Jacobi matrices with random decaying potential \cite{KriValVir} and CMV matrices \cite{KillipStoiciu}.

\medskip

The analogue of $\cH_L$ in dimension $d$ greater than $1$ can be considered. However, due to the irregularity of the white noise, the eigenvalue problem becomes singular already in dimension $2$ and one has to renormalize the operator by infinite constants. This has been carried out under periodic b.c.~by Allez and Chouk~\cite{AllezChouk} in dimension $2$, and by Gubinelli, Ugurcan and Zachhuber~\cite{Max2} in dimension $3$ by means of the recently introduced paracontrolled calculus~\cite{Max}; the construction of the operator under Dirichlet b.c.~in dimensions $2$ and $3$ has been performed in~\cite{LabbeAnd} using the theory of regularity structures~\cite{Hairer2014}.

\medskip

Another generalization would be to consider the \emph{multivariate} Anderson Hamiltonian of the form $-\partial_x^2 + W'$, operating on the vector-valued function space $L^2([0,L], \R^r)$ and where $W'$ is the derivative of a matrix valued Brownian motion. This study has been done in the case of the stochastic Airy operator by Bloemendal and Vir\'ag \cite{BloemendalVirag}. The eigenvalues of the multivariate Anderson Hamiltonian are characterized by a family of coupled SDEs studied by Allez and one of the authors in \cite{AllezDumazCubic}.

\subsection*{Main ideas of the proof:} We now present the main arguments of the proof. For the sake of clarity, we restrict ourselves to the case of Dirichlet boundary conditions.

\subsubsection*{The Riccati transform}

For any $a\in \R$, the differential equation $-f'' + f\xi = -a f$ on $[0,L]$ with initial condition $f(0) = 0, f'(0)=1$ admits a unique solution (by classical ODE arguments). The pair $(-a,f)$ is then an eigenvalue/eigenfunction of the operator $\cH_L$ if and only if $f(L) = 0$. A very convenient tool for studying the solutions of this differential equation is the so-called \emph{Riccati transform} that maps $f$ onto $X_a := f' / f$. One can check that $X_a$ starts from $X_a(0) = +\infty$ and solves
\begin{align}\label{Eq:Diff}
dX_a(t) &= (a - X_a(t)^2) dt + dB(t)\;,\qquad t\in [0,L]\;,
\end{align}
where $B(t):=\langle \xi , \un_{[0,t]}\rangle$ is a one-dimensional Brownian motion. Note that, whenever $X_a$ hits $-\infty$ (that is, whenever $f$ vanishes) it is restarted from $+\infty$. The set of eigenvalues is then in one-to-one correspondence with the set of points $-a$ such that $X_a(L) = -\infty$.

The key point of the Riccati transform is that the processes $X_a$'s satisfy the following \textit{monotonicity property}: if $a < a'$ then $X_a(\cdot) \le X_{a'}(\cdot)$ up to the first hitting time of $-\infty$ of $X_a$. As a consequence, for any $k\ge 1$, the r.v.~$-\lambda_k$ is the largest $a\in\R$ such that $X_a$ explodes exactly $k$ times on $[0,L]$: the Riccati transform $\Chi_k = \varphi'_k / \varphi_k$ of the $k$-th eigenfunction $\varphi_k$ is then given by $X_{-\lambda_k}$. (Note that the above discussion relies on deterministic arguments so that the characterization of the $\lambda_k,\varphi_k$'s in terms of the processes $X_a$'s holds true almost surely.)

The study of the eigenvalues/eigenfunctions thus boils down to the study of the collection of coupled processes $X_a$'s. For any given $a\in\R$, the process $X_a$ is an irreversible diffusion that evolves in the potential $V_a$:
\begin{equation}\label{Eq:V}
V_a(x) = \frac{x^3}{3} - ax\;,\quad \mbox{for all } x\in \R\;, \quad a \in \R\;,
\end{equation}
which admits a well whenever $a >0$ (see Figure \ref{PotentialVa}). For any $a\in\R$, this diffusion hits $-\infty$ in finite time almost surely. The aforementioned result of McKean shows that, as $L$ becomes large, the first eigenvalue goes to $-\infty$ at speed $a_L$ (and the same holds for the next $k$ eigenvalues). Therefore, to study the bottom of the spectrum of $\cH_L$ we can focus on diffusions with a large parameter $a >0$.

\smallskip

A typical realization of the diffusion $X_a$ for a fixed large $a>0$ does the following. It comes down from $+\infty$ very quickly and oscillates around $\sqrt a$ (bottom of the well of $V_a$) for a long time. It makes many attempts to get out of the well, and from time to time, it makes an exceptional excursion to $-\sqrt a$, spends a very short time near that point, and either goes back to $\sqrt a$ or explodes to $-\infty$ very quickly.

\begin{figure}[!h]
\centering
\includegraphics[width=7cm]{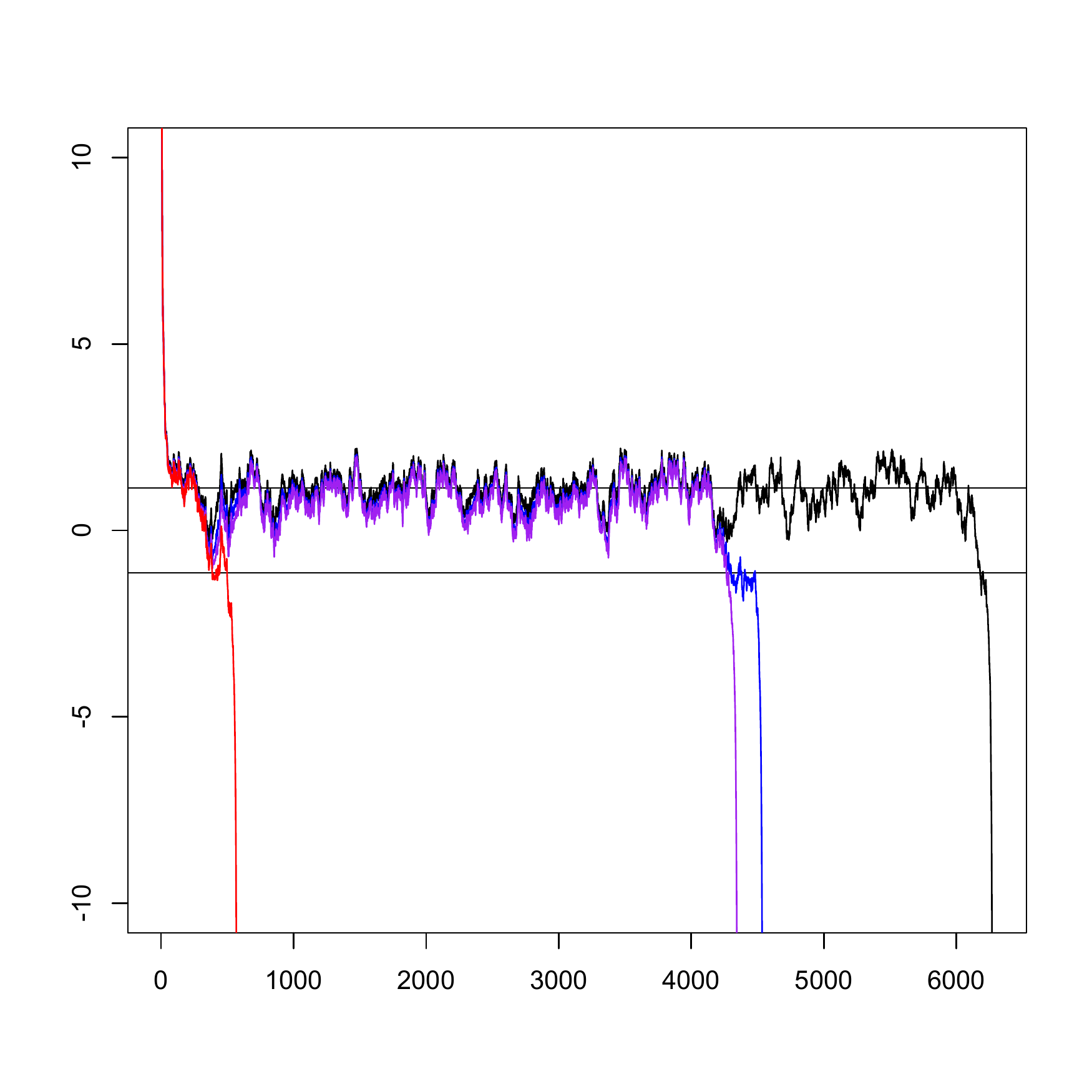}
\caption{(color online) Ordering of the diffusions $X_a$ until their first explosion time for $a=0,55$ (red), $a=0,7$ (purple), $a=0,845$ (blue), $a= 1,3$ (black).}\label{fig:firstexploXa}
\end{figure}

Since the eigenvalues $(\lambda_{k}, k \in \N)$ depend on the realization of the underlying noise $\xi$, the process $\Chi_{k}$ is \textit{not} a diffusion and actually does not look like a typical solution of \eqref{Eq:Diff}. For instance, we will see later that $\Chi_k$ spends a macroscopic time near the unstable point $-\sqrt{|\lambda_k|}$, see Figures \ref{figCompdiff} and \ref{timereversal}. Nevertheless, we will see that it is possible to extract information out of the realizations of \emph{typical} diffusions $X_a$'s.



\subsubsection*{Convergence of the eigenvalues} Let us present the main steps of the proof of our theorems. We start with the convergence of the point process of the rescaled and recentered eigenvalues 
\begin{align*}
\cQ_L := \sum_{k \geq 1} \delta_{4\sqrt{a_L}\,(\lambda_k + a_L)}
\end{align*}
of Theorem \ref{Th:Main}. While tightness is easy to get, we identify the limit by showing that for all fixed $r_1 < \ldots < r_k$ the random variable $\big(\cQ_L((r_1,r_2]),\ldots,\cQ_L((r_{k-1},r_k])\big)$ converges to a vector of independent Poisson r.v. with parameters $e^{r_i}-e^{r_{i-1}}$. 

Observe that $\cQ_L((r_{i-1},r_i]) = \#X_{a_i} - \#X_{a_{i-1}}$ where $a_i = \sqrt a_L - r_i/(4\sqrt a_L)$. We subdivide $(0,L]$ into the $2^n$ disjoint subintervals $(t_j^n,t_{j+1}^n]$ with $t_j^n=j2^{-n}L$, and we introduce the diffusion $X_{a_i}^{j}$ that starts from $+\infty$ at time $t_j^n$ and follows the SDE \eqref{Eq:Diff} with parameter $a_i$. Then, we show that with large probability for all $i$ and $j$:
\begin{itemize}
\item $X_{a_i}$ explodes at most one time on $(t_j^n,t_{j+1}^n]$,
\item $X_{a_i}$ explodes on $(t_j^n,t_{j+1}^n]$ if and only if $X_{a_i}^{j}$ explodes on $(t_j^n,t_{j+1}^n]$.
\end{itemize}
This holds because the potential $V_{a_i}$ possesses a large well and the diffusion goes down from $+\infty$ into the well in a very short time: its starting point quickly becomes irrelevant.

Therefore, it suffices to deal with the diffusions $X_{a_i}^{j}$ restricted to $(t_j^n,t_{j+1}^n]$: these diffusions are independent for different values of $j$, and are monotone in $i$, so that a very simple computation, see Lemma \ref{Lemma:TildePoisson}, allows to get the aforementioned convergence.

An important remark is that we only need the \emph{monotonicity} of the coupled diffusions for this part of the proof: we use no finer information about this coupling. The arguments in the proof could be applied to other situations where the number of explosions of coupled diffusions counts the eigenvalues (e.g. Airy$_{\beta}$ \cite{RamRidVir}, Sine$_{\beta}$ \cite{ValkoViragCarousel} or the Stochastic Bessel Operator \cite{RamirezRider}).

\subsubsection*{The first eigenfunction} Let us now concentrate on the asymptotic behavior of the first eigenfunction. According to the convergence \eqref{CVfirsteigenvalue}, we consider a discretization $M_L$ of a small neighborhood of $a_L$ of mesh $\eps/\sqrt{a_L}$ for a fixed small $\eps >0$ and we will use precise estimates on the typical behavior of $X_a$, simultaneously for all $a \in M_L$. With large probability, $-\lambda_1$ falls within this neighborhood of $a_L$ and there exist two points $a,a'$ in the discretization such that $-\lambda_2 < a \le -\lambda_1 < a'$. By definition, $X_{a}$ explodes one time and $X_{a'}$ does not explode on the time interval $[0,L]$.

\smallskip

From the monotonicity property, we have $X_a \le \Chi_1 < X_{a'}$ up to the first explosion time of $X_a$. With large probability, $X_a$ and $X_{a'}$ are ``typical'': in particular, they remain close to $\sqrt a$ and $\sqrt{a'}$ respectively most of the time. If the mesh of the discretization $M_L$ has been chosen small enough, we deduce that $\Chi_1$ is squeezed in between those two typical diffusions that are close to each other with large probability, up to the first explosion time of $X_a$. However, this does not provide any good control on $\Chi_1$ after this explosion time. In particular, it does not say whether, for instance, $\Chi_1$ remains around $-\sqrt{|\lambda_1|}$, or goes back to $\sqrt{|\lambda_1|}$. To push the analysis further, we rely on a symmetry argument.

\begin{figure}[!h]
\begin{minipage}{7cm}
\includegraphics[width=7cm]{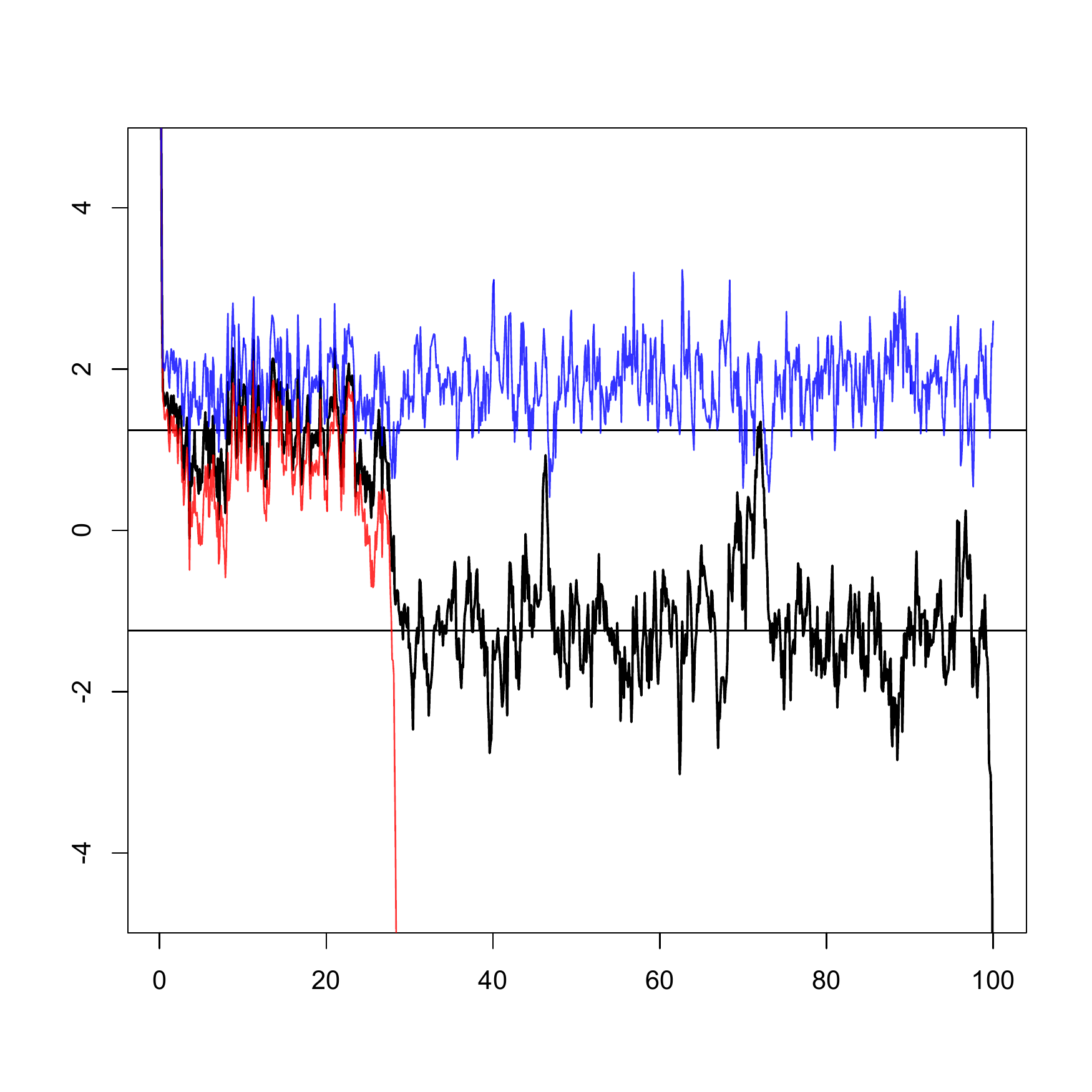}
\caption{(color online) Simulation of $\Chi_1$ (in black) and of two diffusions $X_a$ (in red) and $X_{a'}$ (in blue) with $a < -\lambda_1 < a'$.}\label{figCompdiff}
\end{minipage}\hfill
\begin{minipage}{7cm}
\includegraphics[width=7cm]{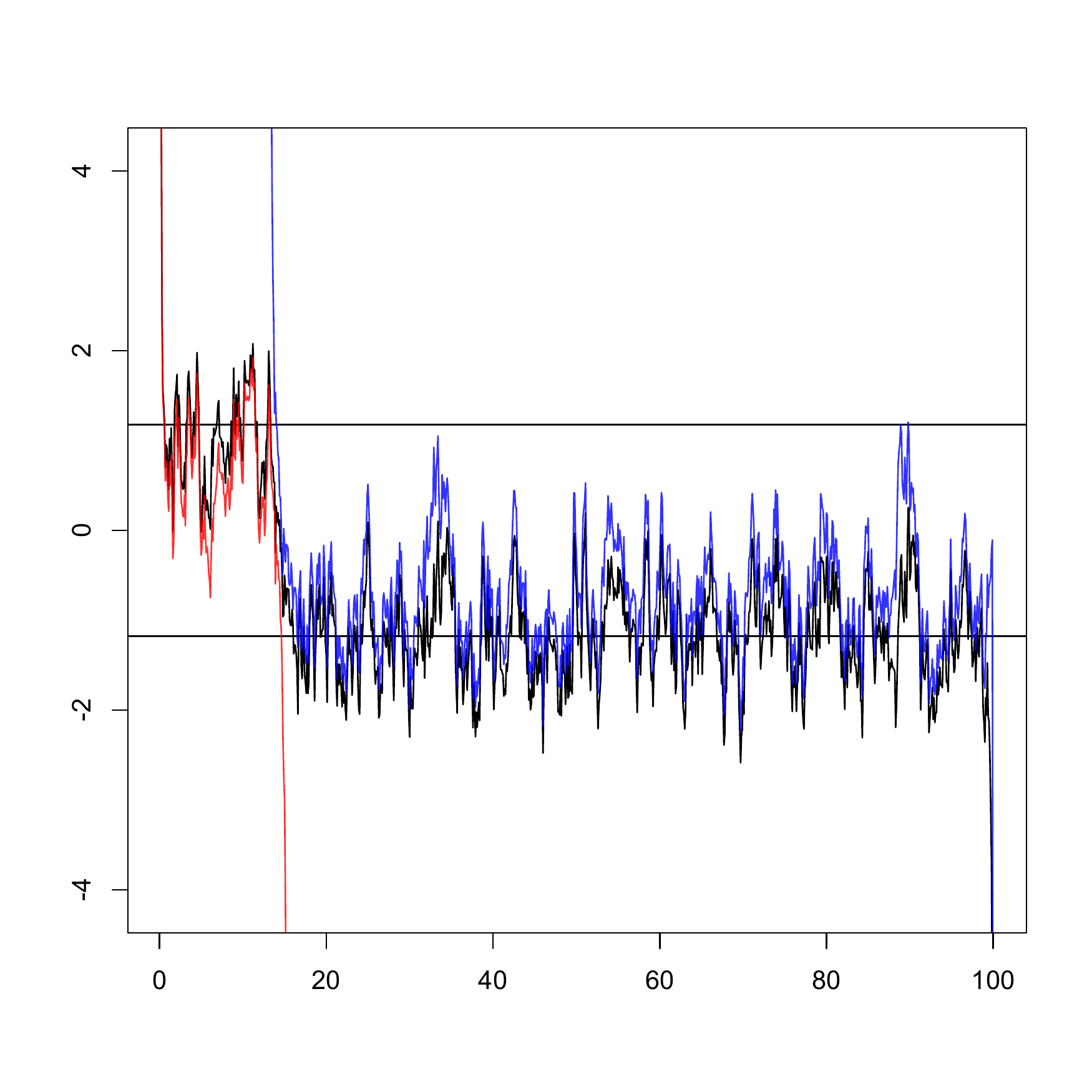}
\caption{(color online) Simulation of $\Chi_1$ (in black) with the two diffusions $X_a$ (in red) and $\hat X_a$ (in blue) until their first explosion, with $a < -\lambda_1$.
}\label{timereversal}
\end{minipage}
\end{figure}


The Riccati transform can be applied not only forward in time from time $0$ but also \emph{backward in time} from time $L$. This yields another set of diffusions $(\hat{X}_a, a\in \R)$, called the \emph{time-reversed diffusions}, which is equal in law to $(-X_a, a\in \R)$. Namely, we have:
$$d\hat{X}_a(t) = V'_a (\hat{X}_a(t)) dt + d\hat{B}(t) \;,$$
starting from $\hat{X}_a(0) = -\infty$ and where $\hat{B}(t) := B(L-t) - B(L)$. A coupling argument shows that the number of explosions of $X_a$ and $\hat{X}_a$ coincide. This allows to say that $\Chi_1$, run backward in time from time $L$, is squeezed in between $\hat{X}_a$ and $\hat{X}_{a'}$ up to the explosion time of $\hat{X}_a$. Then, it remains to show that the explosion times of $X_a$ and $\hat{X}_a$ are very close so that the intervals on which we bound the first eigenfunction overlap (see Figure \ref{timereversal}). 

\smallskip

Since $X_a$ and $X_{a'}$ spend most of their time near $\sqrt{a} \simeq \sqrt{a'} \simeq \sqrt{a_L}$, we deduce by inversing the Riccati transform, that the eigenfunction grows exponentially fast at this rate up to the explosion time of $X_a$. Then, it follows the time-reversed diffusions which spend most of their time near $-\sqrt{a_L}$ so that the first eigenfunction decays exponentially fast from the explosion time until time $L$. This proves the localization of the first eigenfunction near the explosion time of $X_a$.

\medskip

We push this analysis further and obtain a much more precise result on the localization, namely the \emph{shape of the first eigenfunction} of Theorem \ref{Th:Shape}. The explosion of $X_a$ occurs right after an exceptional descent from $+\sqrt{a}$ to $-\sqrt{a}$ whose trajectory is concentrated around the deterministic function $t\mapsto -\sqrt{a} \tanh(\sqrt{a}\,t)$ (appropriately shifted in time). This is a simple consequence of the large deviations principle satisfied by the diffusion, that ensures that the trajectory on this exceptional descent is roughly given by the solution of the ODE
\begin{align*}
dx(t) = V'_a(x(t)) dt\;.
\end{align*}
We will obtain a precise statement about this fact thanks to the Girsanov Theorem. We then show that $\Chi_1$ remains very close to $X_a$ upon this deterministic descent. A careful argument, involving the time-reversed diffusion, shows that the maximum of $\varphi_1$ is located at one of the zeros of $\Chi_1$ achieved during this deterministic descent: since all these zeros lie in a tiny region, we get a precise enough control on the location of the maximum. Applying the inverse of the Riccati transform to $-\sqrt{a} \tanh(\sqrt{a}\,t)$, one gets the deterministic shape $h(t)$ of the statement of the theorem.

\medskip

Up to this point, we have not proved yet that the localization center is asymptotically uniform on $[0,L]$. To that end, we present a coupling argument that relates the localization centers of the first eigenfunctions of $\cH_L$ and of the operator $\tilde{\cH}_L$, the latter being obtained upon replacing the white noise by its image through a Lebesgue preserving bijection. This coupling argument eventually shows that the law of these centers is invariant under (a large enough class of) Lebesgue preserving bijections so that it is necessarily uniform.

\subsubsection*{The next eigenfunctions} For the $k$-th eigenfunction with $k>1$, the situation is slightly more involved though the strategy is the same. We show that with large probability, $-\lambda_k$ falls in the aforementioned neighborhood of $a_L$ and that there exist $a,a'$ in the discretization of that neighborhood such that $-\lambda_{k+1} < a \le -\lambda_k < a' \le -\lambda_{k-1}$. The typical diffusions $X_a$ and $X_{a'}$ explode $k$ and $k-1$ times respectively. We show that they remain close to each other up to the \textit{additional} explosion time of $X_a$, thus providing a good control on $\Chi_k$ up to this time. Then, we rely on the time-reversed diffusions to complete the picture, as in the case of the first eigenfunction. We refer to Figure \ref{fig:1st4eigenfunctions} for an illustration.

\begin{figure}[!h]
\centering
\includegraphics[width = 7cm]{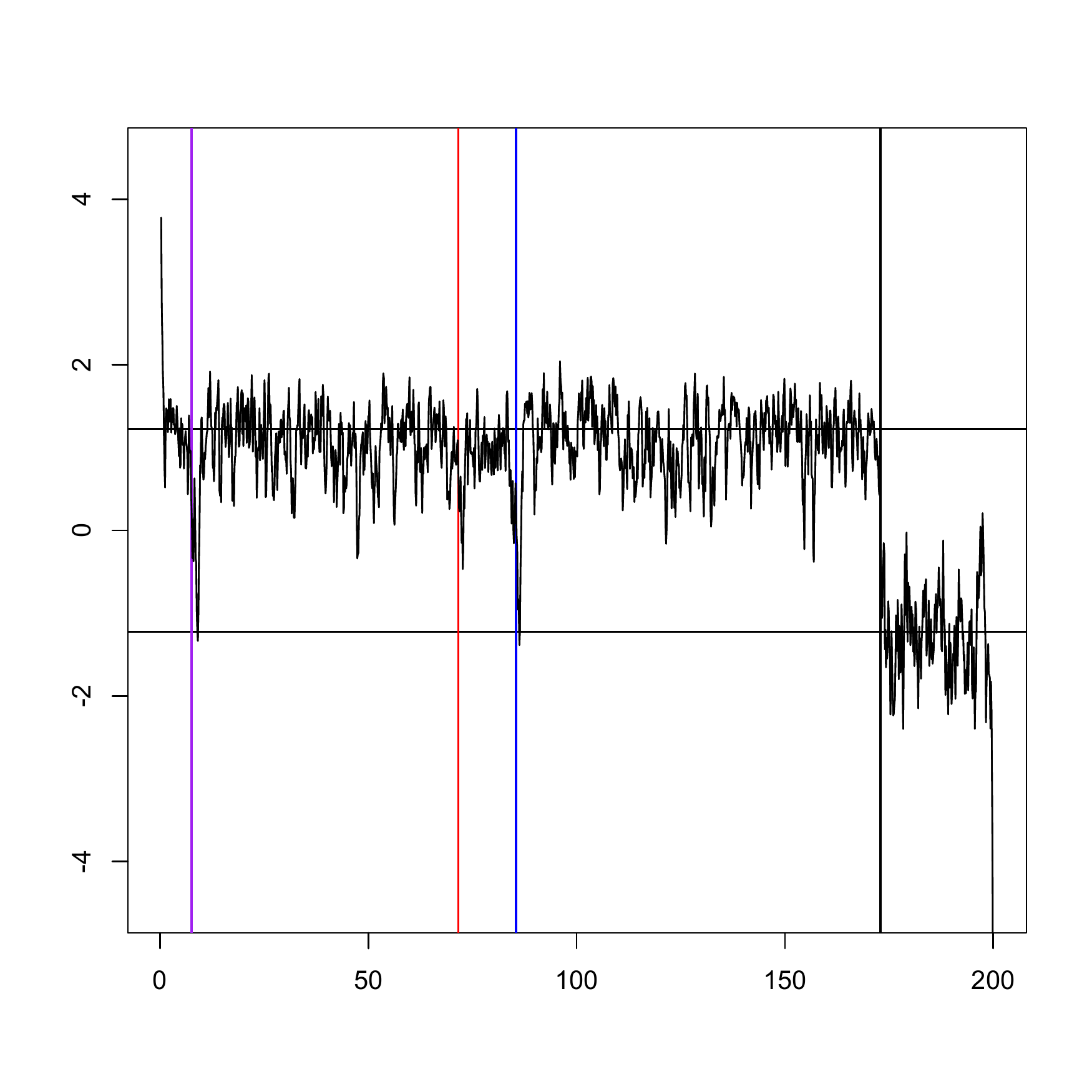} \\
\includegraphics[width = 5cm]{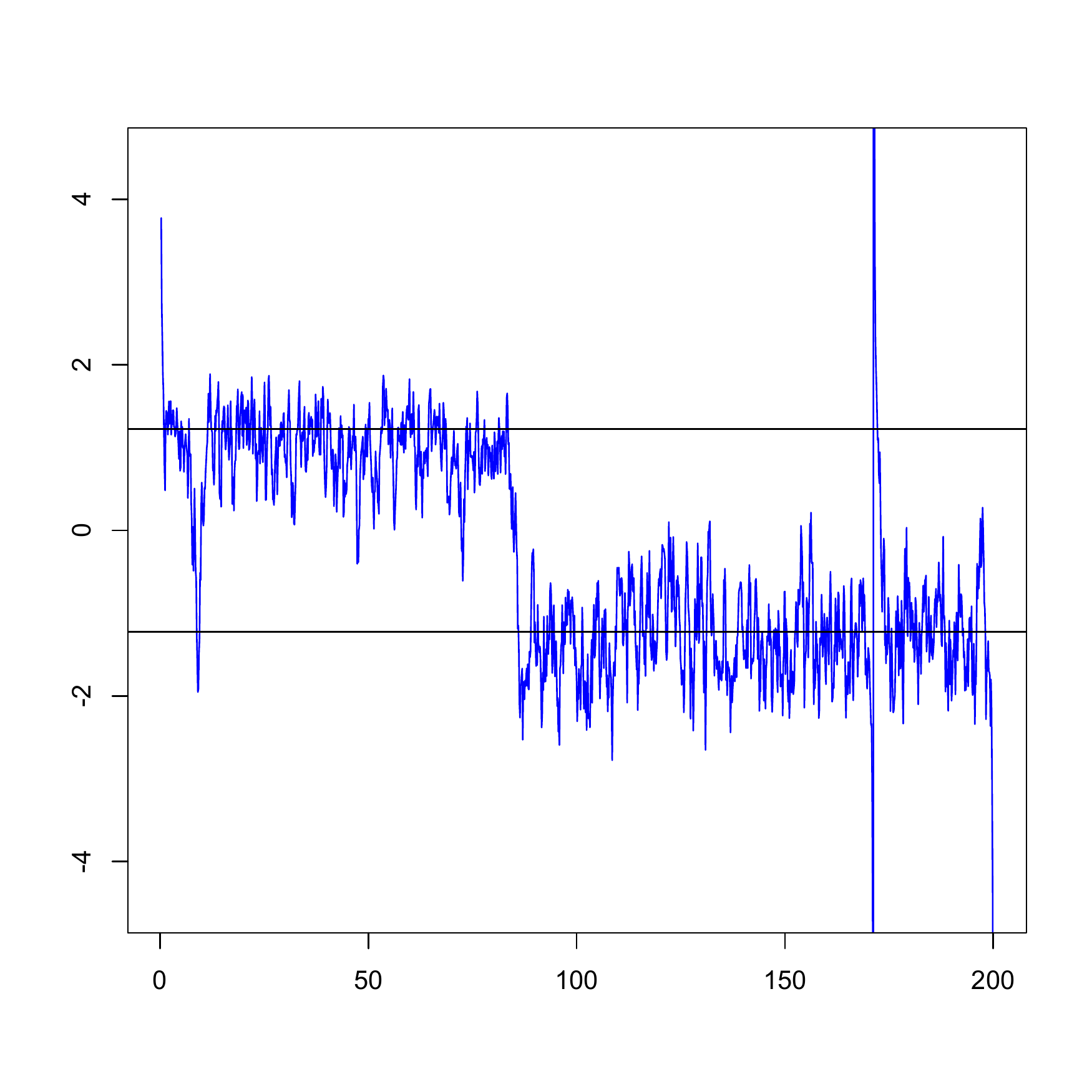}
\includegraphics[width = 5cm]{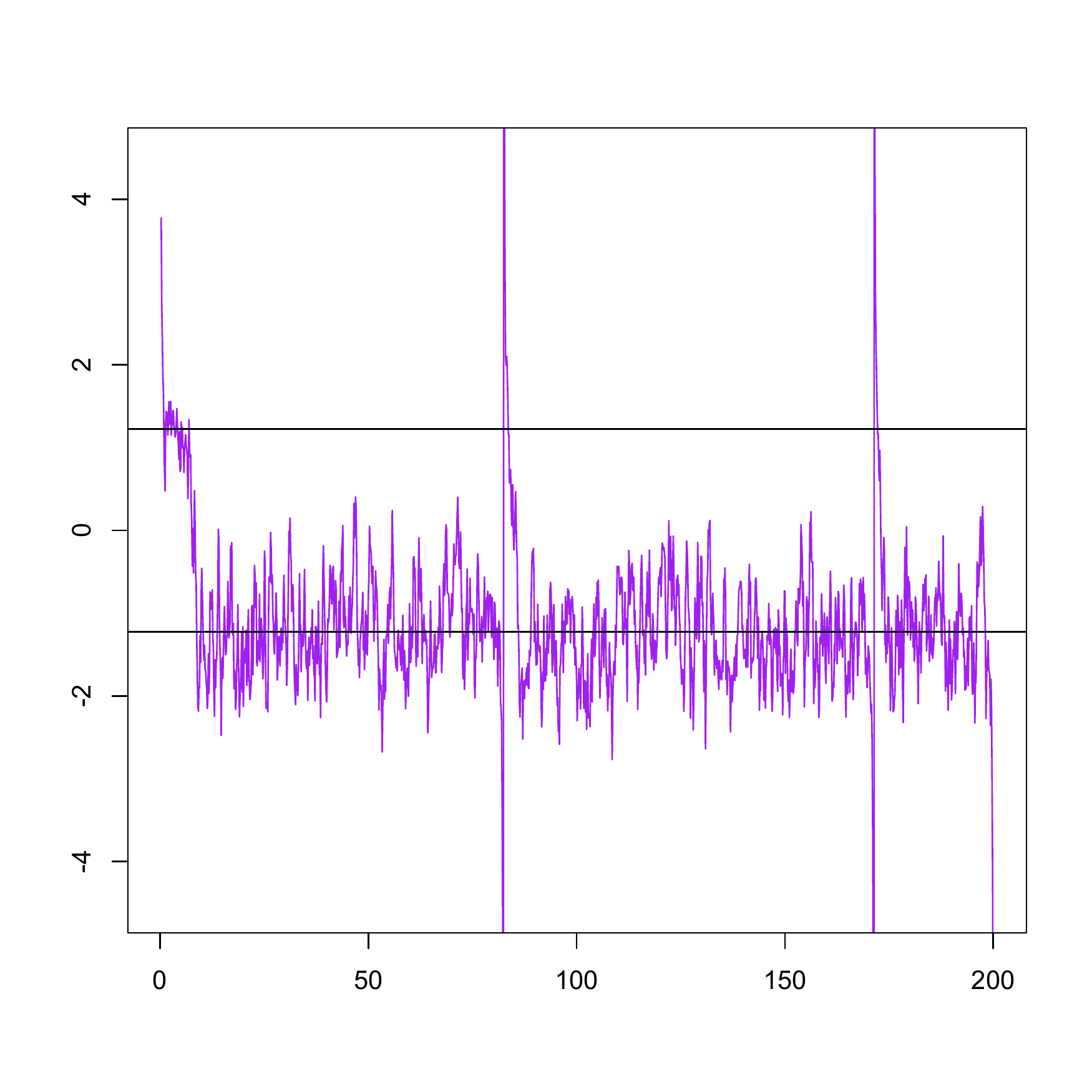}
\includegraphics[width = 5cm]{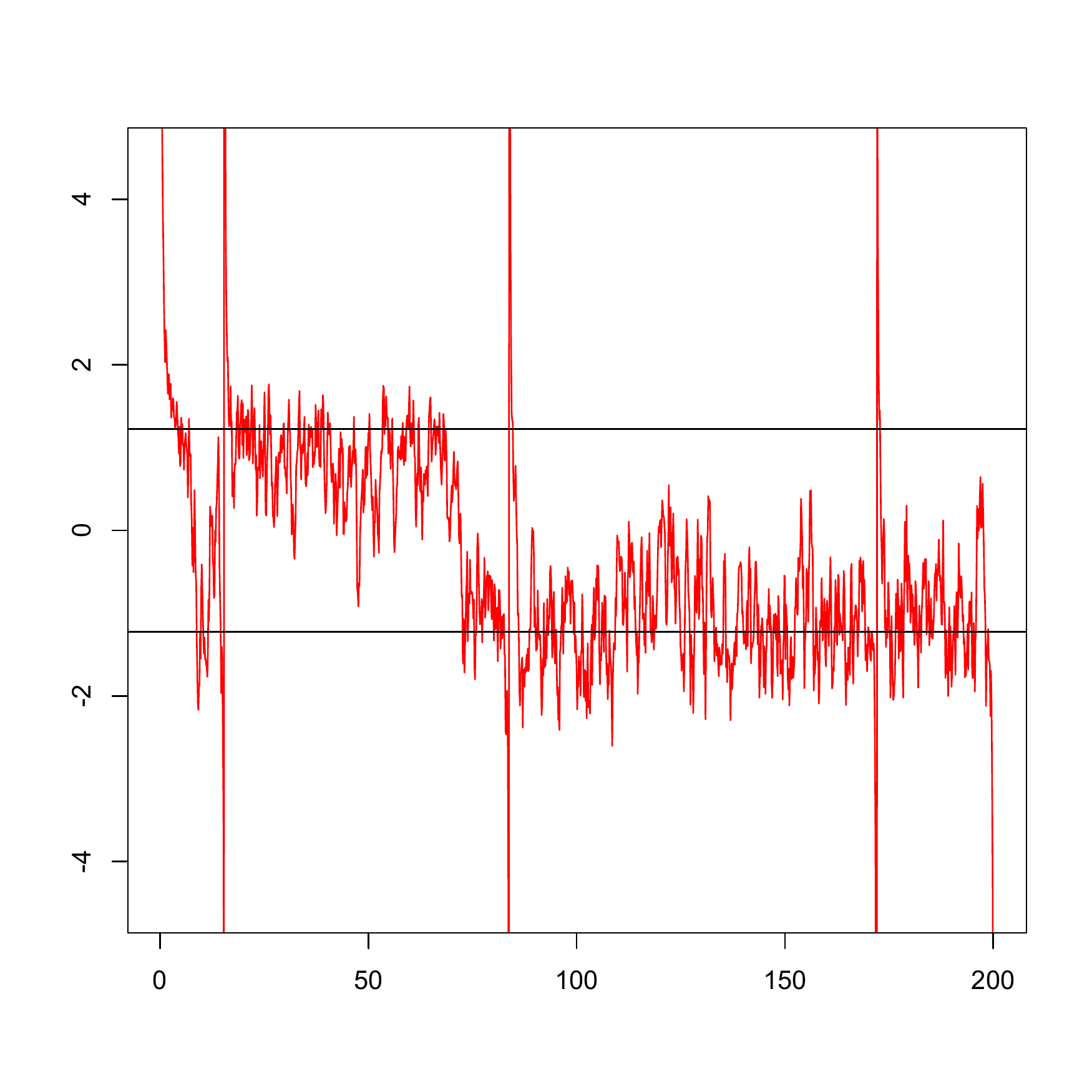}
\caption{Simulation of the Riccati transform of the first $4$ eigenfunctions (successively black, blue, purple, red). On the first picture, the position of the maxima of the first four eigenfunctions are represented with their color. Notice that the number of explosions equals the numbering of the eigenfunction. The explosions occur near the maxima of the previous eigenfunctions.}\label{fig:1st4eigenfunctions}
\end{figure}

\subsection*{Organization of the paper: } In Section \ref{Sec:Poisson}, we collect some first estimates on the diffusion $X_a$ and we prove the convergence of the point process of eigenvalues towards the Poisson point process with intensity $e^x dx$. The proofs of the first estimates are postponed to Section \ref{Sec:Diff}. In Section \ref{Sec:ProofsTh}, we prove Theorems \ref{Th:Main}, \ref{Th:Shape} and \ref{Th:LocalMaxZeros}. The proofs are based on more precise estimates on the diffusion $X_a$ on its exceptional excursions to $-\sqrt{a}$, which are obtained in Section \ref{Sec:Crossing}. Finally, in Section \ref{Sec:Neumann} we prove Theorem \ref{Th:Neumann} about Neumann boundary conditions.

\subsection*{Notations} The first hitting time by a continuous function $f$ of a point $\ell\in\R \cup \{- \infty\}$ is defined as follows
$$ \tau_\ell(f) := \inf\{t\geq 0: f(t) = \ell\}\;,$$
Moreover, for any integrable function $f$ we define its average on $[s,t]$ by setting
$$ \fint_s^t f(x) dx := \frac1{t-s} \int_s^t f(x) dx\;.$$
When $s=t$, we set this quantity to $f(s)$.

\subsection*{Acknowledgements} We thank the anonymous referees for their careful reading of the paper and their helpful remarks. LD thanks Romain Allez and Benedek Valk\'o for useful discussions. The work of LD is supported by the project MALIN  ANR-16-CE93-0003. CL is grateful to Julien Reygner for several useful comments on a preliminary version of this paper. The work of CL is supported by the project SINGULAR ANR-16-CE40-0020-01.

\section{Convergence of the eigenvalues towards the Poisson point process}\label{Sec:Poisson}

The goal of this section is to prove that the point process $\cQ_L := (4 \sqrt{a_L}(\lambda_k + a_L))$ of the rescaled eigenvalues converges in law to a Poisson point process of intensity $e^x dx$. This is the first step in the proof of Theorem \ref{Th:Main}. After some technical preliminaries on the well-definiteness of our diffusions, we collect some first estimates on the diffusion $X_a$, and then, we prove the convergence.

\subsection{Technical preliminaries}

For any given continuous function $b$ starting from $0$, we consider the ODE:
\begin{equation*}
\begin{cases}
x(t) =&- \int_0^t V_a'(x(s)) ds + b(t) \;,\quad t>0\;,\\
x(0) =& +\infty\;.
\end{cases}
\end{equation*}
This ODE admits a unique solution that leaves $+\infty$ at time $0+$, and restarts from $+\infty$ whenever it hits $-\infty$. Let $B$ be a standard Brownian motion on the probability space $(\Omega,\cF,\bP)$. Without loss of generality, we can assume that $B(\omega)$ is continuous for all $\omega\in\Omega$. We apply deterministically the solution map associated to the ODE above for every realization of the Brownian motion. This yields a well-defined process $(X_a(t),t\ge 0, a\in \R)$.
\begin{remark}
Actually, the process can be initialized not only at time $0$ but also at some arbitrary time $t_0\ge 0$. The above construction still applies, and yields a family $(X_a^{t_0}(t), t\geq t_0, a\in \R, t_0\ge 0)$ with $X_a^{t_0}(t_0) = +\infty$.
\end{remark}
The monotonicity of the solution map associated to the ODE yields the following property: if $a<a'$ and $X_a(t) \le X_{a'}(t)$ then $X_a(t+\cdot) \le X_{a'}(t+\cdot)$ up to the next explosion time of $X_a$.

\subsection{First estimates on the diffusion $X_a$}

We now study the process just defined above:
\begin{align*}
dX_a(t) = (a - X_a(t)^2) dt + dB(t), \quad X_a(0) = +\infty.
\end{align*}

This diffusion evolves in the potential $V_a(x) = - a x + x^3/3$. It blows-up to $-\infty$ in finite time a.s. and immediately restarts from $+\infty$. From now on, we let $0=\zeta_a(0) < \zeta_a(1) < \zeta_a(2) < \ldots$ be the successive explosion times of the diffusion $X_a$. When $a >0$, the potential $V_a$ has a well centered at $\sqrt{a}$. The diffusion has to cross the barrier from $\sqrt{a}$ to $-\sqrt{a}$, which is of size $\Delta V_a = (4/3) a^{3/2}$, in order to explode to $-\infty$  (see Figure \ref{PotentialVa}).

\begin{figure}[!h]
\includegraphics[width = 7cm]{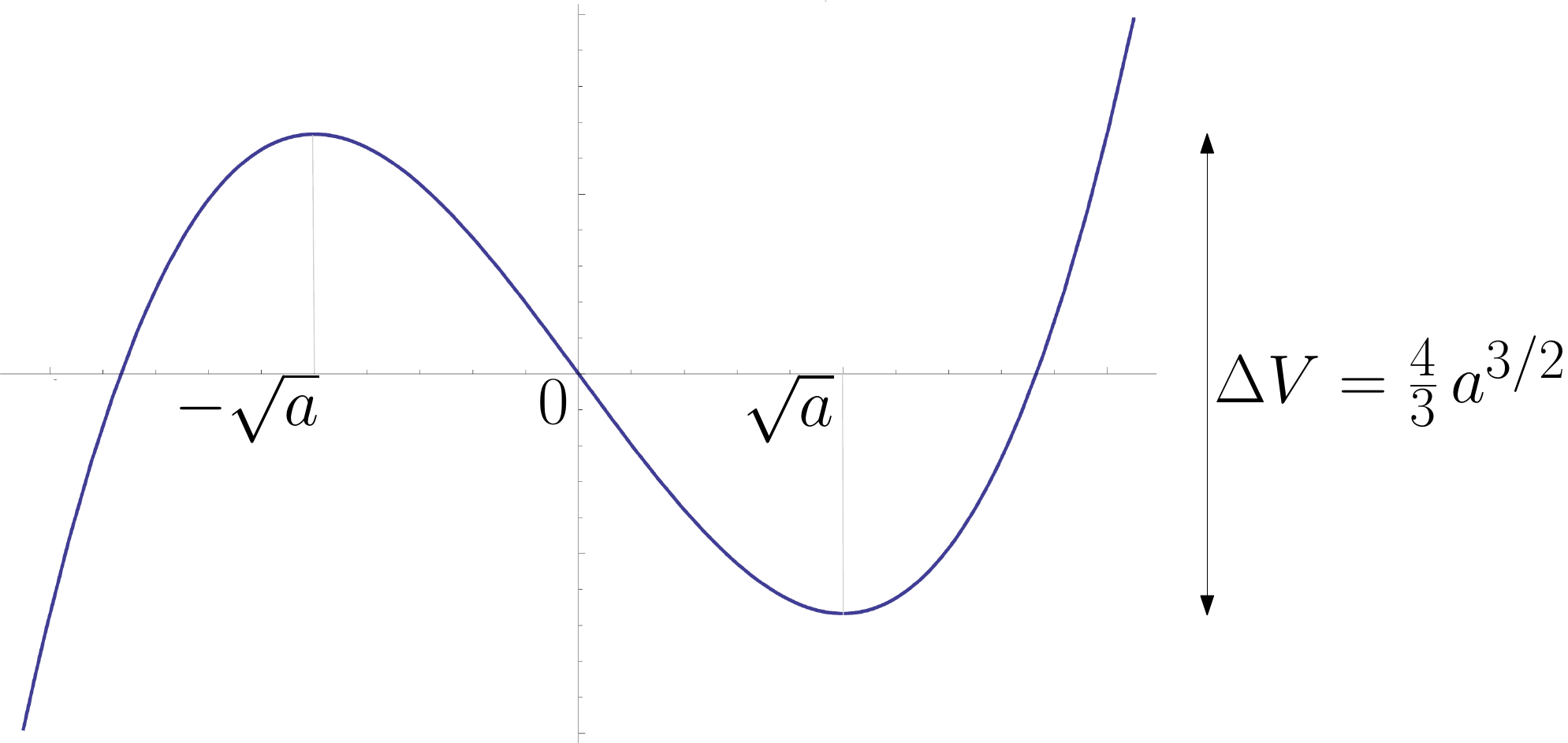}
\caption{The potential $V_a$ when $a >0$.}\label{PotentialVa}
\end{figure}

The expectation of the first explosion time $\zeta_a(1)$ admits the following expression, see~\cite{Fukushima}:
\begin{align}
m(a):= \E[\zeta_a(1)] &= \sqrt{2\pi} \int_{0}^{+\infty} \frac{dv}{\sqrt{v}} \, \exp \left( 2 a v - \frac{1}{6}\, v^3\right) \label{Eq:ma} \\
&= \frac{\pi}{\sqrt{a}}\exp(\frac{8}{3}a^{3/2})(1+o(1)), \quad \mbox{when }a \to +\infty\;. \label{asymp_ma}
\end{align} 
We see that the explosion time is of order $\exp(2\Delta V_a)$, which is in line with Kramers' law.
\begin{remark}
Recall that the number of explosions of $X_a$ in $[0,L]$ is equal to the number of eigenvalues below $-a$. By the law of large numbers, it readily implies that the integrated density of states $N(\lambda)$ mentioned in the introduction is precisely given by 
\begin{align}\label{eqDensityOfStates}
N(\lambda) = \frac{1}{m(-\lambda)} \quad \mbox{for all }\lambda \in \R \,.
\end{align}
\end{remark}
The following result is due to McKean, we refer to \cite{McKean, AllezDumazTW} for a proof.
\begin{proposition}\label{propo:blowupexpo}
As $a\to\infty$, the r.v.~$\zeta_a(1)/m(a)$ converges in distribution to an exponential law of parameter $1$.
\end{proposition}
\noindent As a direct corollary, we get the following result.
\begin{proposition}\label{Prop:PoissonExplo}
The point process of the explosion times of $X_a$ rescaled by $m(a)$ converges in law, for the vague topology, to a Poisson point process on $\R_+$ with intensity $1$.
\end{proposition}

It is then natural to introduce $L\mapsto a_L$ as the reciprocal of $a\mapsto m(a)$. A simple computation yields the following asymptotics
\begin{align}
a_L =  \Big(\frac{3}{8} \ln L\Big)^{2/3} + 3^{-1/3} \cdot 2^{-1} (\ln L)^{-1/3} \Big(\ln(\frac{3^{1/3}}{2 \pi} \ln^{1/3} L) + o(1)\Big)\;.\label{asymp_aL}
\end{align}
The diffusion $X_a$, with $a$ close to $a_L$, typically explodes a finite number of times on the time interval $[0,L]$. More precisely, for all $r \in \R$, using \eqref{asymp_ma}, we have as $L \to \infty$:
\begin{align*}
m\big(a_L + \frac{r}{\sqrt{a_L}}\big) = \frac{\pi}{\sqrt{a_L}}\exp(\frac{8}{3}a_L^{3/2}+4r)(1+o(1))\;,
\end{align*}
so that the order of magnitude of the number of explosions of $X_a$ on $[0,L]$ for $a = a_L + r/(4\sqrt{a_L})$ is given by $L/m(a_L + r/(4\sqrt{a_L})) = \exp(-r) (1+ o(1))$.

\bigskip

Until the end of this subsection, we drop the subscript $a$ from $X_a$ and from the explosion times $\zeta_a(k)$, $k\ge 0$ to alleviate the notations. Let us analyse the diffusion $X$ until its first explosion time $\zeta := \zeta(1)$. We set
\begin{equation}\label{Eq:ta}
t_a = \frac{\ln a}{\sqrt{a}}\;, \quad \mbox{and}\quad h_a = \frac{\ln a}{a^{1/4}}\;.
\end{equation}

A typical realization of the process $X$ is close to a deterministic path when it comes down from $+\infty$ (entrance) and when it explodes to $-\infty$ (explosion):
\begin{description}
\item[1- Entrance] Let $\ell := \sqrt{a} + h_a$. We have
\begin{align*}
\tau_{\ell}(X) = \frac38 t_a - \frac{\ln\ln a}{2\sqrt a} + O(\frac1{\sqrt a}),
\end{align*}
together with the bound
$$ \sup_{t\in (0,\frac38 t_a]} \big|X(t) - \sqrt{a} \coth (\sqrt a t)\big| \le 1\;.$$
\item[2- Explosion] Let $\ell := -\sqrt{a} - h_a$. We have
\begin{align*}
\tau_{-\infty}(X) - \tau_{\ell}(X) = \frac38 t_a - \frac{\ln\ln a}{2\sqrt a} + O(\frac1{\sqrt a}).
\end{align*}
together with the bound
$$ \sup_{t\in [\tau_{\ell}(X),\zeta)} \big|X(t)+\sqrt{a} \coth (\sqrt a (\zeta-t))\big| \le 1\;.$$
\end{description}

We denote by $\mathcal D_a^N$ the event on which the $N$ first trajectories $(X(t),\;t \in [\zeta(k),\zeta(k+1)))$ for $k \in \{0,\cdots,N-1\}$ satisfy the estimates above.

\begin{proposition}[Entrance and explosion]\label{Prop:Da} 
Fix $N\in \bbN$. For all the parameters $a$ large enough, we have 
\begin{align*}
\bbP[\mathcal{D}_a^N] \ge 1-\frac{1}{\ln a}\;.
\end{align*}
\end{proposition}
\noindent The proof of Proposition \ref{Prop:Da} relies on a comparison with an ordinary differential equation (in the same way as in \cite{DumazVirag}): it is postponed to Section \ref{Sec:Diff}.

\medskip

For $t_0 > 0$, let $X^{t_0}$ be the diffusion following the same SDE as $X$ but starting from $+\infty$ at time $t_0$. The following \textit{synchronization} estimate ensures that the diffusion $X^{t_0}$ gets very close to the original diffusion $X$ after a short time of order $O(t_a)$, and that they stay together until their first explosion. 

We set $\tau^k_{-2\sqrt{a}}(X):=\inf\{t\ge \zeta_a(k-1): X(t) = - 2\sqrt a\}$.

\begin{proposition}[Synchronization]\label{Propo:DiffClose}
Fix $N\ge 1$ and $(t_0(a))_{a >0}$ a family of deterministic starting times which may depend on $a$. There exists $C>0$ such that the following holds true with a probability at least $1-O(1/\ln a)$ as $a \to \infty$. If there exists $k\in \{1,\ldots,N\}$ such that $t_0(a)\in [\zeta_a(k-1), \zeta_a(k))$ then we have the bounds
\begin{align*}
&\big| X(t) - X^{t_0(a)}(t)\big| \le h_a\;,\quad \forall t \in [t_0(a)+\frac38 t_a,\tau^k_{-2\sqrt{a}}(X)]\;,\\
&\big|\tau_{-\infty}(X^{t_0(a)}) - \zeta_a(k)\big| < \frac{C}{\sqrt{a}}\;.
\end{align*}
\end{proposition}

The proof of Proposition \ref{Propo:DiffClose} is based on a coupling with a stationary diffusion. Since the arguments are rather standard, we postpone the proof to Section \ref{Sec:Diff}.

\subsection{Proof of the convergence of the point process of the eigenvalues}\label{subsec:proofeigenvalues}
We consider
\begin{align*}
\cQ_L := \sum_{k \geq 1} \delta_{4\sqrt{a_L}\,(\lambda_k + a_L)}\;,
\end{align*}
which is a random element of the set of measures on $\R$ with finite mass on $(-\infty, r)$ for any $r\in\R$, and endowed with the topology that makes continuous the maps $\mu\mapsto \int f \mu$ for all bounded and continuous functions $f$ with support bounded on the right. By~\cite[Lemma 14.15]{Kallenberg}, the family $(\cQ_L)_{L >1}$ is tight if for every $r > 0$ and every $\eps > 0$, there exists $c>0$ such that
\begin{equation}\label{Eq:Tightness}
\sup_{L>1} \bbP\big[\cQ_L((-\infty,r]) > c \big] < \eps\;.
\end{equation}
Since $\cQ_L((-\infty,r])$ is equal to the number of explosions of $X_{a_L - r/(4\sqrt{a_L})}$ on the time-interval $[0,L]$, we deduce from Proposition \ref{Prop:PoissonExplo} that there exists $c'>0$ such that
$$\varlimsup_{L\in \bbN, L\rightarrow\infty} \bbP\big[\cQ_L((-\infty,r]) > c' \big] < \eps\;,$$
and therefore, that there exist $c>0$ such that
$$\sup_{L\in \bbN} \bbP\big[\cQ_L((-\infty,r]) > c \big] < \eps\;.$$
Since $a_L = a_{\lfloor L\rfloor} + o(1/\sqrt{a_{\lfloor L\rfloor}})$, we can bound the number of explosions of $X_{a_L - r/(4\sqrt{a_L})}$ by the number of explosions of $X_{a_{\lfloor L\rfloor} - r'/(4\sqrt{a_{\lfloor L\rfloor}})}$ for some appropriately chosen $r'$, thus ensuring that the latter bound extends to all $L\in (1,\infty)$. We deduce that \eqref{Eq:Tightness} holds true, thus ensuring the tightness of $(\cQ_L)_{L>1}$.

It remains to prove that any limit $\cQ_\infty$ is distributed as a Poisson point process with intensity $e^x dx$. By classical arguments, it suffices to show that for every $k\geq 1$, and every $-\infty = r_0 < r_1 < r_2 < \ldots < r_k$, if we set
$$ Q_L(i) := \cQ_L((r_{i-1},r_i]) \;,\quad i=1,\ldots,k\;,$$
then $(Q_L(1),\ldots,Q_L(k))$ converges in law to a vector of independent Poisson r.v.~with parameters $p_i$ where
$$ p_i:=e^{r_{i}}-e^{r_{i-1}}\;,\quad i=1,\ldots,k\;.$$ 
Let us fix from now on $r_1 < r_2 < \ldots < r_k$, and notice that if we denote by $\# X_{a_i}$ the number of explosions of $X_{a_i}$ on $[0,L]$, where $a_i = a_L - r_i/(4\sqrt{a_L})$, then we have almost surely 
$$ Q_L(i) = \#X_{a_i} - \# X_{a_{i-1}}\;,\quad i=1,\ldots,k\;.$$

To identify the limiting law of $(Q_L(1),\ldots,Q_L(k))$, we follow an indirect approach. First, we introduce some simpler r.v.~$\tilde{Q}_L(i)$ which are shown to converge to the right limits, and then we prove that they are actually close in probability to the original ones.

We discretize the time-interval $[0,L]$ in such a way that the studied diffusions will typically explode at most once on each sub-interval with large probability. Let $n\geq 1$ be given. We set $t_j^n := j 2^{-n}L$ for all $j\in \{0,\ldots,2^n\}$ and we let $X^j_a:= X^{t_j^n}_a$ be the diffusion starting from $+\infty$ at time $t_j^n$. The main idea of the proof is to approximate the number of explosions of $X_{a_i}$ on $(t_j^n,t_{j+1}^n]$ by the number of explosions of $X_{a_i}^{j}$ on the same interval. Such an approximation is justified by Proposition \ref{Propo:DiffClose}. The advantage of considering the diffusions $X_{a_i}^{j}$ restricted to $(t_j^n,t_{j+1}^n]$ is twofold: first, for every $j$, they are ordered in $i$ by the monotonicity of the diffusions; second, for different values $j$ they are independent.\\

For every $j\in\{0,\ldots,2^n-1\}$, and for every $i\in\{1,\ldots,k\}$, we set $U_j(i) = 1$ if $X_{a_i}^{j}$ explodes on $(t_j^n,t_{j+1}^n]$, and $U_j(i) = 0$ otherwise. Recall that $a_1 > a_2 > \ldots > a_k$. Note that on each time-interval $(t_j^n,t_{j+1}^n]$, we have the ordering $X_{a_k}^{j} \leq \cdots \leq X_{a_1}^{j}$ until the first explosion of $X_{a_k}^{j}$. 

Then, we set
$$\tilde{Q}^{(n)}_L(i) := \sum_{j=0}^{2^n-1} \Big(U_j(i)-U_j(i-1)\Big)\;,\quad i=1,\ldots,k\;.$$
\begin{lemma}\label{Lemma:TildePoisson}
The vector $(\tilde{Q}^{(n)}_L(i))_{i=1,\ldots,k}$ converges in distribution, as $L\rightarrow\infty$ and then $n\rightarrow\infty$, to a vector of independent Poisson r.v.~with parameters $p_i$.
\end{lemma}
\begin{proof}
First of all, the collection (indexed by $j=0,\ldots,2^n-1$) of the $k$-dimensional vectors $(U_j(i),i=1,\ldots,k)$ is i.i.d., and for every $j$ we have almost surely $U_j(1)\le \ldots \le U_j(k)$ since $a_1 > \ldots > a_k$. Henceforth we have for $i \in \{1,\cdots,k-1\}$,
\begin{align*}
\bbP[U_j(1) = \ldots = U_j(i) = 0, U_j(i+1) = \ldots = U_j(k) = 1] &= \bbP[U_j(i) = 0,\,U_j(i+1) = 1]\\
&= \bbP[U_j(i+1)=1]-\bbP[U_j(i)=1] \;,
\end{align*}
as well as
$$ \bbP[U_j(1)= \ldots = U_j(k) = 0] = \bbP[U_j(k)=0]\;,$$
and
$$ \bbP[U_j(1)= \ldots = U_j(k) = 1] = \bbP[U_j(1)=1]\;.$$
Consequently, the law of the vector $U_j$ is completely characterized by its one-dimensional marginals. Since $m(a_L)/m(a_i) \rightarrow e^{r_i}$ as $L\rightarrow\infty$, Proposition \ref{Prop:PoissonExplo} gives as $L\rightarrow \infty$
$$ \bbP[U_j(i) = 1] \rightarrow 1-\exp(- 2^{-n} e^{r_i})\;,$$
for any $j$. Moreover, the random variables $U_j$, $j\in\{0,\ldots,2^n-1\}$ are independent. We have all the elements at hand to compute the limiting distributions of the $\tilde{Q}^{(n)}_L(i)$'s. Fix $\ell_1,\ldots,\ell_k$ such that $\ell=\sum_1^k \ell_i \le 2^n$. Then, we find:
\begin{align*}
&\lim_{L\rightarrow \infty} \bbP\Big[\bigcap_{i=1}^k \{\tilde{Q}^{(n)}_L(i) = \ell_i\} \Big]\\
&=\lim_{L\rightarrow \infty} {2^n \choose \ell_1, \ldots, \ell_k, 2^n-\ell} \Big(\prod_{i=1}^k \bbP[U_1(1)=\ldots=U_1(i-1)=0,U_1(i)=\ldots=U_1(k)=1]^{\ell_i}\Big)\\
&\qquad\qquad\qquad\qquad \times\bbP[U_1(1)=\ldots=U_1(k)=0]^{2^n-\ell}\\
&= {2^n \choose \ell_1, \ldots, \ell_k, 2^n-\ell} \Big(\prod_{i=1}^k \big(1-\exp(-2^{-n} e^{r_i})-1+\exp(- 2^{-n} e^{r_{i-1}})\big)^{\ell_i}\Big) \exp(- 2^{-n} e^{r_k} (2^n - \ell))\\
&= \prod_{i=1}^k\frac{\big(e^{r_i}-e^{r_{i-1}}\big)^{\ell_i}}{\ell_i!} e^{-(e^{r_i}-e^{r_{i-1}})} (1+O(2^{-n}))\;,
\end{align*}
so that the limiting distribution as $n\rightarrow\infty$ is the product of Poisson distributions with parameter $p_i$.
\end{proof}

We now introduce an event on which $\tilde{Q}^{(n)}_L$ and $Q_L$ coincide. Let $\mathcal{F}_{L,n}$ be the event on which for every $i\in\{1,\ldots,k\}$ and $j \in \{0,\cdots,2^n-1\}$:
\begin{enumerate}
\item[(i)] $X_{a_i}$ never explodes twice in any interval $(t_j^n,t_{j+1}^n]$,
\item[(ii)] $X_{a_i}$ explodes on $(t_j^n,t_{j+1}^n]$ if and only if $X_{a_i}^{j}$ explodes on $(t_j^n,t_{j+1}^n]$.
\end{enumerate}
On the event $\mathcal{F}_{L,n}$, the auxiliary random variables $\tilde Q_L^{(n)}(i)$ indeed coincide with $Q_L(i)$.\\

If we prove that the probability of $\mathcal{F}_{L,n}$ tends to $1$ when $L\rightarrow\infty$ and then $n\rightarrow\infty$, then by Lemma \ref{Lemma:TildePoisson} we get:
\begin{align*}
\bbP[Q_L(1) = \ell_1,\ldots,Q_L(k) = \ell_k] = \prod_{i=1}^k \frac{p_i^{\ell_i}}{\ell_i !} e^{-p_i} + \eps_{L,n}\;,
\end{align*}
where $\eps_{L,n}$ goes to $0$ as $L$ and then $n$ go to infinity. This would ensure that the vector $Q_L$ converges in distribution to a vector of independent Poisson random variables with parameters $p_i$, and would conclude the proof of this section.

Therefore, it remains to show that $\bbP[\mathcal{F}_{L,n}]$ goes to $1$ as $L\rightarrow\infty$ and then $n\rightarrow\infty$. By Proposition \ref{Prop:PoissonExplo} and Proposition \ref{Propo:DiffClose}, the probability that for all $i\in \{1,\ldots,k\}$:\begin{enumerate}
\item[(a)] $X_{a_i}$ never explodes twice in any interval $(t_j^n,t_{j+1}^n]$,
\item[(b)] $X_{a_i}$ never explodes in any interval $(t_{j+1}^n-t_{a_L},t_{j+1}^n]$,
\item[(c)] the next explosion times of $X_{a_i}$ and $X_{a_i}^{j}$ after time $t_j^n$ lie within a distance of order $1/\sqrt{a_L}$ from each other,
\end{enumerate}
goes to $1$ as $L\rightarrow\infty$ and then $n\rightarrow\infty$. Property (a) ensures (i). The monotonicity of the diffusions ensures the following assertion: if $X_{a_i}^{j}$ explodes on $(t_j^n,t_{j+1}^n]$, then $X_{a_i}$ explodes on $(t_j^n,t_{j+1}^n]$. The converse assertion is implied by (b) and (c), so that (ii) follows.

\section{Localization of the eigenfunctions}\label{Sec:ProofsTh}

In this section, we first introduce the time-reversed diffusions associated to the eigenvalue problem. Then, we collect some fine estimates on the diffusion during its exceptional excursion that leads it from $\sqrt a$ to $-\sqrt a$. Finally, we provide the proof of Theorems \ref{Th:Main}, \ref{Th:Shape} and \ref{Th:LocalMaxZeros}.

\subsection{Time reversal}\label{Sec:Reversal}

We define the time-reversed operator
$$ \hat{\cH}_L = - \partial^2_x + \hat{\xi}\;,\quad x\in (0,L)\;,$$
with homogeneous Dirichlet boundary conditions, where $\hat{\xi}(\cdot):= \xi(L-\cdot)$ in the distributional sense. Let $(\hat{\varphi}_{k},\hat{\lambda}_{k})_{k\geq 1}$ be the corresponding eigenfunctions/eigenvalues and observe that almost surely for all $k\ge 1$
$$ \hat{\varphi}_{k}(\cdot) = \varphi_{k}(L-\cdot)\;,\quad \hat{\lambda}_{k} = \lambda_{k}\;.$$
We can therefore study the Riccati transforms of the reversed operator and introduce the associated family of diffusions. For convenience, we take the opposite of the Riccati transform:
$$ \hat{\Chi}_{k}(t) := - \frac{\hat{\varphi}_{k}'(t)}{\hat{\varphi}_{k}(t)}\;,\quad t\in (0,L)\;,$$
with boundary conditions $\hat{\Chi}_{k}(0)=-\infty$ and $\hat{\Chi}_{k}(L)=+\infty$. We then have the identity:
$$ \hat{\Chi}_{k}(L-t) = \Chi_{k}(t)\;.$$
It is natural to introduce the collection of diffusions $(\hat{X}_a(t),t\ge 0)$, $a\in\R$, as solutions of:
\begin{equation}\label{Eq:DiffReversed}
d\hat{X}_a(t) = V'_a (\hat{X}_a(t)) dt + d\hat{B}(t) \;,
\end{equation}
starting from $\hat{X}_a(0) = -\infty$ and where $\hat{B}(t) := B(L-t) - B(L)$. We then let $\hat{\zeta}_a(k)$ be the $k$-th explosion time of $\hat{X}_a$. We also set $\hat{\zeta}_a(0) = \zeta_a(0) = 0$.

\begin{remark}
Time $0$ for the diffusion $\hat{X}_a$ corresponds to time $L$ for the diffusion $X_a$.
\end{remark}

\begin{lemma}\label{Prop:Reversal}
Almost surely, for every $k\ge 1$ if $X_a$ explodes to $-\infty$ exactly $k$ times on $[0,L]$ then
\begin{enumerate}
\item $\hat{X}_a$ explodes exactly $k$ times to $+\infty$ on $[0,L]$,
\item for every $i\in \{1,\ldots,k\}$, we have $\zeta_a(i-1) \le L-\hat{\zeta}_a(k-i+1) \le \zeta_a(i)$.
\end{enumerate}
\end{lemma}

\begin{proof}
(1) is immediate as the numbers of explosions of $X_a$ and $\hat X_a$ are related to the number of eigenvalues below $-a$ of $\cH_L$ and $\hat{\cH}_L$, and these two operators share the same eigenvalues.\\
To prove the second property, note that almost surely, for all $t_0 \in \Q\cap (0,L)$ the operators
$$ \cH_{0,t_0} := -\partial_x^2 + \xi\;,\quad x\in (0,t_0)\;,$$
and
$$ \hat{\cH}_{L-t_0,L} := -\partial_x^2 + \hat{\xi}\;,\quad x\in (L-t_0,L)\;,$$
share the same eigenvalues. Assume that for some $i$,  $L-\hat{\zeta}_a(k-i+1) > \zeta_a(i)$ and pick a rational number $t_0$ in between these two values. Since the diffusion $X_a$ explodes at least $i$ times on the interval $[0,t_0]$, the operator $\cH_{0,t_0}$ has at least $i$ eigenvalues below $-a$. On the other hand, the diffusion $\hat{X}_a$ explodes at most $i-1$ times on $[L-t_0,L]$ so that the diffusion $\hat{X}_a^{L-t_0}$ that starts at $-\infty$ at time $L-t_0$ cannot explode more than $i-1$ times either. Consequently, $\hat{\cH}_{L-t_0,L}$ has at most $i-1$ eigenvalues below $-a$. This raises a contradiction. The inequality $\zeta_a(i-1) \le L-\hat{\zeta}_a(k-i+1)$ follows by symmetry, thus concluding the proof.
\end{proof}

\subsection{Fine estimates on the exceptional excursions}\label{subsec:fineestimates}

In this subsection, we collect some precise estimates on the behavior of the diffusion $X_a$ during its first exceptional excursion to $-\sqrt{a}$.  

\smallskip

We let $\theta_a := \tau_{-\sqrt a}(X_a)$ be the first hitting time of $-\sqrt a$, $\upsilon_a$ be the last hitting time of $0$ before $\theta_a$ and $\iota_a$ the last hitting time of $\sqrt{a}$ before $\theta_a$. We take similar definitions for the time-reversed diffusion $\hat X_a$: $\hat{\theta}_a$ is the first hitting time of $+\sqrt a$, $\hat\upsilon_a$ is the last hitting time of $0$ before $\hat\theta_a$ and $\hat\iota_a$ is the last hitting time of $-\sqrt a$ before $\hat\theta_a$. Note that $\theta_a < \infty$ a.s.~since $X_a$ explodes in finite time a.s.

\smallskip

We call \emph{excursion} of $X_a$ a portion of the trajectory that starts at $\sqrt a$, reaches $-\sqrt a$ while staying (strictly) below $\sqrt{a}$ and then comes back to $\sqrt a$ (either after an explosion and a restart from $+\infty$ or without explosion). Note that the diffusion $X_a$ starts its first excursion at time $\iota_a$.

\medskip

In the next proposition, we control the behavior of $X_a$ during its first excursion. We refer to Figure \ref{fig:TypicalXa} for an illustration of the path $X_a$. Recall that $t_a= \ln a / \sqrt a$ and $h_a = \ln a / a^{1/4}$.
\begin{figure}[!h]
\centering
\includegraphics[width=10cm]{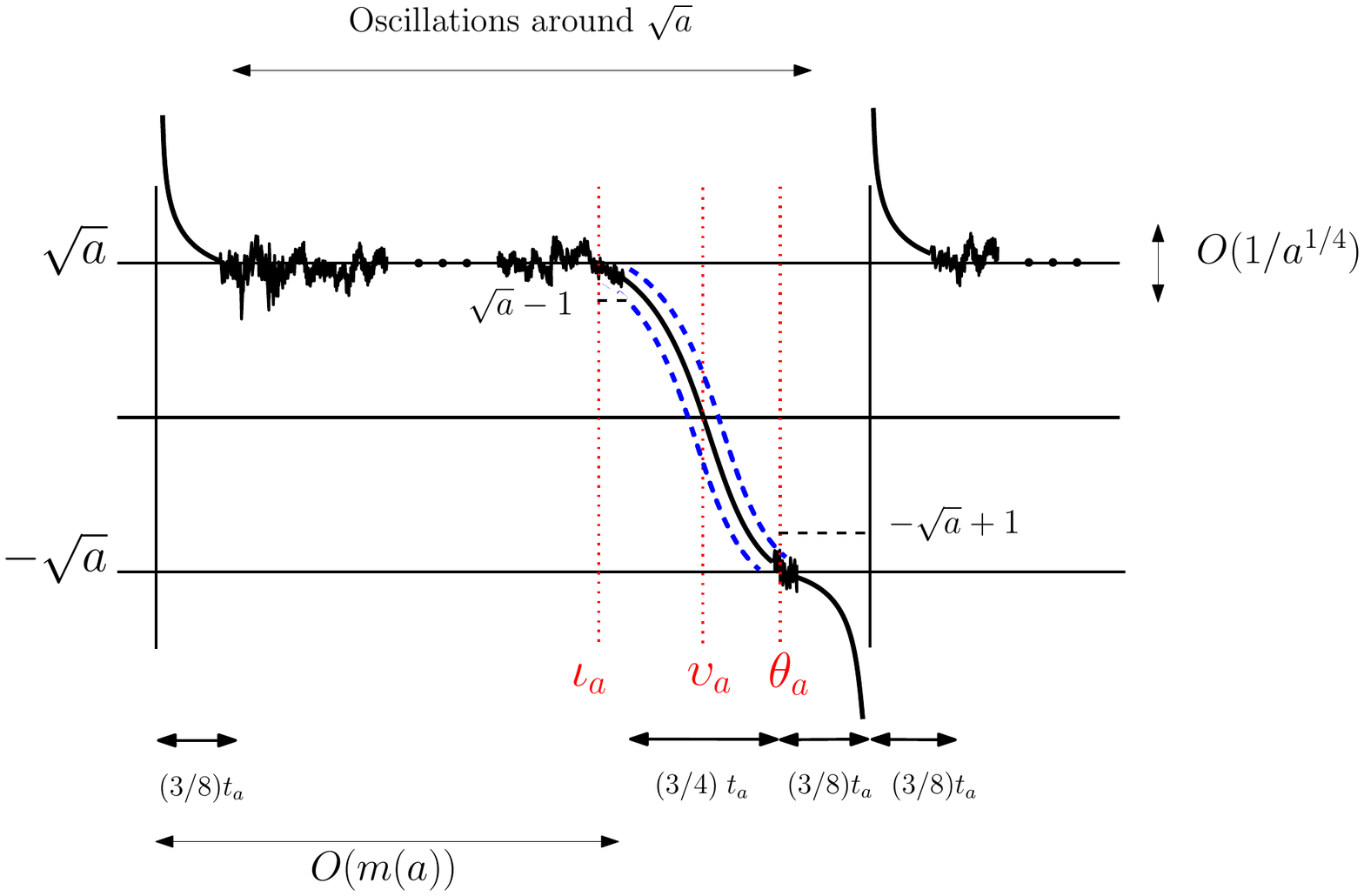}
\caption{Schematic path of $X_a$ and its first excursion to $-\sqrt{a}$ in the case of an explosion}\label{fig:TypicalXa}
\end{figure}

\begin{proposition}[Typical diffusion on its first excursion] \label{Prop:BoundsCross1}
There exists some constant $C > 0$ such that for all $a$ large enough, with a probability at least $1-O(1/\ln\ln a)$, the following holds:
\begin{itemize}
\item[(i)] \emph{Oscillations around $\sqrt{a}$}: 
\begin{align*}
\mbox{For all }t \in [\frac{3}{8} t_a, \tau_{-2\sqrt{a}}(X_a)), \quad \fint_{\frac{3}{8} t_a}^t X_a(s) ds \in \big[\sqrt{a} - h_a, \sqrt{a} + h_a\big].
\end{align*}
\item[(ii)] \emph{Crossing}: 
For all $t \in [\iota_a,\theta_a]$,
$$\big|X_a(t) - \sqrt a \tanh(-\sqrt{a}(t-\upsilon_a))\big| \le C \frac{\sqrt{a}}{\ln a}\;,$$
and
\begin{align*}
&\upsilon_a - \iota_a \ge (3/8) t_a - C \frac{\ln \ln a}{\sqrt a}\;, \\
 &\big|\theta_a-\upsilon_a - (3/8) t_a\big| \le C \frac{\ln \ln a}{\sqrt a}\;.
\end{align*}
\item[(iii)] \emph{Explosion and/or return to $\sqrt{a}$}: 
If after time $\theta_a$ the diffusion $X_a$ explodes before coming back to $+\sqrt a$, then $X_a$ stays below $-\sqrt a + 1$ in $[\theta_a, \zeta_a(1))$ and explodes within a time $(3/8)\, t_a + (\ln\ln a)^2 / (4\sqrt a)$.

If it does not explode, $X_a$ returns to $\sqrt{a}$ within a time $(3/4)\, t_a + C (\ln\ln a)^2 / \sqrt a$.
\end{itemize}
\end{proposition}

\begin{remark}~ 
\begin{itemize}
\item In this proposition we do not bound from above $\upsilon_a - \iota_a$, as we do not need to control this time for our purposes.
\item In (iii), the time necessary to explode (resp. return to $\sqrt{a}$) is in fact $(3/8) t_a + O(\ln \ln a/\sqrt{a})$ (resp. $(3/4) t_a + O(\ln \ln a/\sqrt{a})$). The choice of $(\ln\ln a)^2/\sqrt{a}$ is arbitrary and any time of the form $C_a \ln\ln a/\sqrt{a}$ with $C_a \to \infty$ would give a probability going to $1$.
\end{itemize}
\end{remark}

The next proposition controls the difference between two diffusions $X_a$ and $X_{a + \eps}$, for $\eps$ not too large, during the first excursion of $X_a$. Again, the bounds are not optimal but they suffice for our purposes.
\begin{proposition}[Coupling during the excursion]\label{Prop:BoundsCross2}
Let $\eps = \eps_a \in (a^{-2/3},1)$. There exists some constant $C > 0$ such that with a probability at least $1-O(1/\ln a)$, the following holds for all $a$ large enough:

\smallskip

For all $t \in [\upsilon_a-\frac1{16} t_a,\upsilon_a + \frac1{16} t_a]$, $\big|X_{a+\eps}(t) - X_{a}(t)\big| \le 1\;,$

\medskip

For all $t \in [\upsilon_a + \frac1{16} t_a, \theta_a- \frac{1}{16} t_a]$, $X_{a+\eps}(t) \leq -\sqrt{a} +C \,a^{3/7}$,
\medskip

For all $t \in [\upsilon_a + \frac1{16} t_a, \theta_a]$, $X_{a+\eps}(t) \leq \sqrt{a} -1$.
%
\end{proposition}

As we will see later on, the above estimates provide a good control on the first eigenfunction over the time-interval $[\iota_a, \theta_a]$ (for some well chosen $a$). We will control the first eigenfunction after time $\theta_a$ by using the time-reversed diffusions. The idea is that after the stopping time $\theta_a$, the Brownian motion makes no exceptional event with large probability so that the time reversal $\hat X_a$ should oscillate around $-\sqrt{a}$.

\begin{proposition}[Control after time $\theta_a$]\label{Prop:BoundsCross3}
For any $c \ge 1$ and any $L=L_a  \leq c\; m(a)$ the following holds true with probability greater than $1-(\ln a)^{-1/2}$ for all $a$ large enough. If $L \ge \theta_a+ 10 \,t_a$ and $\hat X_a$ does not explode before time $L- \theta_a - 10 \,t_a$, then it does not explode in the time-interval $[0,L - \theta_a]$ and we have the bounds:
\begin{align*}
 \hat{X}_a(L-\theta_a) \le -\sqrt{a} + h_a\;, \qquad \sup_{t \in [\theta_a,L]}\fint_{\theta_a}^{t} \hat{X}_a(L-s) ds \le -\sqrt{a} + h_a\;.
\end{align*}
\end{proposition}
\begin{remark}\label{rem:afterthetaforXa+eps}
For any $\eps = \eps_a  \in (a^{-2/3},1)$, a straightforward adaptation of the proof of Proposition \ref{Prop:BoundsCross3} shows that under the same conditions we have the further bounds:
\begin{align*}
 -\sqrt{a} - h_a \le \hat{X}_{a+\eps}(L-\theta_a)\;,\qquad -\sqrt{a} - h_a \le \inf_{t \in [\theta_a,L-(3/8)t_a]}\fint_{\theta_a}^{t} \hat{X}_{a+\eps}(L-s) ds\;.
\end{align*}
with probability greater than $1-(\ln a)^{-1/2}$.
\end{remark}

For the sake of readability, we postpone the proofs of those technical propositions to Section \ref{Sec:Crossing} and proceed to the proof of the main results of this paper.

\subsection{Preliminaries for the proof of Theorems \ref{Th:Main}, \ref{Th:Shape} and \ref{Th:LocalMaxZeros}}

The collection (indexed by $L$) of random variables
$$\big(4\sqrt{a_L}(\lambda_k + a_L),U_{k}/L, m_{k}\big)_{k \geq 1}\;,$$
is tight since $\mathcal{Q}_L$ is tight, see Section \ref{Sec:Poisson}, and since the remaining coordinates belong to the product space $[0,1]^\N\times \cP^\N$ which is compact. Let $\big(\lambda_k^{\infty},U^{\infty}_{k}, m^{\infty}_k\big)_{k \geq 1}$ be the limit of a converging subsequence. We already know that $(\lambda_k^{\infty})_{k \geq 1}$ is a Poisson point process of intensity $e^x dx$. Our goal is to show that for any $k\ge 1$, $(U^{\infty}_{1},\ldots,U^{\infty}_k)$ is i.i.d.~uniform over $[0,1]$ and independent of $\cQ_\infty$, that $m^\infty_i = \delta_{U^\infty_i}$ for every $i\in \{1,\cdots, k\}$ and to obtain the precise description of the eigenfunctions of Theorems \ref{Th:Shape} and \ref{Th:LocalMaxZeros}.

From now on, $k$ is fixed. As in Subsection \ref{subsec:proofeigenvalues}, we subdivide $[0,L]$ into $2^n$ macroscopic sub-intervals $[t^n_j, t^n_{j+1}]$ where $t_j^n := j 2^{-n} L$ and denote by $X_a^j := X_a^{t_j^n}$ the diffusion starting from $+\infty$ at time $t_j^n$ and by $\hat X_a^{j+1} := \hat X_a^{L - t_{j+1}^n}$ the time-reversed diffusion starting from $-\infty$ at time $L-t_{j+1}^n$. The underlying idea of the proof is that, on every time-interval $[t^n_j, t^n_{j+1}]$, the diffusion $X_a^j$ is a faithful approximation of $X_a$ until their first explosions and the content of Propositions \ref{Prop:BoundsCross1}, \ref{Prop:BoundsCross2} and \ref{Prop:BoundsCross3} can be applied to $X_a^j$ and  $\hat X_a^{j+1}$ (on an interval of length $2^{-n} L$ instead of $L$).
\begin{remark}
The aforementioned approximation is supported by the following heuristics: the restriction of every eigenfunction of $\cH_L$ to a small interval surrounding its maximum coincides, up to small errors, to the principal eigenfunction of the operator $\cH_L$ restricted to that interval.
\end{remark}

\medskip

We gather in the following event the ingredients we need to approximate accurately the Riccati transforms of the eigenfunctions by ``typical'' diffusions. Since we want very precise control simultaneously on several diffusions and since one has to rule out many undesired boundary effects, its definition is quite long. At first reading, one may skip the details and proceed to the next subsection.
\medskip

For any $\eps >0$, we consider the following discretization of mesh $\eps / \sqrt a_L$ of the interval $[a_L - 1/(\eps \sqrt a_L), a_L + 1/(\eps \sqrt a_L)]$:
$$ M_L := \left\{ a\in \Big[a_L - \frac1{\eps \sqrt a_L}, a_L + \frac1{\eps \sqrt a_L}\Big]:\; \exists n\in \bbZ,\; a=a_L + \frac{n \eps}{\sqrt{a_L}}\right\}\;.$$

From now on, we will use the notation $t_L := t_{a_L} = \ln(a_L)/\sqrt{a_L}$. Notice that, for all $a \in M_L$, the difference $t_L - t_a = O\big(\ln(a_L) / (\eps a_L^2)\big)$ is negligible compared to $1/\sqrt{a_L}$ so that the estimates already obtained carry through without modifications upon replacing $t_a$ by $t_{L}$. 

\medskip

Let $\cE(n,\eps)$ be the event on which there exists a random subset
$$\cA := \{a_k < a'_k \le a_{k-1} < a'_{k-1} \le \ldots \le a_1 < a'_1\}$$
of the grid $M_L$ such that the following conditions (a), (b) and (c) are fulfilled.

\medskip

(a) \textbf{Squeezing of the k first eigenvalues:} We have (see Figure \ref{Fig:Squeezing})
$$ a_k \le -\lambda_k < a'_k \le a_{k-1} \le \ldots < a'_2 \le a_1 \le -\lambda_1 < a'_1\;.$$

\smallskip

(b) \textbf{Control of the excursions and explosions:} For all $a\in \cA$ and all $j\in \{0,\ldots,2^{n}-1\}$ we have:\begin{enumerate}[label=(b)-(\roman*)]
\smallskip
\item None of the explosion times of $X_a$ fall at distance less than $2^{-2n} L$ from any $t_j^n$. \label{(b)-(i)}
\item Over the time-interval $[t_j^n,t_{j+1}^n]$, the diffusions $X_a$ and $X^j_a$ start at most one excursion to $-\sqrt{a}$ that hits $-\sqrt{a}$ before time $t_{j+1}^n$. Over the first and last time-intervals (i.e. $j=0$ and $j=2^n-1$), the diffusion $X^j_a$ does not hit $-\sqrt a$. \label{(b)-(ii)}
\item If $X_a^j$ explodes on $[t_j^n,t_{j+1}^n]$, then neither $X_a^{j-1}$ nor $X_a^{j+1}$ explodes on $[t_{j-1}^n,t_{j}^n]$ and $[t_{j+1}^n,t_{j+2}^n]$.\label{(b)-(iii)}
\item If $X_a^j$ explodes on $[t_j^n,t_{j+1}^n]$, then so does $X_{a_k}^j$ and their explosion times lie at a distance at most $(\ln \ln a_L)^2/(3\sqrt{a_L})$ from each other. \label{(b)-(iv)}
\end{enumerate}
and similarly for the time reversed diffusions (up to obvious modifications in the statement due to the identity in law between $X_a$ and $-\hat{X}_a$).

\smallskip

(c) \textbf{Synchronization and typical diffusions:} For all $a\in \cA$ and all $j \in \{0,\ldots,2^{n}-1\}$ we have:\begin{enumerate}[label=(c)-(\roman*)]
\smallskip
\item \emph{Synchronization of $X_a$ and  $X^j_a$:} if $t_j^n \in [\zeta_a(l-1),\zeta_a(l))$ for some $l \in \{1,\cdots,k\}$ then
\begin{align*}
&\big| X_a(t) - X_a^{j}(t)\big| \le h_a\;,\quad \forall t \in [t_j^n+\frac38 t_L,\tau^l_{-2\sqrt{a}}(X_a)]\;,\\
&\big|\tau_{-\infty}(X_a^{j}) - \zeta_a(l)\big| < \frac{C}{\sqrt{a}}\;,
\end{align*}
where $\tau^l_{-2 \sqrt{a}}(X_a) := \inf \{t \geq \zeta_a(l-1)\::\: X_a = -2 \sqrt{a}\}$.\label{(c)-(i)}
\item \emph{Typical behavior of $X^j_a$ and $X_a$:} The diffusions $X^j_a$ and $X_a$ follow the behavior described in Proposition \ref{Prop:BoundsCross1} and satisfy the conditions of the event $\cD_a^k$.\label{(c)-(ii)}
\item \emph{Coupling:} for all $a' \in \cA$, $X^j_a$ and $X^j_{a'}$ follow the behavior described in Proposition \ref{Prop:BoundsCross2}.\label{(c)-(iii)}
\item \emph{Time-reversal:} the diffusion $\hat X^{j+1}_a$ follows the behavior described in Proposition \ref{Prop:BoundsCross3}.\label{(c)-(iv)}
\end{enumerate}
and similarly for the time-reversed diffusions.

 \begin{figure}[!h]
\includegraphics[width=11cm]{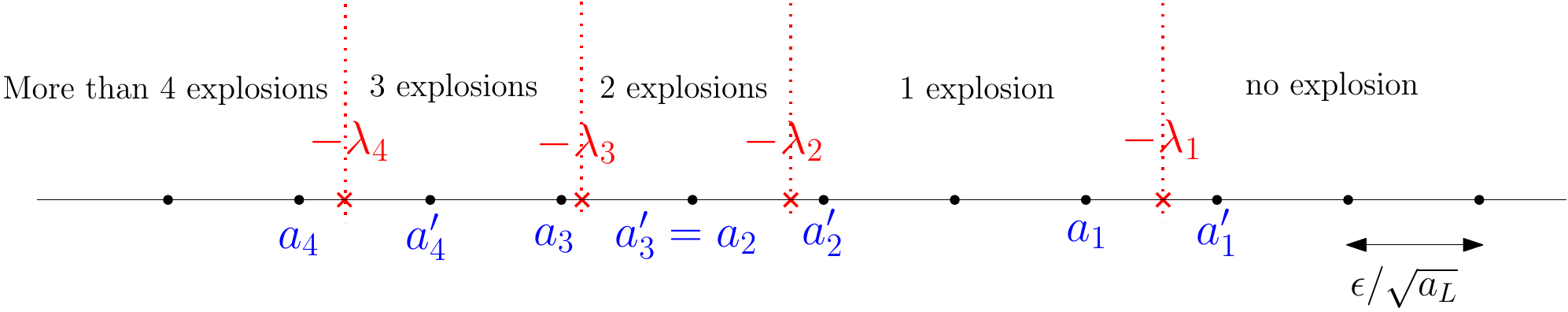}
\caption{Locations of the (opposite of the) eigenvalues together with their approximations $a_k$ and $a'_k$. Black dots correspond to points of $M_L$.}\label{Fig:Squeezing}
\end{figure}

\begin{proposition}\label{Propo:cE}
There exists a constant $C>0$ such that for all $\eps$, we have
$$ \varliminf_{n\to\infty}\varliminf_{L\to\infty} \P(\cE(n,\eps)) \ge 1-C\eps\;.$$
\end{proposition}

\noindent The proof of this proposition is postponed to Section \ref{Sec:Crossing}.

\begin{remark}
It is not difficult to prove that properties (a), \ref{(b)-(i)}, \ref{(b)-(ii)} and \ref{(b)-(iii)} hold with large probability thanks to the convergence of the associated point processes. However, the proof of \ref{(b)-(iv)} is more involved and rely on coupling arguments for different parameters $a$'s. Note that \ref{(b)-(iii)} and \ref{(b)-(iv)} are used only for the localization of the eigenfunctions $\varphi_i$ for $i \geq 2$ proved in paragraph \ref{paragraph:nexteigen}.
\end{remark}

\begin{remark}
Note that \ref{(b)-(ii)} implies that $X_a$ explodes at most once in each interval $[t_j^n,t_{j+1}^n]$. Therefore, \ref{(b)-(i)} together with the synchronization \ref{(c)-(i)} imply that $X_a$ explodes in $[t_j^n,t_{j+1}^n]$ iff $X^j_a$ explodes in $[t_j^n,t_{j+1}^n]$. We deduce that the total number of explosions of $X_a$ on $[0,L]$ equals the sum over $j$ of the number of explosions of $X_a^j$ on $[t_j^n,t_{j+1}^n]$ on the event $\cE(n,\eps)$.
\end{remark}

\begin{remark}
At a technical level, the symmetries that we rely on are similar in spirit to those appearing in~\cite{Hendrik} where the invariant measure of the stochastic Allen-Cahn equation is studied.
\end{remark}

In the next subsection, we focus on the first eigenfunction: we show that $m_1^\infty = \delta_{U_1^\infty}$, that the shape of $\varphi_1$ around $U_1$ is asymptotically given by the inverse of a hyperbolic cosine and that $\varphi_1$ decays at the exponential rate $\sqrt{a_L}$ from its localization center. In the subsequent subsection, we prove the same results for the $i$-th eigenfunction with $i\in \{2,\ldots,k\}$, together with the behavior of its local maxima and zeros. In the last subsection, we prove that the r.v. $(U_1^\infty,\ldots,U_k^\infty)$ are i.i.d.~uniform over $[0,1]$ and independent of $\cQ_\infty$. These three subsections therefore conclude the proof of Theorems \ref{Th:Main}, \ref{Th:Shape} and \ref{Th:LocalMaxZeros}.

\subsection{The first eigenfunction $\varphi_1$.}\label{subsec:1steigen}

Recall that we denote the Riccati transform of $\varphi_1$ by $\Chi_1= X_{-\lambda_1}$. Let $a := a_{1}$ and $a' := a'_{1}$ be the approximations of $-\lambda_1$. We work deterministically on the event $\cE(n,\eps)$ for some arbitrary $\eps >0$. Recall that $X_{a}$ and $\hat X_{a}$ explode exactly one time on $[0,L]$ while $X_{a'}$ and $\hat X_{a'}$ do not explode. Moreover, these diffusions are ``typical'' in the sense of (b) and (c) of $\cE(n,\eps)$.

\medskip

There is an interval $[t_j^n,t_{j+1}^n)$ that contains the (unique) explosion time of $X_a$ which occurs in $[t_j^n + 2^{-2n} L,t_{j+1}^n  - 2^{-2n} L]$ by \ref{(b)-(i)}. Using \ref{(c)-(i)}, we know that $X_a^j$ and $X_a$ synchronize after time $t_j^n + (3/8) t_L$ and that $X^j_{a}$ explodes as well (only one time, by \ref{(b)-(ii)}). Moreover, using Lemma \ref{Prop:Reversal} on the interval $[t_j^n,t_{j+1}^n]$, we deduce that $\hat X^{j+1}_{a}(L-\cdot)$ explodes one time in $[t_j^n, t_{j+1}^n]$ and by (b) and (c) again, that the time-interval $[t_j^n + 2^{-2n} L,t_{j+1}^n - 2^{-2n} L]$ contains the explosion time of $\hat X_{a}(L-\cdot)$ as well.

The diffusion $X^{j}_a$ makes only one excursion to $-\sqrt{a}$ on the time interval $[t^n_j,t^n_{j+1})$ (\ref{(b)-(ii)}) which therefore explodes to $-\infty$ before coming back to $\sqrt{a}$: let us denote by $\theta$ its first hitting time of $-\sqrt a$ after time $t_j^n$, by $\upsilon$ its last hitting time of $0$ before time $\theta$ and by $\iota$ its last hitting time of $\sqrt{a}$. The same holds for $\hat X^{j+1}_{a}$ and we use the notations $\hat \theta$, $\hat \upsilon$ and $\hat \iota$. See Figure \ref{fig:TypicalXahatXa} for an illustration of the above diffusions.

\begin{figure}[!h]
\centering
\includegraphics[width=11cm]{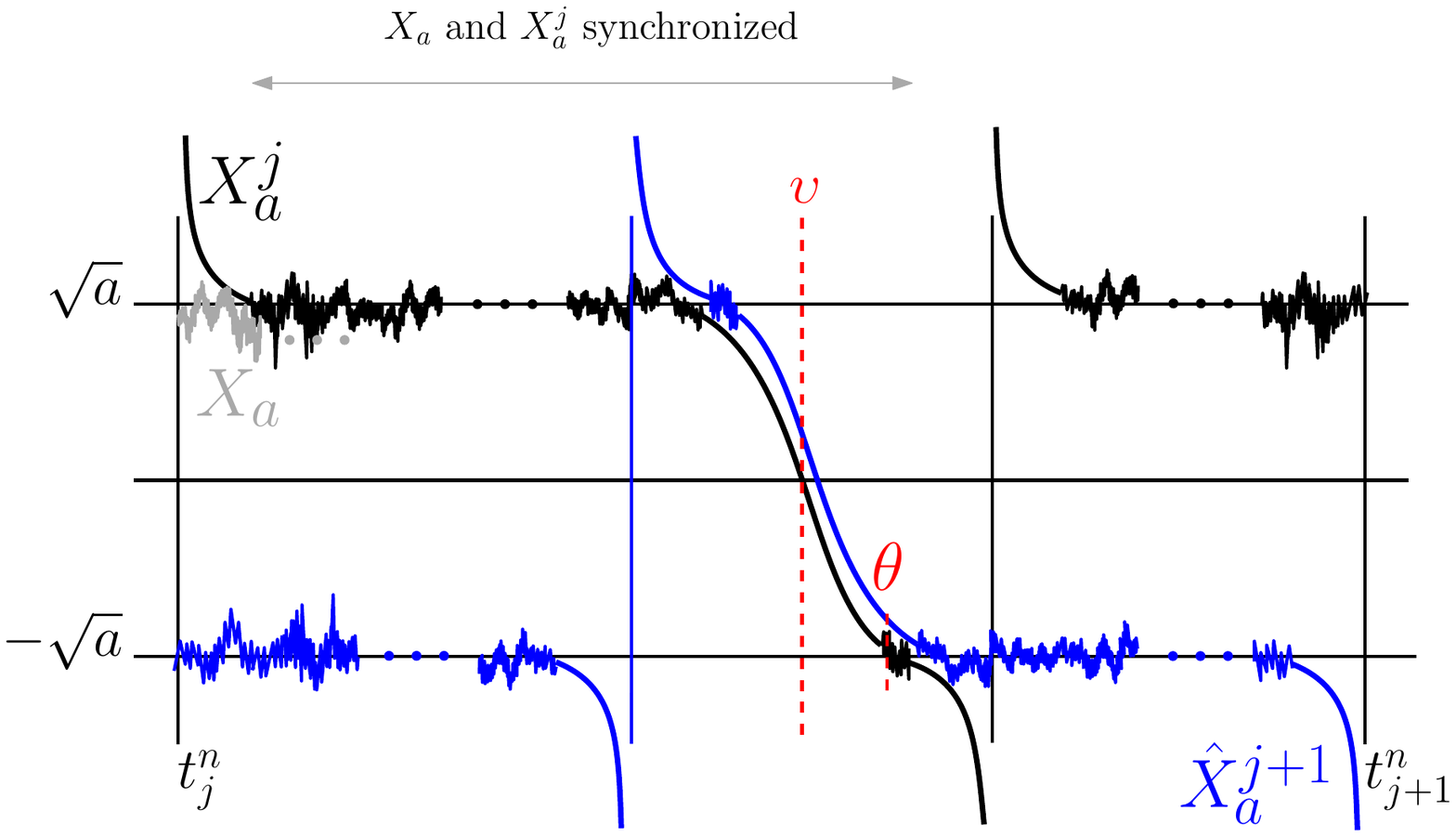}
\caption{Typical behavior of $X_{a}$, $X^j_{a}$ and $\hat X^{j+1}_{a}$ on the time-interval $[t^n_j,t^n_{j+1}]$}\label{fig:TypicalXahatXa}
\end{figure}

\medskip

For technical reasons, we are also led to set
$$ {\iota_+} := \max(\iota, t_j^n + (3/8)t_L)\;,\quad \hat{\iota}_+ := \max(\hat{\iota}, L-t_{j+1}^n + (3/8)t_L)\;.$$
Indeed, we will apply the synchronization estimates \ref{(c)-(i)} which hold true only after time $t_j^n + (3/8)t_L$.\\

The next proposition gives a lower bound on the first eigenfunction $\varphi_1$ around $\upsilon$ and the exponential decay after time $\theta$.

\begin{proposition}\label{propo:boundfirsteigenfunction} On the event $\cE(n,\eps)$, the first eigenfunction satisfies the following:
\begin{itemize}
\item[(1)] \emph{Lower bound around $\upsilon$:} 
\begin{align*}
\forall t \in [{\iota_+}, \theta],\quad &\frac{\varphi_1(t)}{\varphi_1(\upsilon)} \geq \frac1{\cosh(\sqrt{a_L}(t-\upsilon))}\Big(1- 2 C\, |t-\upsilon| \frac{\sqrt a_L}{\ln a_L}\Big)\; .
\end{align*}
\item[(2)] \emph{Exponential decay after time $\theta$:}
\begin{align*}
\forall t\in [\theta,L],\quad &\frac{\varphi_{1}(t)}{\varphi_{1}(\theta)} \le\exp\big(-(\sqrt{a_L} -\kappa_L) (t-\theta)\big) ,\\
\forall t\in [\theta,L - \frac38 t_L], \quad &\frac{\varphi_{1}(t)}{\varphi_{1}(\theta)} 
 \ge \exp\big(-(\sqrt{a_L} + \kappa_L) (t-\theta)\big)\;,
 \end{align*}
where $\kappa_L := \ln^2 a_L/a_L^{1/4}$.
\end{itemize}
\end{proposition}
Note that thanks to the time reversal, we get the upper bound:
\begin{align}\label{Eq:TimereversedBoundeigen}
\forall t \in [L-\hat \theta, L-\hat{\iota}_+],\quad &\frac{\varphi_1(t)}{\varphi_1(L-\hat \upsilon)} \leq \frac1{\cosh(\sqrt{a_L}(t-(L- \hat \upsilon)))}\Big(1+ 2 C |t-(L-\hat \upsilon)| \frac{\sqrt a_L}{\ln a_L}\Big)\;,
\end{align}
as well as the exponential growth before time  $L-\hat \theta$.

\begin{proof}
For (1), thanks to the synchronization, the Riccati transform $\Chi_1$ is bounded from below by $X^j_a - h_a$ on the time interval $[{\iota_+},\theta]$. The behavior of the first excursion of $X_a^j$ described in \ref{(c)-(ii)} yields:
\begin{align*}
\frac{\varphi_{1}(t)}{\varphi_{1}(\upsilon)} &\geq \exp(\int_{\upsilon}^t X^j_{a}(s) ds - h_a |t-\upsilon|) \\
&\geq \exp(\int_{\upsilon}^t \sqrt a \tanh(-\sqrt a(s-\upsilon)) ds - (C \frac{\sqrt a_L}{\ln a_L} +h_a) |t-\upsilon|))\\
&\geq \frac1{\cosh(\sqrt{a_L}(t-\upsilon))}\Big(1- 2 C \,|t-\upsilon| \frac{\sqrt a_L}{\ln a_L}\Big)\;,
\end{align*}
which gives the inequality (1). 
\smallskip

For (2): Let us first prove the exponential decay until $L - (3/8) t_L$.
By \ref{(c)-(ii)}, the only explosion of $X_a^j$ on $[t_{j}^n,t_{j+1}^n]$ occurs at time $\theta + (3/8) t_L(1 +o(1))$.  Applying Lemma \ref{Prop:Reversal}, we deduce that $\hat{X}_a^{j+1}(L-\cdot)$ does not explode on $[\theta+10 t_L,t_{j+1}^n]$ so that \ref{(c)-(iv)} yields the bound
\begin{equation}\label{Eq:BdMeanXj+1}
\fint_{\theta}^{t} \hat{X}_a^{j+1}(L-s) ds \le -\sqrt{a} + h_a\;,\quad t\in[\theta,t_{j+1}^n - (3/8) t_L]\;.
\end{equation}
We thus get the following bound for all $t\in [\theta,L- (3/8) t_L]$ and all $L$ large enough:
$$ \int_\theta^t \hat{X}_a(L-s) ds \le -(\sqrt a_L - \kappa_L/2)(t-\theta)\;.$$
Indeed, for $t\le t_{j+1}^n - \frac38 t_L$ this is a direct consequence of \eqref{Eq:BdMeanXj+1} and of the synchronization estimate. For $t>t_{j+1}^n- \frac38 t_L$, it comes from \eqref{Eq:BdMeanXj+1}, the synchronization estimate, the typical behavior of $\hat{X}_a$, the fact that $t_{j+1}^n - \theta > 2^{-2n} L$ and the simple calculation:
\begin{align}
\int_\theta^t \hat{X}_a (L-s) ds &= \int_\theta^{t_{j+1}^n- \frac38 t_L}\hat{X}_a(L-s) ds + \int_{t_{j+1}^n- \frac38 t_L}^{L-\frac38 t_{L}}\hat{X}_a(L-s) ds - \int_t^{L-\frac38 t_{L}}\hat{X}_a(L-s) ds \notag \\
&\le -(\sqrt{a_L}-2h_{a_L})(t_{j+1}^n- \frac38 t_L-\theta) - (\sqrt{a_L}-h_{a_L})(L-t_{j+1}^n)\notag\\
&\;\;\;\;+ (\sqrt{a_L}+h_{a_L})(L-\frac38 t_L-t)\notag\\
&\le -(\sqrt a_L-\kappa_L/2)(t-\theta)\;,\label{ineqpfUB}
\end{align}
for all $L$ large enough. Since $\Chi_1(t) \le \hat{X}_a(L-t)$ for all $t\in[\theta,L)$, we deduce the upper bound of the first inequality until time $L - (3/8) t_L$ but with the improved rate $\sqrt{a_L} -\kappa_L/2$. 

For the lower bound, the proof is similar except that one has to replace $\hat X_a$ and $\hat X_a^j$ by $\hat X_{a'}$ and $\hat X_{a'}^j$ and use Remark \ref{rem:afterthetaforXa+eps}.

For the interval $[L - (3/8) t_L, L]$, thanks to \ref{(c)-(ii)} and using the fact that $\hat X_{a'}(t) \le \Chi_{1}(L-t) \le \hat X_{a}(t)$ for all $t\in(0,\tau_{+\infty}(\hat X_a))$ we get
\begin{align*}
\sup_{t\in (0,(3/8) t_L]} \Big|\Chi_{1}(L-t)+\sqrt{a_L} \coth (\sqrt a_L (L-t))\Big| \le 1\;.
\end{align*}
We deduce that
\begin{align*}
\forall t \in [L-(3/8)t_L,L],\quad \varphi_1(t) \leq \varphi_1(L- \frac38 t_L) \exp(- (\sqrt{a_L}-1) (t - L + \frac38 t_L))\;,
\end{align*}
and the desired upper bound follows, using the exponential decay until $L-(3/8) t_L$ already established but with an improved rate, and $L-\theta \geq 2^{-2n}L- O(t_L)$ (thanks to \ref{(b)-(i)} and \ref{(c)-(ii)}).

\end{proof}

We will now show that the times $\upsilon$, $L- \hat \upsilon$ and $U_1$ are very close to each other so that the time intervals $[\iota_+, \theta]$ and $[L-\hat \theta, L- \hat{\iota}_+]$ overlap and the previous proposition thus describes the behavior of the first eigenfunction around its maximum. This is the key point and most difficult part of the proof. It relies on the ordering of the forward/backward diffusions, the coupling of $X_a$ and $X_{a'}$ and the exponential decays of Proposition \ref{propo:boundfirsteigenfunction}.
\begin{proposition}\label{propo:upsilonhatupsilon}
On the event $\cE(n,\eps)$, the eigenfunction $|\varphi_1|$ reaches its maximum at distance at most $2 C/(\sqrt a_L \, \ln a_L)$ from $\upsilon$ and $L-\hat \upsilon$. Moreover, for all $t \in  [L- \hat \theta, \theta]$, we have
\begin{align}\label{equivaroundU1}
\frac{\varphi_1(t)}{\varphi_1(U_1)} = \frac{1}{\cosh(\sqrt{a_L}(t-U_1))}\Big( 1+ O\Big(|t-U_1| \frac{\sqrt a_L}{\ln a_L}\Big) + O\Big(\frac{1}{\ln a_L}\Big)\Big)\; 
\end{align}
\end{proposition}

\begin{proof} 

Let us first prove that all the points where the eigenfunction $|\varphi_1|$ reaches its maximum over the time-interval $[\upsilon - (1/16)t_L,L]$ lie at distance at most $2 C/(\sqrt a_L \ln a_L)$ from $\upsilon$. Indeed:
\begin{itemize}
\item Using (c), for all $t \in [\upsilon - (1/16)t_L, \upsilon + (1/16) t_L]$, we have 
$X^j_a -1 \le \Chi_1(t) \leq X^j_{a'}(t) \leq X^j_a +1$ and:
\begin{align}\label{ineqChi1}
 - C \frac{\sqrt a_L}{\ln a_L} - 1 \le \Chi_1(t)-\sqrt a_L \tanh(-\sqrt a_L(t-\upsilon)) \leq C \frac{\sqrt a_L}{\ln a_L} + 1.
\end{align}
This implies that all the points where $|\varphi_1|$ reaches its maximum over $[\upsilon - (1/16)t_L, \upsilon + (1/16) t_L]$ lie at distance at most $2 C/(\sqrt a_L \ln a_L)$ from $\upsilon$.
\item By \ref{(c)-(iii)}, the diffusion $X_{a'}^j$ is below $-\sqrt{a}(1 + o(1))$ on the interval $[\upsilon + (1/16) t_L, \theta - (1/16) t_L]$. Therefore, $|\varphi_1|$ decreases on this time interval and we have the bound
$$ |\varphi_1(\theta - \frac{1}{16} t_L)| \le |\varphi_1(\upsilon + \frac{1}{16} t_L)| \exp\Big( - \sqrt{a}\,(1 + o(1)) t_L \big(\frac{1}{4} +o(1)\big)\Big)\;,$$
where we used $\theta - \upsilon = (3/8)t_L(1+o(1))$. Since $X_{a'}^j$ stays below $\sqrt{a}$ on $[\theta - (1/16) t_L, \theta]$, we deduce that $|\varphi_1|$ remains below $|\varphi_1(\upsilon + \frac{1}{16} t_L)|$ on the time interval $[\upsilon + (1/16) t_L, \theta]$.
\item From Proposition \ref{propo:boundfirsteigenfunction}, the eigenfunction $|\varphi_1|$ decays exponentially on the time-interval $[\theta,L]$.
\end{itemize}

By time-reversal, the eigenfunction $|\varphi_1|$ reaches its maximum over the time-interval $[0, L- \hat \upsilon + (1/16)t_L]$ at distance at most $2 C/(\sqrt a_L \ln a_L)$ from $L- \hat \upsilon$.

Note that in the time-interval $[L- \hat \tau_{+\infty}(\hat X^{j+1}_a), \tau_{-\infty}(X^j_{a})]$, the diffusion $\hat X^{j+1}_{a}(L- \cdot)$ is bounded from below by $X^j_{a}(\cdot)$ as the Riccati transform of the first eigenfunction of the operator
$$ -\partial^2_x + \xi\;,\quad x\in(t_j^n,t_{j+1}^n)\;,$$
stays in between those two diffusions (see Figure \ref{fig:TypicalXahatXa}). Using that the diffusion $X_a^j$ is bounded from below by $t \mapsto \sqrt{a} \tanh(-\sqrt{a} (t - \upsilon)) - C \sqrt{a}/\ln a$ on the time interval $[\iota, \upsilon]$, we deduce that $\hat{X}_a^{j+1}(L-\cdot)$ cannot vanish on the time-interval $[\iota, \upsilon - 2 C/(\sqrt{a} \ln a)]$. As it reaches $\sqrt{a}$ in $[\iota,t_{j+1}^{n}]$ and makes at most one excursion to $\sqrt{a}$ on $[t_j^n, t_{j+1}^n]$, we deduce that $L - \hat \theta$ is larger than $\iota$ and $L - \hat \upsilon$ is larger than $\upsilon - 2C/(\sqrt{a_L} \ln a_L)$. This implies that the maximum over the time-interval $[\upsilon - (1/16)t_L,L]$ is a maximum over the whole interval $[0,L]$. Therefore $U_1$ and $\upsilon$ are at distance smaller than $2C/(\sqrt{a_L} \ln a_L)$ from each other and the same holds for $L - \hat \upsilon$.\\
Using \eqref{ineqChi1}, we obtain for all $t \in [\upsilon - (1/16)t_L, \upsilon + (1/16) t_L]$
\begin{align*}
\frac1{\cosh(\sqrt{a_L}(t-\upsilon))}\Big(1- 2 C\, |t-\upsilon| \frac{\sqrt a_L}{\ln a_L}\Big) \le \frac{\varphi_1(t)}{\varphi_1(\upsilon)} \leq \frac1{\cosh(\sqrt{a_L}(t-\upsilon))}\Big(1+ 2 C\, |t-\upsilon| \frac{\sqrt a_L}{\ln a_L}\Big)\;,
\end{align*}
and therefore $\varphi_1(U_1) / \varphi_1(\upsilon)$ and $\varphi_1(U_1) / \varphi_1(L - \hat\upsilon)$ are both of order $1+ O(1/\ln^2 a_L)$.\\
Note that $t_j^n + (3/8)t_L < L- \hat \theta $ since otherwise there would be an explosion of $\hat{X}_a(L-\cdot)$ close to $t_j^n$ thus violating \ref{(b)-(i)}; similarly we have $\theta < t^n_{j+1} - (3/8)t_L$. Then, combining the lower bound of Proposition \ref{propo:boundfirsteigenfunction}-(1) and its time-reversed version \eqref{Eq:TimereversedBoundeigen}, together with the inequalities $\iota_+ < L- \hat \theta < \theta < L- \hat{\iota}_+$ we obtain \eqref{equivaroundU1}.
\end{proof}

%

We now conclude the proof of the statements of Theorems \ref{Th:Main} and \ref{Th:Shape} regarding the first eigenfunction (except the statement on the law of $U_1^\infty$, which is presented in Subsection \ref{Subsec:Unif}). We work on the event $\cE(n,\eps)$:
\begin{itemize}
\item \textbf{Exponential decay.}

The simple inequalities $e^{-|x|} \leq 1/\cosh(x) \leq 2 \,e^{-|x|}$ applied to \eqref{equivaroundU1}, together with Proposition \ref{propo:boundfirsteigenfunction}-(2) and its time-reversed version give the exponential decay on $[0,L]$. 

%

\smallskip
\item \textbf{Localization around $U_1$ and shape of the eigenfunction.}

Using the exponential decay at rate $\sqrt{a_L}$, we easily get
\begin{align*}
m_1([\theta/L,1]) \le \varphi_1(\theta)^2 \,O\Big(\frac{1}{\sqrt{a_L}}\Big)\;.
\end{align*}
Moreover, the time reversal gives:
\begin{align*}
m_1([0, (L- \hat \theta)/L]) \le \varphi_1(L- \hat \theta)^2 \,O\Big(\frac{1}{\sqrt{a_L}}\Big) \;.
\end{align*}
Integrating  \eqref{equivaroundU1} of Proposition \ref{propo:upsilonhatupsilon}, we obtain
\begin{align*}
m_1([L- \hat \theta/L, \theta/L]) = \varphi_1(U_1)^2 \frac{2}{\sqrt a_L} \Big(1 +O\Big(\frac{1}{ \ln a_L}\Big)\Big)\;.
\end{align*}
Moreover, using that $\theta - U_1 = (3/8) t_L (1 + o(1))$, we get
\begin{align*}
\frac{\varphi_1(\theta)}{\varphi_1(U_1)} \leq \exp(-(3/8) \ln a_L( 1+ o(1)))\;,
\end{align*}
and similarly $|\varphi_1(L- \hat \theta)| \leq |\varphi_1(U_1)|\,\exp(-(3/8) \ln a_L( 1+ o(1)))$.

As $m_1$ is a probability measure, we deduce that $\varphi_{1}^2(U_1) = (\sqrt a_L/2) (1+o(1))$ and the statement about the shape of the first eigenfunction follows. To obtain the shape of the Brownian motion, it suffices to combine the identity
$$ \Chi_1(t) = \Chi_1(U_1) + \int_{U_1}^t (-\lambda_1 - \Chi_1(s)^2) ds + B(t)-B(U_1)\;,$$
with the estimate \eqref{ineqChi1}, noticing that $|a_L+\lambda_1| = O(1/\sqrt a_L)$.\\
It is then easy to see using again \eqref{equivaroundU1} that any interval centered around $U_1$ and of length much greater than $1/\sqrt{a_L}$ has a mass going to $1$ as $L \to \infty$ so that $m_1$ is asymptotically as close as desired to a Dirac mass at $U_1/L$.

\end{itemize}

\subsection{The next eigenfunctions.}\label{paragraph:nexteigen}

We now turn to the $i$-th eigenfunction, for some $i\in \{2,\ldots,k\}$. Let $a = a_{i}$ and $a' = a_{i}'$. Again, we work deterministically on the event $\cE(n,\eps)$. 

The main idea is that the trajectory $\Chi_i$ follows the forward diffusions $X_{a_i}$ and $X_{a'_{i}}$ until the additional explosion of $X_{a_i}$. It then follows the time-reversed diffusions $\hat X_{a_i}$ and $\hat X_{a_i'}$ up to time $L$ (see Figure \ref{RiccatiT4emeVP}).

\begin{figure}[!h]
\centering
\includegraphics[width = 10cm]{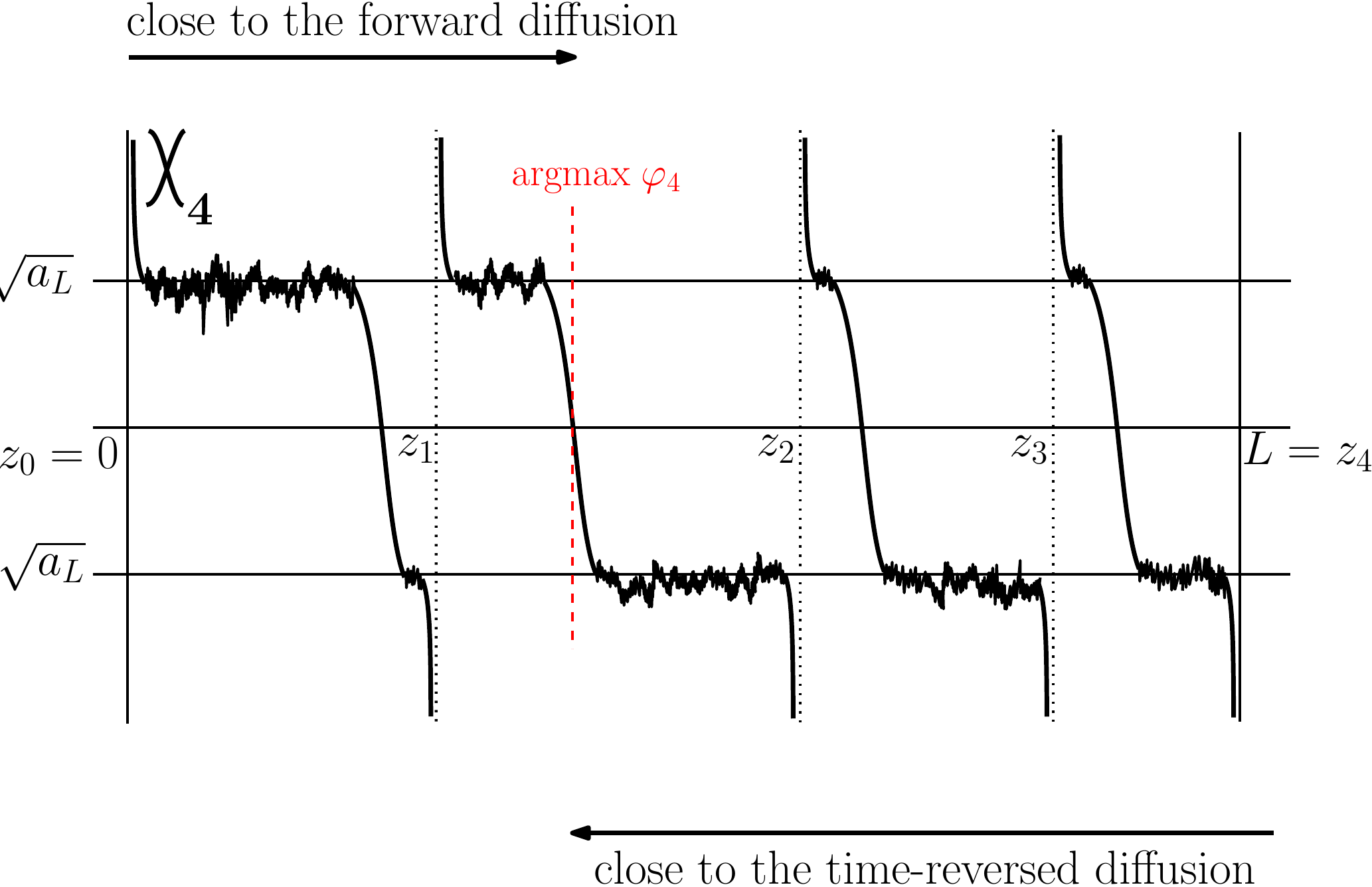}
\caption{Schematic path of $\Chi_4$.}\label{RiccatiT4emeVP}
\end{figure}

There exist $j_1 < j_2 < \ldots < j_i$ such that the $i$ explosion times of $X_a$ lie in the intervals $[t_{j_\ell}^n,t_{j_\ell +1}^n)$, $\ell=1,\ldots, i$. Each $X^{j_\ell}_a$ explodes once, by \ref{(b)-(i)} and \ref{(c)-(i)}, and makes only one excursion to $-\sqrt{a}$, by \ref{(b)-(ii)}. We then call 
$\theta_{\ell}$ the location of its first hitting time of $-\sqrt a$ (resp. 
$\hat \theta_{\ell}$ for the time-reversed diffusions). There exists a unique $i_* \in \{1,\ldots,i\}$ such that, for $j_*=j_{i_*}$, $X^{j_*}_a$ explodes but $X^{j_*}_{a'}$ does not. Without loss of generality, we can suppose that $i_* \ge 2$: indeed, the case $i_*=1$ corresponds to the case $i_*=i$ upon reversing time.

Let us now study $\varphi_{i}$. We denote by $0=z_0 < z_1 < \ldots < z_i =L$ the zeros of $\varphi_{i}$, or equivalently the explosion times of $\Chi_{i}$. We also let $x_\ell$ be the (first if many) point where $|\varphi_i|$ reaches its maximum on the interval $[z_{\ell-1},z_\ell]$.

We first describe the behavior of $\varphi_i$ on  $[z_{\ell-1},z_\ell]$ with $\ell < i_*$. 
\begin{lemma}\label{Lemma:SuccessiveMax}
Let $\ell < i_*$. On the event $\cE(n,\eps)$, we have:
\begin{itemize}
\item[(I)] Deterministic entrance and explosion: 
\begin{align*}
\forall t \in [z_{\ell-1}, z_{\ell-1} + \frac38 t_L],\quad \varphi_i(t)&=\varphi_i'(z_{\ell-1}) \frac{\sinh(\sqrt{a_L} (t-z_{\ell-1}))}{\sqrt a_L} (1+o(1))\;,\\
\forall t \in [z_{\ell}-\frac38 t_L, z_{\ell}],\quad \varphi_i(t)&=\varphi_i'(z_{\ell}) \frac{\sinh(\sqrt{a_L} (t-z_{\ell}))}{\sqrt a_L} (1+o(1))\;.
\end{align*}
\item[(II)] Exponential decay: For all $t\in [z_{\ell -1}+\frac38 t_{L}, z_{\ell}-\frac38 t_{L}]$
\begin{align*}
\frac1{100}\exp\big(-(\sqrt{a_L} + \frac1{2}\kappa_L) |t-x_\ell|\big) \le \frac{\varphi_{i}(t)}{\varphi_{i}(x_\ell)} \le 100 \exp\big(-(\sqrt{a_L} - \frac1{2}\kappa_L) |t-x_\ell|\big) ,
\end{align*}
where $\kappa_L := \ln^2 a_L/a_L^{1/4}$.
\smallskip
\item[(III)] Inverse of hyperbolic cosine around $x_\ell$:
\begin{align*}
\forall t\in [L - \hat \theta_{\ell}, \theta_\ell], \quad \frac{\varphi_i(t)}{\varphi_i(x_\ell)} = \frac{1}{\cosh(\sqrt{a_L}(t-x_\ell))}\Big( 1+ O\Big(|t-x_\ell| \frac{\sqrt a_L}{\ln a_L}\Big) + O\Big(\frac{1}{\ln a_L}\Big)\Big)\;.
\end{align*}
\item[(IV)] We have the bounds
$$ |x_\ell + 3/8 t_L - \theta_\ell| \le \frac{5C \ln\ln a_L}{\sqrt a_L}\;,\quad |z_\ell - \theta_\ell - 3/8 t_L | \le \frac23 \frac{(\ln\ln a_L)^2}{\sqrt a_L}\;.$$
\end{itemize}
\end{lemma}

\begin{proof}
We treat in detail the case $\ell = 1$. The other cases follow from the same arguments, one simply has to notice that the diffusions $X_a$ and $X_{a'}$ have a delay at the starting time of the interval, but since this delay is negligible compared to $t_{L}$ the proof carries through (indeed, the only probabilistic estimate in the proof of Lemma \ref{Lemma:CDI} is a control of the Brownian motion on $[0,t_L]$).

Set $j=j_1$. Let $\theta_a$, $\theta_{a'}$ be the first hitting times of $-\sqrt a$, $-\sqrt{a'}$ by $X_a^{j}$ and $X_{a'}^{j}$, and similarly for $\upsilon_a, (\iota_a)_+, \upsilon_{a'},(\iota_{a'})_+$. Moreover, let $\hat\theta_a$, $\hat\theta_{a'}$ be the first hitting times of $\sqrt a$, $\sqrt{a'}$ by $\hat{X}_a^{j+1}$ and $\hat{X}_{a'}^{j+1}$. Applying the arguments of the proof of Proposition \ref{propo:boundfirsteigenfunction}-(1) we obtain
\begin{align*}
\forall t \in [(\iota_a)_+, \theta_a],\quad &\frac{\varphi_i(t)}{\varphi_i(\upsilon_a)} \geq \frac1{\cosh(\sqrt{a_L}(t-\upsilon_a))}\Big(1- 2 C\, |t-\upsilon_a| \frac{\sqrt a_L}{\ln a_L}\Big)\;.
\end{align*}
Applying the arguments of the proof of the \emph{time-reversed} version of Proposition \ref{propo:boundfirsteigenfunction}-(2) we obtain
\begin{equation}\label{Eq:Bnd0Lhat}\begin{split}
\forall t\in [0,L-\hat\theta_a],\quad &\frac{\varphi_{i}(t)}{\varphi_{i}(L-\hat\theta_a)} \le\exp\big(-(\sqrt{a_L} -\frac1{2}\kappa_L) (L-\hat\theta_a-t)\big) \;,\\
\forall t\in [(3/8) t_L,L-\hat\theta_a], \quad &\frac{\varphi_{i}(t)}{\varphi_{i}(L-\hat\theta_a)} \ge \exp\big(-(\sqrt{a_L} +\frac1{2}\kappa_L) (L-\hat\theta_a-t)\big)\;.\end{split}
 \end{equation}
Note that, compared to the original version of the bound, here we have an additional factor $1/2$ in front of $\kappa_L$ but in the proof, this prefactor can actually be chosen arbitrarily. This is enough to deduce that the maximum of $|\varphi_i|$ over $[0,L-\hat\theta_a]$ is achieved at $L-\hat\theta_a$.\\

The proof now deviates from that presented in the previous subsection for the first eigenfunction. Indeed, we are in a situation where $X_a$ and $X_{a'}$ remain close to each other and explode roughly at the same time (while the latter did not explode in the previous case). More precisely, by \ref{(b)-(iv)} the explosions times of $X_a^j$ and $X_{a'}^j$ lie at distance at most $(\ln \ln a_L)^2/(3\sqrt a_L)$ from each other. Since with \ref{(c)-(ii)}, their behaviors are almost deterministic on a time-interval of length $(3/8 + 3/4) t_L$ before their explosion times and since they are very close to each other on $[\upsilon_a-(1/16)t_L,\upsilon_a+(1/16)t_L]$ we deduce that
\begin{equation}\label{Eq:upsthet}
0 \le \upsilon_{a'} - \upsilon_a \le 2C /(\ln a_L \sqrt{a_L})\;,\quad\mbox{which implies }\quad  0 \le \theta_{a'} - \theta_a \le 3 C \ln\ln a_L / \sqrt{a_L}\;.
\end{equation}

By monotonicity $X_a \le \Chi_i \le X_{a'}$ until the first explosion time of $X_a$. Using the synchronization estimates \ref{(c)-(i)}, together with \ref{(c)-(ii)} on $X_a^j$ and $X_{a'}^j$ and \eqref{Eq:upsthet}, we deduce that
\begin{align}
&\forall t \in [(\iota_a)_+, \theta_a], \quad - C \frac{\sqrt a_L}{\ln a_L} - 1 \le \Chi_i(t)-\sqrt a_L \tanh(-\sqrt a_L(t-\upsilon_a)) \label{Eq:ChiTanh1}\\
&\forall t \in [(\iota_{a'})_+, \theta_{a'}], \quad  \Chi_i(t)-\sqrt a_L \tanh(-\sqrt a_L(t-\upsilon_a)) \leq 3C \frac{\sqrt a_L}{\ln a_L} + 1\;.  \label{Eq:ChiTanh2}
\end{align}
In addition, $X_{a'}^j$ remains below $-\sqrt a_L + 1$ on $[\theta_{a'},\zeta_{a'}]$, and it cannot be higher than $-\sqrt a_L + 2C \sqrt a_L/\ln a_L$ on the interval of time $[\theta_a,\theta_{a'}]$ by \ref{(c)-(ii)} and \eqref{Eq:upsthet}. We thus deduce that $\Chi_i(t) \le -\sqrt{a_L} + 2C \sqrt a_L/\ln a_L$ for all $t\in [\theta_a,z_1]$. Therefore all the points where the maximum of $|\varphi_i|$ over $[(\iota_a)_+ \vee (\iota_{a'})_+,z_1]$ is achieved lie at distance at most $4C/(\ln a_L \sqrt{a_L})$ from $\upsilon_a$.\\

Let us now prove that the intervals $[0,L-\hat{\theta}_a]$ and $[(\iota_a)_+ \vee (\iota_{a'})_+,z_1]$ overlap (here the argument is the same as in the previous subsection). To that end, we observe that $X^j_{a'}(t) \le \hat{X}^{j+1}_{a}(L-t)$ for all $t\in [L-\hat{\tau}_{+\infty}(\hat{X}^{j+1}_{a}), \tau_{-\infty}(X^j_{a'})]$ and that the latter interval is not empty (these two assertions follow from Lemma \ref{Prop:Reversal} and from the fact that the Riccati transform of the first eigenfunction of $-\partial^2_x + \xi$ restricted to $(t_j^n,t_{j+1}^n)$ lies in between these two diffusions). Henceforth $L-\hat\theta_a > \iota_{a'}$ and $L-\hat{\upsilon}_a > \upsilon_{a'} - C /(\ln a_L \sqrt{a_L}) >  \upsilon_{a} - C /(\ln a_L \sqrt{a_L})$. Moreover, $L-\hat\theta_a > t_j^n +(3/8)t_L$ since otherwise $\hat{X}_a(L-\cdot)$ would explode close to $t^n_j$. Therefore, $L-\hat\theta_a > (\iota_{a'})_+$. The same argument applies upon replacing $X^j_{a'}$ by $X^j_{a}$ and yields the inequality $L-\hat\theta_a > (\iota_{a})_+$ which concludes the proof of the overlapping.\\

Consequently, all the points where the maximum of $|\varphi_i|$ over $[0,z_1]$ is achieved lie at distance at most $4C/(\ln a_L \sqrt{a_L})$ from $\upsilon_a$. This proves the first bound of (IV). The second bound of (IV) is a consequence of the squeezing $X_a \le \Chi_i \le X_{a'}$, of \ref{(c)-(i)} and \ref{(c)-(ii)} and of \eqref{Eq:upsthet}.\\
Assertion (III) is deduced from $|x_1-\upsilon_a| \le 4C/(\ln a_L \sqrt{a_L})$, combined with \eqref{Eq:ChiTanh1}, \eqref{Eq:ChiTanh2} and the inequality $(\iota_a)_+ \vee (\iota_{a'})_+ \le L-\hat\theta_a$. It also ensures that for all $t\in [L-\hat\theta_a,\theta_a]$
$$ \frac1{10} \exp(-(\sqrt a_L+\frac12 \kappa_L) |t-x_1|) \le \frac{\varphi_i(t)}{\varphi_i(x_1)} \le 10 \exp(-(\sqrt a_L-\frac12 \kappa_L) |t-x_1|)\;.$$
Using \eqref{Eq:Bnd0Lhat}, these inequalities still hold on $[(3/8)t_L, \theta_a]$. It remains to control the decay on $[\theta_a,z_1-(3/8)t_L]$: by the second bound of (IV) this interval has length at most $(\ln\ln a_L)^2/\sqrt a_L$. Furthermore on this interval of time, we have already seen that $\Chi_i$ remains below $-\sqrt a_L + 2C \sqrt a_L/\ln a_L$ and by (c)-(ii) it remains above $-\sqrt a_L -2$. Henceforth, for all $t\in [\theta_a,z_1-(3/8)t_L]$
$$  \frac12 \exp\big(-\sqrt a_L (t-\theta_a)\big) \le \frac{\varphi_i(t)}{\varphi_i(\theta_a)} \le 2\exp\big(-\sqrt a_L (t-\theta_a)\big)\;.$$
This is enough to get (II).

Recall that $X_a(t) \le \Chi_{i}(t) $ for all $t\in (0,\zeta_{a}(1))$ and  $\Chi_{i}(t) \le X_{a'}(t)$ for all $t\in (0,z_1)$. The deterministic behavior of $X_a$ and $X_{a'}$ near their starting point gives:
\begin{align}\label{Eq:ControlChii}
\sup_{t\in (0,(3/8) t_L]} \Big|\Chi_{i}(t)-\sqrt{a_L} \coth (\sqrt a_L t)\Big| \le 1\;.
\end{align}
Since $\varphi_i'(t) = \Chi_i(t) \varphi_i(t)$, we deduce that for all $0 < t_0 <t < \tau_{-\infty}(\Chi_i)$ we have
$$ \varphi_i(t) = \frac{\varphi_i'(t_0)}{\Chi_i(t_0)} e^{\int_{t_0}^t \Chi_i(s) ds}\;.$$
Using \eqref{Eq:ControlChii} and passing to the limit as $t_0 \downarrow 0$, we deduce that for all $t\in (0,\frac38 t_L]$, we obtain
\begin{equation}\label{Eq:ExactEigen0}
\varphi_i(t) = \varphi_i'(0) \frac{\sinh(\sqrt{a_L} t)}{\sqrt a_L} (1+O(t_L))\;.
\end{equation}

Regarding the behavior near $z_1$: both $X_a$ and $X_{a'}$ explode at a distance negligible compared to $t_L$ from each other thanks to \ref{(b)-(iv)}. The first explosion time of $\Chi_{i}$ necessarily lies in between those two explosion times. The proof of Lemma \ref{Lemma:Explo} then ensures that if we let $y=-\sqrt{a_L} - h_L$ then
\begin{align*}
 \forall t\in [\tau_{y}(\Chi_{i}),z_1),\quad  \Big|\Chi_{i}(t)+\sqrt{a_L} \coth (\sqrt a_L (z_1-t))\Big| \le 1\;.
\end{align*}
Indeed, the only probabilistic ingredient in that proof is a control on the Brownian motion on an interval of size $t_L$ before the explosion time: on the event $\cE(n,\eps)$, we have this control before the explosion times of $X_{a'}$ and $X_a$ so that the control holds true before the explosion time of $\Chi_{i}$. A similar calculation as above then shows that for all $t\in [\tau_{y}(\Chi_{i}),z_1]$, we have
$$ \varphi_i(t) = \varphi_i'(z_1) \frac{\sinh(\sqrt{a_L}(t-z_1))}{\sqrt a_L} (1+O(t_L))\;.$$
Since $z_1 - \tau_{y}(\Chi_{i}) \ge (3/8)t_L - \ln \ln a_L / \sqrt a_L$, and since $\Chi_i(t) \le -\sqrt a_L + 2C \sqrt a_L/\ln a_L$ for all $t\in [\theta_a - \ln \ln a_L / \sqrt a_L,  \tau_{y}(\Chi_{i})]$ we deduce that for all $t\in [z_1-(3/8)t_L,z_1]$
$$ \varphi_i(t) = \varphi_i'(z_1) \frac{\sinh(\sqrt{a_L}(t-z_1))}{\sqrt a_L} (1+o(1))\;.$$
\end{proof}

Let us now describe what happens on the interval $[z_{i_*-1},z_{i_*}]$.
\begin{lemma}\label{Lemma:SuccessiveMax2}
On the event $\cE(n,\eps)$, the eigenfunction $|\varphi_i|$ satisfies (I)-(III) of Lemma \ref{Lemma:SuccessiveMax} on the time interval $[z_{i_*-1},z_{i_*}]$.
Moreover all the points where the maximum over $[z_{i_*-1},z_{i_*}]$ is reached lie at distance at most $4C / (\ln a_L \sqrt{a_L})$ from each other, and we have $|x_{i_*} + 3/8 t_L - \theta_{i_*}| \le 5C \ln\ln a_L /\sqrt a_L$.
\end{lemma}
\begin{proof}
The proof follows from the same arguments as in the case of the first eigenfunction, one simply has to take into account the delay between $X_{a}$ and $X_{a'}$ at the beginning of the interval but this delay is negligible compared to $t_L$.
\end{proof}

Applying successively Lemma \ref{Lemma:SuccessiveMax} and its time-reversed version, together with Lemma \ref{Lemma:SuccessiveMax2}, a straightforward calculation then shows that for all $t\in [0,L]$
\begin{align*}
\frac1{100}\exp\big(-(\sqrt{a_L} + \kappa_L) |t-x_{i_*}|\big)1_{\{t \in D\}} \le \frac{\varphi_{i}(t)}{\varphi_{i}(x_{i_*})} \le 100 \exp\big(-(\sqrt{a_L} - \kappa_L) |t-x_{i_*}|\big) \;,
\end{align*}
where the set $D$ is defined in the statement of Theorem \ref{Th:Shape}. (Note that the improved exponential rate $\sqrt a_L - (1/2) \kappa_L$ allows to compensate the ``absence'' of decay near the zeros).\\

Recall that by \ref{(b)-(iii)}, $|z_\ell-x_{i_*}| > 2^{-n}L$ for all $\ell \ne i_*$. Hence the maximum of $|\varphi_i|$ on $[0,L]$ is $|\varphi_i(x_{i_*})|$ and necessarily $U_i = x_{i_*}$. It is therefore easy to see that
$$ m_i\Big([0,1] \backslash\Big[\frac{z_{i_*-1}}{L},\frac{z_{i_*}}{L}\Big]\Big) \lesssim \varphi_i(x_{i_*})^2 L e^{-\frac12 \sqrt a_L 2^{-n} L}\;.$$
By Lemma \ref{Lemma:SuccessiveMax2}, we also get
$$ m_i\Big(\Big[\frac{z_{i_*-1}}{L},\frac{z_{i_*}}{L}\Big]\Big)= \varphi_i(x_{i_*})^2 \frac{2}{\sqrt a_L} (1+o(1))\;.$$
Therefore,
$$ m_{i}([0,1]) = \varphi_i(x_{i_*})^2 \frac{2}{\sqrt a_L} (1+o(1))\;.$$
Since $m_i$ is a probability measure, we deduce that
$$ |\varphi_i(x_{i_*})| \sim a_L^{1/4} / \sqrt 2\;.$$
Finally, we claim that for all $\ell \in \{1,\ldots,i\}$ we have the bound
$$|x_\ell - U_{\sigma(\ell)}| \le 8C / (\ln a_L \sqrt{a_L})\;.$$
The cases $\ell < i_*$ and $\ell > i_*$ are symmetric, so we only consider the former. Let us provide the arguments in the case $\ell = 1$, using the notations of the proof of Lemma \ref{Lemma:SuccessiveMax} (in particular $j=j_1$). If we let $p$ be the smallest integer such that $X_{a_p}^j$ explodes on $[t_j^n,t_{j+1}^n]$ but $X_{a_p'}^j$ does not, then $\upsilon_{a_p}$ lies at a distance at most $2C / (\sqrt{a_L} \ln a_L)$ from $U_{\sigma(1)}$ and from $\upsilon_a$. Since $x_1$ lies at a distance at most $4C / (\sqrt{a_L} \ln a_L)$ from $\upsilon_a$, the claim follows. Note that $U_{\sigma(i_*)} = U_i$.\\
The rest of the statements of Theorems \ref{Th:Main}, \ref{Th:Shape} and \ref{Th:LocalMaxZeros} regarding the $i$-th eigenfunction directly follows from Lemma \ref{Lemma:SuccessiveMax} and Lemma \ref{Lemma:SuccessiveMax2}, except the statement on the law of $U_i^\infty$ which will be proven in the next subsection.

\subsection{Convergence towards uniform r.v.}\label{Subsec:Unif}

\begin{proposition}\label{Prop:uniform}
The r.v.~$(U^{\infty}_1,\ldots,U^{\infty}_k)$ are i.i.d.~uniform over $[0,1]$ and independent of $\cQ_\infty$.
\end{proposition}

This proposition, combined with the results of Section \ref{subsec:proofeigenvalues}, show that $(\lambda_i^\infty,U_i^\infty)_{i\ge 1}$ is distributed as a Poisson point process on $\R\times[0,1]$ with intensity $e^x dx \otimes dt$. Indeed, since the intensity is a product measure, the mere fact that $(\lambda_i^\infty)_{i\ge 1}$ and $(U_i^\infty)_{i\ge 1}$ are independent and have the right distributions yields the desired property.\\

Before we proceed to the proof, we prove a simple characterization of i.i.d.~uniform r.v. Let $\sigma$ be a bijection from $[0,1]$ into itself and assume that there exists $n_0\geq 1$ such that the restriction of $\sigma$ to every interval of the form $[s_i^{n_0}, s_{i+1}^{n_0})$ is affine with slope $1$, where $s_i^{n_0}=i 2^{-n_0}$ for $i\in\{0,\ldots,2^{n_0}-1\}$. In other terms, $\sigma$ is completely characterized by its values at the points $s_i^{n_0}$ and in between we have $\sigma(x) = \sigma(s_i^{n_0}) + (x- s_i^{n_0})$, $x\in[s_i^{n_0},s_{i+1}^{n_0})$.
\begin{lemma}\label{Lemma:IID}
Let $V_1,\ldots,V_k$ be $[0,1]$-valued r.v., assumed to be almost surely distinct. Suppose that for all bijections $\sigma$ as above and for all $j_1,\ldots,j_k \in \{0,\ldots,2^{n_0}-1\}$ we have
$$ \P\big[V_1 \in [s_{j_1}^{n_0},s_{j_{1}+1}^{n_0}),\ldots, V_k \in [s_{j_k}^{n_0},s_{j_{k}+1}^{n_0})\big] = \P\big[\sigma(V_1) \in [s_{j_1}^{n_0},s_{j_{1}+1}^{n_0}),\ldots, \sigma(V_k) \in [s_{j_k}^{n_0},s_{j_{k}+1}^{n_0})\big]\;.$$
Then $(V_1,\ldots,V_k)$ is i.i.d. uniform over $[0,1]$.
\end{lemma}
\begin{proof}
Let $\pi$ be the law of $(V_1,\ldots,V_k)$. Let $D$ be the set of all points $v=(v_1,\ldots,v_k) \in [0,1]^k$ that have at least two coordinates equal. By assumption, $\pi(D) = 0$. Consider the $2^{k n_0}$ dyadic hypercubes of the type $[s_{j_1}^{n_0},s_{j_{1}+1}^{n_0})\times \ldots\times[s_{j_k}^{n_0},s_{j_{k}+1}^{n_0})$. Call $D_{n_0}$ the union of all such hypercubes that intersect $D$. We have $D_{n_0} \downarrow D$ as $n_0\rightarrow\infty$ so that $\pi(D_{n_0}) \rightarrow 0$. On the other hand, the assumption of the statement ensures that $\pi$ gives the same measure to every hypercube that does not intersect $D$. The number of all such hypercubes is $2^{n_0} (2^{n_0} -1)\ldots (2^{n_0}-k+1)$; notice that this quantity is equivalent to $2^{kn_0}$ as $n_0\to\infty$. Since $\pi([0,1]^k\backslash D_{n_0}) \to 1$, we deduce that the measure of every hypercube that does not intersect $D$ is equivalent to $2^{-kn_0}$ as $n_0\rightarrow \infty$. By a simple approximation argument, it is then easy to deduce that $\pi(A)$ equals the Lebesgue measure of $A$ for all set $A$ which is a product of intervals. As a consequence, $\pi$ is the Lebesgue measure on $[0,1]^k$ and the statement of the lemma follows.
\end{proof}

\begin{proof}[Proof of Proposition \ref{Prop:uniform}]
Observe that $\cQ_\infty$ is a Poisson point process with intensity $e^x dx$ that can be written
$$ \cQ_\infty = \sum_{i\ge 1} \delta_{\lambda_i^{\infty}}\;,$$
where $(\lambda_i^{\infty})_{i\ge 1}$ is the limit in distribution (for the product topology) of $\big((\lambda_{i}+a_L)/4\sqrt{a_L}\big)_{i\ge 1}$. Moreover, the r.v. $U_{1}^\infty,\ldots,U_k^\infty$ are all distinct a.s. Indeed, on the event $\cE(n,\eps)$ the r.v. $U_1,\ldots,U_k$ all lie at a distance at least $2^{-n}$ from one another so that the r.v. $U_{1}^\infty,\ldots,U_k^\infty$ are all distinct with probability at least $1-O(\eps)$; but since $\eps$ can be taken as small as desired, the latter property holds almost surely.\\
It suffices to show that $(U^{\infty}_1,\ldots,U^{\infty}_k)$ is i.i.d. uniform over $[0,1]$ and independent of $(\lambda^{\infty}_1,\ldots,\lambda^{\infty}_k)$. Indeed, since $k$ is arbitrary, such a result would ensure that $(U^{\infty}_1,\ldots,U^{\infty}_K)$ is independent of $(\lambda^{\infty}_1,\ldots,\lambda^{\infty}_K)$ for any $K\ge k$, so that $(U^{\infty}_1,\ldots,U^{\infty}_k)$ is independent from $(\lambda^{\infty}_i)_{i\ge 1}$ as required.\\
If we show that
\begin{equation}\label{Eq:IdUnif}
\E\Big[\prod_{i=1}^k \un_{U_{i}^\infty \in [s_{j_i}^{n_0},s_{j_{i}+1}^{n_0})} f(\lambda_1^\infty,\ldots,\lambda_k^\infty)\Big] = \E\Big[\prod_{i=1}^k \un_{\sigma(U_{i}^\infty) \in [s_{j_i}^{n_0},s_{j_{i}+1}^{n_0})}  f(\lambda_1^\infty,\ldots,\lambda_k^\infty)\Big]\;,
\end{equation}
holds for all continuous functions $f$ on $\R^k$ with compact support, all integers $j_i\in \{0,\ldots,2^{n_0}-1\}$ and all bijections $\sigma$ as above, then Lemma \ref{Lemma:IID} ensures that $(U_1^\infty,\ldots,U_k^\infty)$ is i.i.d. uniform on $[0,1]$, and classical arguments yield the independence from $(\lambda_1^\infty,\ldots,\lambda_k^\infty)$.\\

To prove \eqref{Eq:IdUnif}, we proceed as follows. Fix a bijection $\sigma$ and an integer $n_0\ge 1$ as above. Let $\tilde{\xi} := \xi \circ \sigma_L^{-1}$ where $\sigma_L: [0,L]\ni x\mapsto L\sigma(x/L)$. In other words, we let $\tilde{\xi}$ be the unique distribution on $[0,L]$ such that
$$ \langle \tilde{\xi} , f \rangle = \langle \xi, f\circ\sigma_L\rangle\;,\quad \forall f\in \cC^\infty([0,L])\;.$$
Since $\sigma$ preserves the Lebesgue measure on $[0,1]$, we deduce that $\tilde{\xi}$ has the same law as $\xi$. Therefore, we can define the corresponding operator $\tilde{\cH} = -\partial^2_x + \tilde{\xi}$ as well as the diffusions $\tilde{X}_a$ subject to the noise $\tilde{B}$. Similarly, we can introduce the r.v.~$\tilde{a}_i < \tilde{a}'_i$ which approximate the $k$ first eigenvalues of $\tilde{\cH}_L$.

\medskip

Take $n\ge n_0$. On the event $\cE(n,\eps)$, the following holds. For all $a \geq a_0$ such that $a \in M_L$, the number of explosions of $X_a$ (resp.~$\tilde{X}_a$) coincides with the sum of the number of explosions of $X_a^{j}$ (resp.~$\tilde{X}_a^{j}$), $j=0,\ldots,2^{n}-1$. Furthermore, we have the identity
$$ \tilde{X}_{a}^{j}(s_j^n L+t) = X_{a}^{2^n\sigma^{-1}(s_j^n)}\big(\sigma_L^{-1}(s_j^n L)+t\big)  \;,\quad t\in [0,2^{-n} L]\;.$$
As a consequence, the number of explosions of $X_a$ and $\tilde{X}_a$ coincide and we deduce that $a_i=\tilde{a}_i$ and $a_i'=\tilde{a}'_i$. This already ensures that the $k$ first eigenvalues of $\cH_L$ and $\tilde{\cH}_L$ lie at a distance at most $\eps/\sqrt{a_L}$ from each other on the event $\cE(n,\eps)$. Additionally, the r.v.~$\tilde{U}_{i}/L$ falls into the same subinterval of length $2^{-n}$ as $\sigma(U_{i}/L)$. Since $\cE(n,\eps)$ has a probability of order $1-O(\eps)$ for $L$ and $n$ large enough, we find:
\begin{align*}
\E\Big[\prod_{i=1}^k \un_{\tilde{U}_{i}/L \in [s^{n_0}_{j_i},s^{n_0}_{j_i+1}]}\, f(\tilde{\lambda}_{1},\ldots,\tilde{\lambda}_{k})\Big]
= \E\Big[\prod_{i=1}^k \un_{\sigma(U_{i}/L) \in [s^{n_0}_{j_i},s^{n_0}_{j_i+1}]}\, f(\lambda_{1},\ldots,\lambda_{k})\Big] + O(\eps) + \cO\Big(\frac{\eps}{\sqrt{a_L}}\Big)\;.
\end{align*}
On the other hand, since $\cH_L$ and $\tilde{\cH}_L$ have the same statistics, we immediately have
$$ \E\Big[\prod_{i=1}^k \un_{\tilde{U}_{i}/L \in [s^{n_0}_{j_i},s^{n_0}_{j_i+1}]}\, f(\tilde{\lambda}_{1},\ldots,\tilde{\lambda}_{k})\Big]  = \E\Big[\prod_{i=1}^k \un_{U_{i}/L \in [s^{n_0}_{j_i},s^{n_0}_{j_i+1}]}\, f(\lambda_{1},\ldots,\lambda_{k})\Big] \;.$$
By passing to the limit on a converging subsequence, we deduce that the identity \eqref{Eq:IdUnif} holds up to an error of order $\eps$, which can be taken as small as desired, thus concluding the proof.
\end{proof}

\section{Proofs of some first estimates on the diffusion $X_a$}\label{Sec:Diff}

\subsection{Invariant measure}

For every given $a>0$, the process $(X_a(t),t\ge 0)$ admits a unique invariant (but not reversible) probability measure $\mu_a(dx)=f_a(x)dx$ with
\begin{align}\label{density:Invmeasure}
f_a(x) = \frac2{m(a)} \exp(-2V_a(x)) \int_{-\infty}^x \exp(2V_a(y)) dy\;.
\end{align}
To check this fact, one simply has to show that $\cG^* f_a = 0$, where $\cG^*$ is the forward generator of the diffusion $X_a$
$$ \cG^* f(x) = \frac12 f''(x) + \big(V_a'(x) f(x)\big)'\;.$$
We also introduce the scale function associated with our diffusion:
\begin{align*}
S(x) = \int_{-\infty}^{x} \exp(2 V_a(y)) dy = \int_{-\infty}^{x} \exp(-2 a y + \frac{2}{3} y^3) dy \;.
\end{align*}
If we let $\P_x$ be the law of our diffusion starting from $x\in (y,z)$, then we have:
\begin{align}\label{Eq:ScaleHit}
\P_{x}[\tau_{y} < \tau_{z} \big] &= \frac{S(z) - S(x)}{S(z) - S(y)}\;.
\end{align}
In the next lemma, we establish some useful estimates on the invariant measure
\begin{lemma}\label{Lemma:InvMeas}
For all $y(a)$ that goes to $\infty$ as $a\rightarrow\infty$ but is negligible compared to $a^{3/4}$, we have the bound
$$\mu_a([\sqrt{a} - \frac{y(a)}{a^{1/4}},\sqrt{a} +\frac{y(a)}{a^{1/4}}]) \ge 1-e^{-y(a)^2}\;,$$
uniformly over all $a$ large enough. Furthermore, for all $c\in (0,1)$, there exists $\rho(c) >0$ such that
\begin{align*}
\mu_a([(1-c) \sqrt{a},(1+c) \sqrt{a}]) &\geq 1 - e^{-\rho(c) a^{3/2}}\;,
\end{align*}
uniformly over all $a$ large enough.
\end{lemma}
\begin{proof}
Note that the density of the invariant measure writes: 
\begin{align}\label{Eq:InvMeas}
\forall x \in \R, \quad f_a(x) = \frac2{m(a)} \exp(-2V(x)) S(x).
\end{align}
Some rough estimates ensure that there exists $\rho'>0$ such that for all $a$ large enough
$$ \mu_a\big(\R \backslash [-\sqrt{a}/2, 3\sqrt{a}/2]\big) \le e^{-\rho' a^{3/2}}\;.$$
Consequently, we can focus on the interval $[-\sqrt{a}/2, 3\sqrt{a}/2]$. Therein, the scale function is almost constant. Indeed, there exists $\rho'' > 0$ such that for all $x\in [-\sqrt{a}/2, 3\sqrt{a}/2]$, we have $S(x) = S(\sqrt a) \big(1+O(e^{-\rho'' a^{3/2}})\big)$. Moreover, $S(\sqrt{a}) = \sqrt{\pi/2} \;a^{-1/4} \exp(4a^{3/2}/3)(1+o(1))$. Thus, using \eqref{Eq:InvMeas} we find for all $x \in [-\sqrt{a}/2, 3\sqrt{a}/2]$,
\begin{align}\label{Eq:faExplicit}
f_a(x) = \sqrt{\frac{2}{\pi}} a^{1/4} \exp\big(-2 \sqrt{a}(x- \sqrt{a})^2 (1+ \frac{(x- \sqrt{a})}{3\, \sqrt a})\big)(1+ o(1))\;,
\end{align}
Recall that $f_a$ integrates to $1$. Therefore, to get the first bound of the statement it suffices to control the integral of $f_a$ over $[-\sqrt{a}/2, 3\sqrt{a}/2] \backslash [\sqrt{a} - \frac{y(a)}{a^{1/4}},\sqrt{a} +\frac{y(a)}{a^{1/4}}]$: this follows from a simple computation based on \eqref{Eq:faExplicit}. The second bound is obtained similarly.
\end{proof}

\subsection{Entrance and exit}
In this paragraph, we will evaluate how much time the diffusion takes to go from $+\infty$ down to various levels around $\sqrt{a}$ (resp. the time it takes to explode from various levels near $-\sqrt{a}$). Recall the definition of $t_a$ and $h_a$ given in \eqref{Eq:ta}. 

\begin{lemma}[Entrance]\label{Lemma:CDI}
Take $x(a) = \sqrt a + h_a/4$ and $M=2\sqrt t_a \ln\ln a$. On the event $\{\sup_{t\in[0,t_a]} |B(t)| < M \}$, whose probability is at least $1-\exp(-(\ln\ln a)^2)$, we have
\begin{align}
\sup_{t\in (0,\tau_{x(a)}(X)]} \big|X(t)-\sqrt{a} \coth (\sqrt a t)\big| \lesssim M \;. \label{Eq:Entrance}
\end{align}
As a consequence, we get the following asymptotics for all $u>1$ and all $c\ge 1/4$:
\begin{align}
\tau_{\sqrt a + c\,h_a}(X) &= \frac{3}{8}\, \frac{\ln a}{\sqrt{a}} + \frac1{2\sqrt{a}}\ln\frac{2}{c\ln a} + o\Big(\frac{1}{\sqrt{a}}\Big)\;, \label{Eq:tCDI}\\
\tau_{u \sqrt a}(X) &= \frac{C(u)}{\sqrt{a}} + o\Big(\frac{1}{\sqrt{a}}\Big)\;, \label{Eq:tCDI2}
\end{align}
where $u\mapsto C(u)$ is the reciprocal of $x \mapsto \coth x$.\\
Furthermore, we have the additional bound for all $a$ large enough
\begin{equation}\label{Eq:Entrance2}
\sup_{t\in (\tau_{x(a)}(X),\frac38 t_a]} \big|X(t)-\sqrt{a} \coth (\sqrt a t)\big| \le 1 \;.
\end{equation}
\end{lemma}
\begin{proof}
We adapt the arguments of the proof of~\cite[Proposition 2]{DumazVirag}. Let $Z(t) = X(t) - B(t)$. Necessarily, $Z$ solves the ODE
\begin{equation*}
dZ(t) = \Big(a-Z(t)^2 \Big(1+\frac{B(t)}{Z(t)}\Big)^2\Big)dt\;,\qquad Z(0)=+\infty\;.
\end{equation*}
We are therefore led to considering the ODE
\begin{equation*}
dF(t) = \Big(a-C F(t)^2\Big)dt\;,\qquad F(0)=+\infty\;,
\end{equation*}
with $C$ being either $C_1$ or $C_2$ where
$$ C_1 = \Big(1 - \frac{M}{\ell(a) - M}\Big)^2\;,\qquad C_2 = \Big(1 + \frac{M}{\ell(a) - M}\Big)^2 \;,$$
and $\ell(a)$ is a value in $[x(a),\infty)$ that will be chosen later on. The generic solution is given by $F(t) = \sqrt{a/C} \coth(\sqrt{a \,C} \,t)$ for all $t>0$. On the event $\{\sup_{t\in[0,t_a]} |B(t)| < M\}$ and for $a$ large enough, we have the following bound
\begin{equation*}
F_2(t) \leq Z(t) \leq F_1(t)\;,\qquad t\in[0,\tau_{\ell(a)}(X)\wedge t_a]\;. 
\end{equation*}
Consequently the hitting time of $\ell(a)$ by the diffusion $X$ satisfies
$$ \tau_{\ell(a) + M}(F_2)\wedge t_a \leq \tau_{\ell(a)}(X)\wedge t_a \leq \tau_{\ell(a) - M}(F_1)\wedge t_a \;.$$
We now derive the asymptotics of these upper and lower bounds. Let $\ell_\pm(a)$ be $\ell(a) \pm M$.

When $\ell(a)=u \sqrt a$, we have:
$$ \tau_{\ell_{\pm}(a)}(F) = \frac{C(u)}{\sqrt{a}}\Big(1 + \cO\Big(\frac{M}{\ell_{\pm}(a)}\Big)\Big)\;.$$
On the other hand, when $\ell(a) = \sqrt a + c\frac{\ln a}{a^{1/4}}$, the asymptotic expansion of $\coth$ at infinity readily yields
$$ \tau_{\ell_{\pm}(a)}(F) = -\frac{1}{2\sqrt{aC}} \ln\frac12\Big(\frac{\ell_{\pm}(a)}{\sqrt{a}}\sqrt{C}-1+o(e^{-2\sqrt{aC}\tau_{\ell_{\pm}(a)}(F)})\Big)\;,$$
uniformly over all $C$ in a neighborhood of $1$. Since $M \ll (\ell(a) - \sqrt a)$,
a simple calculation shows that both $\tau_{\ell(a) + M}(F_2)$ and $ \tau_{\ell(a) - M}(F_1)$ admit the following expansion
$$\frac{3}{8}\, \frac{\ln a}{\sqrt{a}} + \frac1{2\sqrt{a}}\ln\frac{2}{c\ln a} + o\Big(\frac1{\sqrt{a}}\Big)\;,$$
as $a\rightarrow\infty$.\\
We have proven the asserted asymptotics on the hitting times: they ensure that $\tau_{\ell(a)}(X) < t_a$ for $a$ large enough, so that
\begin{equation*}
F_2(t)-M \leq X(t) \leq F_1(t)+M\;,\qquad t\in[0,\tau_{\ell(a)}(X)]\;. 
\end{equation*}
for all $\ell(a) \in [x(a),\infty)$. It is simple to check that $F_2(t)-M \le \sqrt{a} \coth(\sqrt{a} t) \le F_1(t)+M$ for all $t\ge 0$. Hence, for all $\ell(a) \in [x(a),\infty)$
\begin{align*}
\sup_{t\in [\tau_{2\ell(a)},\tau_{\ell(a)}]}\big|X(t) - \sqrt{a} \coth(\sqrt{a} t)\big| &\le \sup_{t\in[\tau_{2\ell(a)+M}(F_2),\tau_{\ell(a)-M}(F_1)]} \big|(F_1(t)+M) - (F_2(t) - M)\big|\\
&\le \sup_{t\in[\tau_{2\ell(a)+M}(F_2),\tau_{\ell(a)-M}(F_1)]} \Big|\frac{F_1(t)}{F_2(t)}-1\Big| F_2(t) + 2M\;.
\end{align*}
From the explicit expressions of $F_1$ and $F_2$, we obtain:
$$ \sup_{t>0}\Big|\frac{F_1(t)}{F_2(t)} - 1\Big| \lesssim \frac{M}{\ell(a)}\;,$$
uniformly over all choices of $\ell(a) \in [x(a),\infty)$. Therefore,
$$ \sup_{t\in [\tau_{2\ell(a)+M}(F_2),\tau_{\ell(a)-M}(F_1)]}\Big|\frac{F_1(t)}{F_2(t)} - 1\Big|F_2(t) \lesssim M\;,$$
so that
$$ \sup_{t\in [\tau_{2\ell(a)},\tau_{\ell(a)}]}\big|X(t) - \sqrt{a} \coth(\sqrt{a} t)\big| \lesssim M\;,$$
uniformly over all choices of $\ell(a) \in [x(a),\infty)$. Patching together these estimates, we get \eqref{Eq:Entrance}.\\
To complete the proof of the lemma, it remains to control the diffusion on the interval $[\tau_{x(a)}(X),(3/8)t_a]$ and on the event $\{\sup_{t\in[0,t_a]} |B(t)| < M\}$. To that end, we take $\ell(a) = \sqrt{a}/2$ and, by the arguments at the beginning of the proof, we find
$$ F_2(t)-M \le X(t) \le F_1(t)+M\;,\qquad t\in[0,\tau_{\ell(a)}(X)\wedge t_a]\;.$$
Since $F_2((3/8)t_a)-M = \sqrt a - O(M)$ and since $F_2$ is decreasing, we easily deduce that $\tau_{\ell(a)+M}(F_2) > (3/8)t_a$ and therefore $\tau_{\ell(a)}(X)> t_a$. This yields \eqref{Eq:Entrance2}.
\end{proof}

We have an analogous result right before the explosion time, we keep the notations from the previous lemma.
\begin{lemma}[Explosion]\label{Lemma:Explo}
On the event $\{\sup_{t\in[0,t_a]} \big|B(\tau_{-x(a)}(X)+t)-B(\tau_{-x(a)}(X))\big| < M\}$, whose probability is at least $1-\exp(-(\ln\ln a)^2)$, we have:
$$ \sup_{t\in [\tau_{-x(a)}(X),\tau_{-\infty}(X))} \Big|X(t)+\sqrt{a} \coth (\sqrt a (\tau_{-\infty}(X)-t))\Big| \lesssim M \;.$$
Therefore, for $u> 1$ and $c\ge 1/4$, as $a\to\infty$
\begin{align*}
\tau_{-\infty}(X) - \tau_{-\sqrt a - c\, h_a}(X) &= \frac{3}{8}\, \frac{\ln a}{\sqrt{a}} + \frac1{2\sqrt{a}}\ln\frac{2}{c\ln a} + o\Big(\frac{1}{\sqrt{a}}\Big)\;, \\
\tau_{-\infty}(X) - \tau_{-u\sqrt a}(X) &= \frac{C(u)}{\sqrt{a}} + o\Big(\frac{1}{\sqrt{a}}\Big)\;.
\end{align*}
\end{lemma}
\begin{proof}
The proof is essentially the same as that of Lemma \ref{Lemma:CDI} so we do not provide the details.
\end{proof}


\begin{proof}[Proof of Proposition \ref{Prop:Da}]
By Lemmas \ref{Lemma:CDI} and \ref{Lemma:Explo}, the trajectory $(X(t), t\in [0,\zeta(1)))$ satisfies the two requirements of the event $\cD^N_a$ with a probability at least $1-\exp(-(\ln\ln a)^2)$ uniformly over all $a$ large enough. Since the random paths $(X(t+\zeta(k)),\;t \in [0,\zeta(k+1)-\zeta(k)))$, $k\in \{0,\cdots,N\}$ are i.i.d, we easily deduce the statement of Proposition \ref{Prop:Da}.
\end{proof}

\subsection{Oscillations}

The goal of this subsection is to prove Proposition \ref{Propo:DiffClose}. From now on, $Y$ is taken to be a solution of \eqref{Eq:Diff} starting from the stationary measure $\mu_a$. The key step consists in showing that $Y$ spends most of its time near $\sqrt{a}$. 
We introduce the notations:
\begin{align*}
&\tau_{\le\ell}(f) := \inf\{t\geq 0: f(t) \le \ell\}\;,\quad \tau_{\ge\ell}(f) := \inf\{t\geq 0: f(t) \ge \ell\}\;.\\
&\tau^k_{-\infty}(f) := \inf\{ t \geq \tau^{k-1}_{-\infty}(f) : f(t) = -\infty\}\;, \quad \tau_{-\infty}^0(f) := 0.
\end{align*}
For $r >0$, we define the following interval around the bottom of the well $\sqrt{a}$:
\begin{align*}
I_a(r) := \big[\sqrt{a} - r h_a, \sqrt{a} + r h_a\big].
\end{align*}
We start with a simple estimate on the exit times of $X$.
\begin{lemma}\label{Lemma:RBM}
For any $T>0$ and any $0 < d < D < 2\sqrt a$ we have
$$ \P(\exists t\in [0,T], X(t) \notin [\sqrt a - D,\sqrt a + D] \,|\, X(0) \in [\sqrt a - d,\sqrt a + d]) \le \frac{8\sqrt T}{D-d} e^{-\frac{(D-d)^2}{2T}}\;.$$
\end{lemma}
\begin{proof}
We let $A$ be the reflected Brownian motion starting from $d$, obtained as the solution of the following Skorohod problem
$$ dA(t) = dB(t) + d\ell(t)\;, \quad \int_{t\geq 0} A(t) d\ell(t) =0\;,\quad A(0)=d\;.$$
For $x\in [\sqrt a-d,\sqrt a+d]$, consider the diffusion $X$ starting from $x$. We claim that $A(t) - (X(t) - \sqrt a) \ge 0$ for all $t\ge 0$. Indeed, for any time $t\ge 0$ at which $A(t)=X(t) - \sqrt a$, we have $d(A(t)-(X(t)-\sqrt{a})) > 0$ since either $A(t)=X(t)-\sqrt{a}=0$ and then $d\ell_t >0$ and $-V'(X(t))=0$, or $A(t)=X(t) - \sqrt{a} > 0$ and then $d\ell(t) = 0$ and $-V'(X(t)) < 0$. Since $A(t)$ has the same law as $|B(t)+d|$ and since the supremum (resp. the infimum) of $B$ on $[0,T]$ has the same law as $|B(T)|$ (resp. $-|B(T)|$) we obtain
$$ \P(\sup_{t\in [0,T]} X(t) > \sqrt a + D | X(0) =x ) \le \P( \sup_{t\in [0,T]} A(t) > D) \le 2 \P(|B(T)| > D-d) \le \frac{4\sqrt T}{D-d} e^{-\frac{(D-d)^2}{2T}} \;.$$
The same argument applies to bound $X$ from below, thus concluding the proof.
\end{proof}

\begin{lemma}[Oscillations of $Y$]\label{Lemma:ControlY}
There exists $c>0$ such that the probability that 
\begin{align*}
\fint_{0}^t Y(s) ds \in I_a(1/2)\;,\quad \forall t\in [0,\tau_{\le -3\sqrt{a}}]\;,
\end{align*}
is larger than $1-\exp(-c (\ln a)^2)$ uniformly over all $a$ large enough.
\end{lemma}


\begin{remark}
We could improve the bounds by choosing an interval of size $C\sqrt{\ln a}/a^{1/4}$ around $\sqrt{a}$ at the cost of decreasing the lower bound on the probability.
\end{remark}

\begin{proof}
Let $L_a = \exp(b\, a^{3/2})$ for some large enough $b>0$. We first control the integral of the statement for all $t\in [1/m(a),\tau_{\le -3 \sqrt{a}}(Y) \wedge \tau_{\ge 4\sqrt{a}}(Y) \wedge L_a]$. For any such $t$, we have:
\begin{align*}
\int_0^t Y(s) ds = \int_0^t Y(s) \un_{\{Y(s) \in I_a(1/4)\}}ds + \int_0^t Y(s) \un_{\{Y(s) \notin I_a(1/4)\}} ds\;.
\end{align*}
For $k\in \bbN$, we set
$$E_k := \Big\{\fint_0^{2^{-k}\, L_a} \un_{\{Y(s) \notin I_a(1/4)\}} ds \geq \frac1{a}\Big\}\;.$$
Using the Markov inequality, we find $\P(E_k) \le a\, \mu_a(I_a(1/4)^c)$. Let $k^*$ be the smallest integer such that $2^{-k^*}L_a <1/m(a)$. Observe that $k^*$ is of order $a^{3/2}$ thanks to our hypothesis on the growth of $L_a$ so that thanks to Lemma \ref{Lemma:InvMeas},
$$ \P\big[\cup_{k=0}^{k^*} E_k\big] \le \sum_{k\le k^*} \P[E_k] \lesssim a^{5/2} e^{-(\ln a)^2/16}\;.$$
On the event $E_k^c$, we get for all $t \in [2^{-k-1} L_a, 2^{-k} L_a]$ such that $t \leq \tau_{\le-3 \sqrt{a}}(Y) \wedge \tau_{\ge 4 \sqrt{a}}(Y)$, and all $a$ large enough, the following upper bound
\begin{align*}
\int_0^t Y(s) ds &\leq \Big(\sqrt{a} +h_a /4\Big) t + 4 \sqrt{a} \int_0^t 1_{\{Y(s) \notin I_a(1/4)\}}ds\\
&\leq \Big(\sqrt{a} + h_a/2\Big) t\;,
\end{align*}
and the following lower bound
\begin{align*}
\int_0^t Y(s) ds 
&\geq \Big(\sqrt{a} - h_a/2\Big) \,t\;.
\end{align*}
To conclude the proof, we show that with large probability $Y$ does not hit $[4\sqrt{a},\infty)$ before $(-\infty,-3\sqrt{a}]$, nor exit $I_a(1/4)$ within a time $1/m(a)$, and that $Y$ hits $(-\infty,-3\sqrt{a}]$ before time $L_a$. Regarding the first claim, we have for all $a>0$ large enough
\begin{align*}
\P[\tau_{\ge 4\sqrt{a}}(Y) < \tau_{\le -3\sqrt{a}}(Y)] &= \int_{x\in\R} \P_{x}[\tau_{4\sqrt{a}} < \tau_{-3 \sqrt{a}}] \mu_a(dx)\\
&\le \mu_a([-2 \sqrt{a}, 2\sqrt{a}]) \;\P_{2 \sqrt{a}}[\tau_{4\sqrt{a}} < \tau_{-3 \sqrt{a}}] + O(e^{-\rho' a^{3/2}})\\
&\le e^{-\rho a^{3/2}},
\end{align*}
for some positive $\rho$, using \eqref{Eq:ScaleHit} and some simple estimates on the scale function. The second claim follows from Lemma \ref{Lemma:RBM} combined with the bounds of Lemma \ref{Lemma:InvMeas} on the invariant measure. The last claim follows from $\P[ \tau_{\le -3\sqrt{a}}(Y) > L_a] \leq \P_{+\infty}[ \tau_{-\infty}(X)/m(a) > L_a/m(a)]$ and Markov's inequality together with \eqref{asymp_ma}.
\end{proof}

The following lemma shows that, after $X$ has come down from infinity, $X$ and $Y$ stay very close to each other until $Y$ starts following its deterministic path to $-\infty$. 
\begin{lemma}\label{Lemma:DiffXY}
There exist $b,C>0$ such that with probability at least $1-e^{-b(\ln \ln a)^2}$, we have:\begin{enumerate}[label=(\roman*)]
\item $\tau_{\le -3\sqrt{a}}(Y) > (3/8) t_a$ and $\tau_{\le -2\sqrt{a}}(X) > (3/8) t_a$
\item $|X(t)-Y(t)| \le h_a/2$ for all $t \in \big[(3/8) t_a, \tau_{\le -3\sqrt{a}}(Y)\big]$,
\item $\big|\tau_{-\infty}(X) -\tau_{-\infty}(Y)\big| < \frac{C}{\sqrt{a}}$
\item $\fint_{(3/8)t_a}^t X(s) ds \in I_a(1)$ for all $t \in [(3/8)t_a, \tau_{-2\sqrt{a}}(X)]$.
\end{enumerate}
\end{lemma}

\begin{proof}
By Lemma \ref{Lemma:RBM} and Lemma \ref{Lemma:InvMeas}, it is simple to establish the first bound of (i) with $-3\sqrt a$ replaced by $-2\sqrt a$. Monotonicity then easily yields the second bound of (i). Recall that $X(t) \ge Y(t)$ for all $t\in [0,\tau_{-\infty}(Y)) \supset [0,\tau_{\le-3\sqrt{a}}(Y))$. The difference $Z=X-Y$ satisfies the ODE:
$$ dZ(t) = -Z(t) (X(t)+Y(t))\;,\quad t\ge 0\;.$$
By Lemma \ref{Lemma:CDI}, there exists $z >0$ such that if we set $t_1 := (3/8) t_a + \ln(z/\ln a)/(2 \sqrt{a})$ then with probability greater than $1- \exp(- (\ln\ln a)^2)$, we have $X(t_1) \in I_a(1/3)$. By the first estimate on the invariant measure in Lemma \ref{Lemma:InvMeas}, we deduce that there exists $c' >0$ such that for all $a$ large enough
$$ \bbP\big[Z(t_1) \leq h_a/2\big] \geq \P[X(t_1) \in I_a(1/3), \;Y(t_1) \in I_a(1/6)] 
\geq 1- e^{-c' (\ln\ln a)^2}\;.$$
By Lemma \ref{Lemma:ControlY} applied from time $t_1$ (recall that $Y$ is stationary) we have with probability at least $1-e^{-c(\ln a)^2}$
$$\fint_{t_1}^t Y(s) ds \in I_a(1/2)\;,\quad \forall t\in [t_1,\tau'_{\le -3\sqrt{a}}(Y)]\;,$$
where $\tau'_{\le -3\sqrt{a}}(Y) := \inf\{t\ge t_1: Y(t) \le -3\sqrt a\}$. By (i), we can assume that $\tau'_{\le -3\sqrt{a}}(Y) = \tau_{\le -3\sqrt{a}}(Y)$. Henceforth, there exists $b'>0$ such that the probability that for all $t\in [t_1,\tau_{\le-3\sqrt{a}}(Y)]$
$$ Z(t) = Z(t_1) \exp(-\int_{t_1}^t (X(s)+Y(s)) ds) \leq Z(t_1) \exp(-2\int_{t_1}^t Y(s) ds) \le \frac{h_a}{2}e^{-2(\sqrt{a}-h_a/2)(t- t_1)}\;,$$
is larger than $1- \exp(- b'(\ln\ln a)^2)$ for all $a$ large enough. This proves (ii). Applying Lemma \ref{Lemma:Explo}, we get (iii). Applying Lemma \ref{Lemma:ControlY} from time $(3/8)t_a$ and using (ii), we easily get (iv).
\end{proof}
%
%

More can be said about $X-Y$. Since $Y(\tau_{-\infty}(Y)+t), t\ge 0$ has the same distribution as $X$, we deduce that it satisfies the estimates \eqref{Eq:tCDI} and (iv) of Lemma \ref{Lemma:DiffXY}. Since with large probability $\tau_{-\infty}^1(X)-\tau_{-\infty}^1(Y)$ is at most $C/\sqrt{a}$, a simple iteration of the proof of Lemma \ref{Lemma:DiffXY} allows to prove that with probability at least $1-2e^{-b(\ln\ln a)^2}$, we have
$\tau^2_{-3\sqrt{a}}(Y) > \tau_{-\infty}(X)+(3/8) t_a$, $\tau^2_{-2\sqrt{a}}(X) > \tau_{-\infty}(X)+(3/8) t_a$ as well as
$$ |X(t)-Y(t)| \le h_a/2\;,\quad \forall t\in [\tau_{-\infty}(X)+(3/8) t_a, \tau^2_{-3\sqrt{a}}(Y)]\;,$$
and $\tau^2_{-\infty}(X)-\tau^2_{-\infty}(Y) \le C/\sqrt{a}$. Iterating this, we get the following result.
\begin{corollary}\label{Cor:XY}
There exist $b,C>0$ such that for any $N_a \ge 1$ with probability at least $1-N_a e^{-b (\ln \ln a)^2}$ the following holds for all $a$ large enough. For all $k\in \{1,\ldots,N_a\}$ we have 
\begin{align*}
&\tau^k_{-3\sqrt{a}}(Y) > \tau_{-\infty}^{k-1}(X)+\frac{3}{8} t_a\;, \quad \tau^k_{-2\sqrt{a}}(X) > \tau_{-\infty}^{k-1}(X)+(3/8) t_a\;,\\
&\big| X(t) - Y(t) \big| \le h_a/2 \;,\quad \mbox{for all } t\in [\tau_{-\infty}^{k-1}(X)+\frac{3}{8} t_a, \tau^k_{-3\sqrt{a}}(Y)]\;,\\
&\big| \tau_{-\infty}^k(X) - \tau_{-\infty}^k(Y) \big| < \frac{C}{\sqrt{a}}\;.
\end{align*}
where $\tau^k_{-3\sqrt{a}}(Y) := \inf\{t\ge \tau^{k-1}_{-\infty}(Y): Y(t) = -3\sqrt{a}\}$.
\end{corollary}

We conclude this section with the proof of Proposition \ref{Propo:DiffClose}.

\begin{proof}[Proof of Proposition \ref{Propo:DiffClose}] 
Using Corollary \ref{Cor:XY}, with probability greater than $1-e^{-b(\ln\ln a)^2}$, we have:
\begin{equation}\label{Eq:zetatau}\begin{split}
&\tau^k_{-3\sqrt{a}}(Y) > \tau_{-\infty}^{k-1}(X)+\frac{3}{8} t_a\;, \quad \tau^k_{-2\sqrt{a}}(X) > \tau_{-\infty}^{k-1}(X)+(3/8) t_a\;,\\
&\big|X(t) - Y(t)\big| \le h_a/2 \;,\quad \forall t\in \big[\zeta_a(k-1)+ \frac{3}{8} t_a,\tau^k_{-3\sqrt{a}}(Y)\big]\;,\\
&\big|\zeta_a(k) - \tau_{-\infty}^k(Y) \big| < \frac{C}{\sqrt{a}}\;.\end{split}
\end{equation}

Let $Z(t) = X^{t_0}(t) - Y(t)$, $\tau := \inf\{t\ge t_0: Y(t) \le -3\sqrt{a}\}$ and $\tau' := \inf\{t\ge t_0: Y(t) = -\infty\}$ then the same arguments as in the proof of Lemma \ref{Lemma:DiffXY} ensure that with probability at least $1- e^{-b(\ln\ln a)^2}$ we have $\tau > t_0 +(3/8)t_a$ as well as
\begin{align*}
&Z(t) \le h_a/2\;,\quad \forall t\in [t_0+(3/8) t_a,\tau]\;,\\
&\big| \tau_{-\infty}(X^{t_0}) - \tau'\big| < \frac{C}{\sqrt{a}}\;.
\end{align*}

By Lemma \ref{Lemma:RBM}, $Y(t) > 0$ for all $t\in [(t_0-(3/8)t_a)_+,t_0+(3/8)t_a]$ with large probability: we can therefore work on this event from now on. If $t_0 \in [\zeta_a(k-1),\zeta_a(k))$, then by \eqref{Eq:zetatau} and the positivity of $Y$ on $[(t_0-(3/8)t_a)_+,t_0+(3/8)t_a]$ we have $t_0+(3/8)t_a < \tau^k_{-\infty}(Y)$. Therefore $\tau'=\tau^k_{-\infty}(Y)$ and $\tau = \tau^k_{-3\sqrt a}(Y)$. Putting everything together, we get the statement of the proposition.
\end{proof}

\section{Crossing of the well}\label{Sec:Crossing}

This section is devoted to the proof of Propositions \ref{Prop:BoundsCross1}, \ref{Prop:BoundsCross2}, \ref{Prop:BoundsCross3} and of part (b) of Proposition \ref{Propo:cE}. This is achieved through a series of technical lemmas. The proof of Proposition \ref{Prop:BoundsCross1} (i) and (ii) can be found at the end of Subsection \ref{subsec:Boundsuptotheta}. The proof of Proposition \ref{Prop:BoundsCross2} is in Subsection \ref{subsec:coupling}.  The proofs of Proposition \ref{Prop:BoundsCross1} (iii) and Proposition \ref{Prop:BoundsCross3} can be found in Subsection \ref{subsec:boundsaftertheta}. Finally, the proof of part (b) of Proposition \ref{Propo:cE} is presented in subsections \ref{subsec:proofexcursions1} and \ref{subsec:proofexcursions2}.

\subsection{Bounds up to time $\theta_a$}\label{subsec:Boundsuptotheta}

In the sequel, we take the following parameters:
$$ T = \frac34 t_a\;,\quad \delta = \frac{(\ln a)^2}{a^{1/4}},$$
and we study the diffusion $X$ when it crosses the region $[-\sqrt{a},\sqrt{a}]$.

\smallskip

Let us introduce the diffusion:
$$ dH(t) = (-a + H^2(t))dt + dB(t)\;,$$
whose drift is the opposite of the drift of $X$. This diffusion turns out to be a good approximation of the diffusion $X$ starting below $\sqrt{a}$ and conditioned to hit $-\sqrt{a}$ before $\sqrt{a}$ (see also \cite{DumazVirag} where similar techniques were used). This can be stated precisely thanks to Girsanov's theorem.

We will denote by $\bP_x$ the law of this diffusion starting from $x$, while $\P_x$ will denote the law of $X_a$ starting from $x$.

\begin{lemma}[Diffusion from $\sqrt{a}-\delta$ to $-\sqrt{a}+\delta$]\label{Lemma:Girsanov}
Recall that $\tau_0$ denotes the first hitting time of $0$ by $X$. 
For any $c>0$, there exists $C >0$ such that for all $a$ large enough, we have:
$$ \P_{\sqrt{a} - \delta}\big[E(C) \; \big| \; \tau_{-\sqrt{a} + \delta} < \tau_{\sqrt{a} - \delta/2} \wedge T\big] \leq a^{-c}\;,$$
where $E(C)$ is the following event
\begin{align*}
E(C) &:= \Big\{\sup_{t \in [0,\tau_{-\sqrt{a}+\delta}]} \big|X(t)- \sqrt a \tanh(-\sqrt a(t-\tau_0))\big| \geq C \frac{\sqrt a}{\ln a}\Big\}\\
&\quad\bigcup \Big\{\big|\tau_0 - \frac38 t_a\big| \ge C \frac{\ln \ln a}{\sqrt a} \Big\}\bigcup \Big\{\big|\tau_{-\sqrt{a} + \delta} - \tau_0 - \frac38 t_a\big| \ge C \frac{\ln \ln a}{\sqrt a} \Big\}\;.
\end{align*}
\end{lemma}
Notice that, eventually, the stopping time $\tau_0$ will correspond (up to a negligible error) to the r.v. $\upsilon_a$ in the context of Proposition \ref{Prop:BoundsCross1}.
\begin{proof}
To simplify the notations, we set
$$ \tau_+ := \tau_{\sqrt{a}-\delta/2}\;,\quad \tau_- := \tau_{-\sqrt{a}+\delta}\;.$$
If we stop the diffusions at time $T\wedge \tau_-\wedge \tau_+$, then Girsanov's Theorem~\cite[Th.VIII.1.7]{RevuzYor} ensures that the Radon-Nikodym derivative of the law of $X_a$ w.r.t.~the law of $H$ up to time $t$ is given by the exponential of $G_t(H) := 2 \int_0^t (a - H(s)^2) dH(s)$. Applying It\^o's formula to $V(H)$, a simple calculation yields:
\begin{align*}
G_t(H) = 2 \big(V_a(H(0)) - V_a(H(t))\big) + 2 \int_0^t H(s) ds\;.
\end{align*}
We thus get:
\begin{align}
\frac{\P_{\sqrt{a} - \delta}\big[ E \; ;\; \tau_{-} < \tau_{+} \wedge T\big]}{\P_{\sqrt{a} - \delta}\big[\tau_{-} < \tau_{+} \wedge T\big]} =\,& \frac{\bE_{\sqrt{a} - \delta}\big[E \; ;\; \tau_{-} < \tau_{+} \wedge T \;;\; \exp(G_{\tau_-}(H))\big]}{\bE_{\sqrt{a} - \delta}\big[\tau_- < \tau_+ \wedge T\;;\;\exp(G_{\tau_-}(H))]} \notag\\
=\,& \frac{\bE_{\sqrt{a} - \delta}\big[E\;;\; \tau_- < \tau_+ \wedge T\;; \;\exp(2 \int_0^{\tau_-} H(t) dt)\big]}{\bE_{\sqrt{a} - \delta}\big[\tau_- < \tau_+ \wedge T\;;\;\exp(2 \int_0^{\tau_-} H(t) dt)\big]} \notag\\
\leq\,& \exp(4\sqrt{a} \,T) \;\frac{\bP_{\sqrt{a} - \delta}\big[E\;;\; \tau_- < \tau_+ \wedge T\big]}{\bP_{\sqrt{a} - \delta}\big[\tau_- < \tau_+ \wedge T\big]}\;. \label{upboundHgirsanov}
\end{align}
We now study $H$ in the region $[-\sqrt{a} + \delta, \sqrt{a}-\delta/2]$. Notice that $Z := H - B$ satisfies
\begin{align*}
dZ(t) = (-a + (Z(t) + B(t))^2) dt\;.
\end{align*}
Set $M := c_0 \ln a / a^{1/4}$ and assume that $\sup_{t \in [0,T]}|B(t)| \leq M$ holds true. Then, if $|H(t)| \in [\sqrt{a}/2,\sqrt{a}]$ we have
\begin{align}
-a + (Z(t) + B(t))^2 &\leq -a + Z(t)^2\Big(1 + \frac{M}{\sqrt{a}/2 - M}\Big)^2 \notag \\
&\leq  -a + Z(t)^2\Big(1 + \frac{4 M}{\sqrt{a}}\Big)^2\;, \label{ineqEDSbord}
\end{align}
for $a$ large enough. On the other hand, if $|H(t)| \in [0,\sqrt{a}/2]$, we have
\begin{align}
-a + (Z(t) + B(t))^2 
&\leq (-a + Z(t)^2)(1 - \frac{4 M}{ \sqrt{a}}).  \label{ineqEDScentre}
\end{align}
Therefore, whenever $\sup_{t \in [0,T]}|B(t)| \leq M$, and $|H(t)| \leq \sqrt{a}$, we get thanks to \eqref{ineqEDSbord} and \eqref{ineqEDScentre}:
\begin{align*}
-a + (Z(t) + B(t))^2 &\leq -a (1 - \frac{4 M}{ \sqrt{a}}) + Z(t)^2 (1 + \frac{4 M}{\sqrt{a}})^2\;.
\end{align*}
Similarly, we have the lower bound:
\begin{align*}
-a + (Z(t) + B(t))^2 &\geq -a (1 + \frac{4 M}{ \sqrt{a}}) + Z(t)^2 (1 - \frac{4 M}{\sqrt{a}})^2\;.
\end{align*}
We deduce that when $\sup_{t \in [0,T]}|B(t)| \leq M$ and as long as $|H(t)| \leq \sqrt{a}$
we have $Z_-(t) \leq Z(t) \leq Z_+(t)$ with
\begin{align*}
dZ_+(t) &= \big(-a (1 - \frac{4 M}{ \sqrt{a}}) + Z_+(t)^2(1+ \frac{4 \, M}{\sqrt{a}})^2\big) dt\;,\quad Z_+(0) = \sqrt a - \delta\;,\\
dZ_-(t) &= \big(-a (1 + \frac{4 M}{ \sqrt{a}}) + Z_-(t)^2(1- \frac{4 \, M}{\sqrt{a}})^2\big) dt\;,\quad Z_-(0) = \sqrt a - \delta\;.
\end{align*}
Denote by $\kappa = 4 M/\sqrt{a}$. Those last equations have explicit solutions given by 
\begin{align*}
Z_+(t) &= - \frac{\sqrt{a} (1 - \kappa)^{1/2}}{1+ \kappa} \tanh(\sqrt{a} (1 - \kappa)^{1/2}(1+ \kappa)(t - A_+))\;,\\
Z_-(t) &= - \frac{\sqrt{a} (1 + \kappa)^{1/2}}{1- \kappa} \tanh(\sqrt{a} (1 + \kappa)^{1/2}(1-\kappa)(t - A_-))\;,
\end{align*}
where
$$ A_+ = \frac1{2\sqrt{a}} \Big( \ln \frac{2\sqrt a}{\delta} + \frac{3\sqrt a}{2\delta} \kappa + o\Big(\frac1{\ln a}\Big)\Big) \;,\quad A_- = \frac1{2\sqrt{a}} \Big( \ln \frac{2\sqrt a}{\delta} - \frac{3\sqrt a}{2\delta} \kappa + o\Big(\frac1{\ln a}\Big)\Big)\;.$$
Note that
\begin{align*}
A_\pm = \frac{3}{8} \frac{\ln a}{\sqrt{a}} - \frac{\ln\ln a}{\sqrt{a}} + \frac{\ln(2)}{2\sqrt{a}} \pm \frac{3 c_0}{\sqrt{a} \ln a}(1 + o(1)).
\end{align*}

From there, simple calculations show that $Z_+(\cdot) + M$ stays below $\sqrt a - \delta/2$ and passes below $-\sqrt{a}+\delta$ before time $T$, while $Z_-(\cdot) - M$ is above $-\sqrt{a}+\delta$ at time $T -  3\ln \ln a / \sqrt a$. Furthermore, $Z_+(\cdot)+ M$ and $Z_-(\cdot)-M$ always stay at a distance of order $\sqrt{a}/\ln a$ from one another. Their respective crossing times of $0$ lie at $(3/8) t_a + O(\ln\ln a / \sqrt a)$ and are at a distance of order $1/(\ln a \,\sqrt a)$ from one another. Recall that $Z_-(t)-M \le H(t) \le Z_+(t)+M$ for $t \in [0,T]$ when $\sup_{t \in [0,T]}|B(t)| \leq M$ and $\sup_{t \in [0,T]} |H(t)| \leq \sqrt{a}$ so that all the hitting times of $0$ by $H$ are located near the crossing times of $0$ of the latter two curves. Hence, there exists $C>0$ such that on the event $\sup_{t \in [0,T]}|B(t)| \leq M$ we have
\begin{align*}
&\sup_{t \in [0,\tau_{-\sqrt{a}+\delta}]} \big|H(t) - \sqrt a \tanh(-\sqrt a(t-\tau_0))\big| < C \frac{\sqrt a}{\ln a} \;,\;\\
&\big|\tau_0 - \frac38 t_a\big| < C\frac{\ln \ln a}{\sqrt a} \;,\; \big|\tau_{-\sqrt{a} + \delta} - \tau_0 - \frac38 t_a\big| < C\frac{\ln \ln a}{\sqrt a}\;.
\end{align*}
Since for $a$ large enough we have $\bP(\sup_{t \in [0,T]} |B(t)| \geq M) \lesssim \exp(- M^2/(2T)) = a^{- \frac{2}{3} c_0^2}$, we get
\begin{align*}
\eqref{upboundHgirsanov} &\le \exp(4 \sqrt{a}\, T)\, \frac{\bP[\sup_{t \in [0,T]} |B(t)| \geq M]}{\bP[\sup_{t \in [0,T]} |B(t)| \leq M]}\lesssim a^{3-\frac23 c_0^2}\;, 
\end{align*}
for all $a$ large enough. This concludes the proof.
\end{proof}

\begin{lemma}[Time needed to go from $\sqrt{a}-\delta$ to $-\sqrt{a}+\delta$]\label{Lemma:Girsanov2}
For all $r>0$ uniformly over all large $a$ we have:
$$ \bbP_{\sqrt{a}-\delta} \Big[ \tau_{-\sqrt{a}+\delta} > T \Big|\, \tau_{-\sqrt{a}+\delta} < \tau_{\sqrt{a}-\delta/2}\Big] \lesssim a^{-r}\;.$$
\end{lemma}
\begin{proof}
Fix $r>0$. We keep the notation of the preceding proof. Since
$$ \bbP_{\sqrt{a}-\delta} \Big( \tau_- \ge T \, \big|\, \tau_- < \tau_+\Big) = \lim_{T'\rightarrow\infty} \frac{\bbP_{\sqrt{a}-\delta} \big( T \le \tau_- < T'\wedge \tau_+\big)}{\bbP_{\sqrt{a}-\delta} \big( \tau_- < T'\wedge \tau_+\big)}\;,$$
and since $\bbP_{\sqrt{a}-\delta} \big( \tau_- < T'\wedge \tau_+\big) \ge \bbP_{\sqrt{a}-\delta} \big( \tau_- < T\wedge \tau_+\big)$ for $T'>T$, the lemma will follow if we prove that for $c$ large enough we have the following bound
\begin{equation}\label{Eq:BoundLengthCross}
\sup_{T'>T} \frac{\bbP_{\sqrt{a}-\delta} \big( T \le \tau_- < T'\wedge \tau_+\big)}{\bbP_{\sqrt{a}-\delta} \big( \tau_- < T\wedge \tau_+\big)} \lesssim a^{-r}\;.
\end{equation}
By Girsanov's Theorem applied to the diffusions stopped at time $T'\wedge \tau_- \wedge \tau_+$, it suffices to show that
$$ \sup_{T'>T}\frac{\bE_{\sqrt{a}-\delta} \big[ T \le \tau_- < T'\wedge \tau_+ \;;\; e^{2\int_0^{\tau_-} H(t) dt}\big]}{\bE_{\sqrt{a}-\delta} \big[ \tau_- < T\wedge \tau_+ \;;\; e^{2 \int_0^{\tau_-} H(t) dt}\big]} \lesssim a^{-r}\;.$$
We let $n' := \lfloor T'/T \rfloor$ and we set for every $n\ge 0$ and every $x\in \bbR$
$$ A_x(n) := \bE_{x} \Big[ \tau_- < \tau_+ \;;\; \tau_-\in [nT,(n+1)T)\;;\; e^{2\int_0^{\tau_-} H(t) dt}\Big]\;.$$
We have
\begin{align*}
\bE_{\sqrt{a}-\delta} \big[ T\le \tau_- < T'\wedge \tau_+ \;;\; e^{2\int_0^{\tau_-} H(t) dt}\big] \le \sum_{n=1}^{n'} A_{\sqrt{a}-\delta}(n)\;.
\end{align*}
Applying the Markov property at time $T$, and bounding $H$ by $\sqrt{a}$ on the interval $[0,T)$, we get for all $n\ge 1$ and $x \in (-\sqrt{a} + \delta,\sqrt{a} - \delta/2)$,
\begin{align*}
A_{x}(n)= &\bE_{x} \Big[ T \le \tau_- < \tau_+ \;;\; e^{2\int_0^{T} H(s) ds} \;;\; A_{H(T)}(n-1) \Big]\\
\le &\;e^{2\sqrt{a} T}\; \bP_{x} \big[ T \le \tau_- \big]\sup_{y \in (-\sqrt{a}+\delta,\sqrt{a}-\delta/2)} A_y(n-1)\;.
\end{align*}
On the other hand, we obviously have the bound
$$ A_y(0) \le e^{2\sqrt{a} T} \bP_y\big[\tau_- < T \wedge \tau_+\big] \le e^{2\sqrt{a} T}\;.$$
Thus, a simple recursion for the first bound and the monotonicity of the process for the second bound, yields
\begin{align*}
A_{\sqrt{a}-\delta} (n) &\le e^{(n+1) 2\sqrt{a} T} \sup_{y \in (-\sqrt{a}+\delta,\sqrt{a}-\delta/2)} \bP_{y} \big[T < \tau_- \big]^n\\
&\le e^{(n+1) 2\sqrt{a} T} \bP_{\sqrt{a}-\delta/2} \big[T < \tau_- \big]^n\;.
\end{align*}
By the computations made in the proof of Lemma \ref{Lemma:Girsanov}, we have
$$ \bP_{\sqrt{a}-\delta/2} \big[T < \tau_- \big] \lesssim a^{-c'}\;,$$
for any $c'>0$, as well as
\begin{align*}
\bE_{\sqrt{a}-\delta} \big[ \tau_- < T\wedge \tau_+ \;;\; e^{2\int_0^{\tau_-} H(t) dt}\big] &\ge (1-O(a^{-c'})) e^{- 2\sqrt{a} T}\ge \frac12 e^{-2\sqrt{a} T}\;.
\end{align*}
Putting everything together, we get that for all $T' > T$,
$$\bbP_{\sqrt{a}-\delta} \Big[ \tau_- < T \, \big|\, \tau_- < \tau_+\Big]  \lesssim e^{4\sqrt a T} \sum_{n\ge 1} (e^{2\sqrt a T} a^{-c'})^n\;,$$
which is bounded by a term of order $a^{-r}$, for any given $r>0$, provided $c'$ is large enough.
\end{proof}

We now control the portion of trajectory from $-\sqrt{a}+\delta$ to $-\sqrt{a}$.
\begin{lemma}\label{Lemma:Girsanov3}
For all $C>1$ we have
$$ \P_{-\sqrt{a}+\delta}\Big[ \tau_{-\sqrt{a}} > C\frac{\ln \ln a}{\sqrt{a}}\wedge \tau_{-\sqrt{a}+2\delta} \,\big| \, \tau_{-\sqrt{a}} < \tau_{\sqrt{a}-\delta/2}\Big] \lesssim (\ln a)^{2-2C}\;.$$
\end{lemma}
\begin{proof}
We set $S = C\frac{\ln \ln a}{\sqrt{a}}$, $\tau_- := \tau_{-\sqrt{a}}$, $\tau_+ := \tau_{-\sqrt{a}+2\delta}$ and $\tau_{++} := \tau_{\sqrt{a}-\delta/2}$. For convenience, we also set $I(a) := e^{2(V(-\sqrt{a}+\delta)-V(-\sqrt{a}))}$ as this term will pop up in many equations below. We are going to show that we have
$$ \P_{-\sqrt{a}+\delta}\Big[\tau_{-} < S\wedge \tau_{+}\Big] \gtrsim I(a)e^{-2 \sqrt{a}\,S}\;,$$
as well as
\begin{align*}
\P_{-\sqrt{a}+\delta}\Big[S \wedge \tau_{+} \le \tau_{-} < S' \wedge \tau_{++}\Big]\lesssim I(a) e^{-2\sqrt{a}\,S} (\ln a)^{2-2C}\;,
\end{align*}
uniformly over all $S' >0$ and all $a$ large enough. These two bounds yield the statement of the lemma.\\
To prove these bounds we apply Girsanov's Theorem to the diffusions stopped at time $S'\wedge \tau_-\wedge \tau_{++}$. This will allow us to approximate our process by an Ornstein-Uhlenbeck process and use standard estimates about OU.  Regarding the first term, we have:
\begin{align*}
\P_{-\sqrt{a}+\delta}\Big[\tau_{-} < S\wedge \tau_{+}\Big]&= I(a) \bE_{-\sqrt{a}+\delta}\Big[e^{2\int_0^{\tau_{-}} H(s) ds}\;;\; \tau_{-} < S\wedge \tau_{+} \Big]\\
&\geq I(a) e^{- 2 \sqrt{a}\,S}\,\bP_{-\sqrt{a}+\delta}\Big[\tau_{-} < S\wedge \tau_{+} \Big]\;.
\end{align*}
Now observe that the process $Z(t) = H(t) + \sqrt{a}$ solves
$$ dZ(t) = Z(t)(Z(t)-2\sqrt{a}) dt + dB(t)\;,\quad Z(0) = \delta\;,$$
so that $Z(t) \le U(t)$ for all $t\in [0,\tau_{-}(H) \wedge \tau_{+}(H)]$ where $U$ is the Ornstein-Uhlenbeck process defined by
$$ dU(t) = -2(\sqrt{a}-\delta) U(t) dt + dB(t)\;,\quad U(0) = \delta\;.$$
Consequently, writing $\bP^U$ for the law of $U$, we get
\begin{align*}
\bP_{-\sqrt{a}+\delta}\Big[\tau_{-} < S\wedge \tau_{+} \Big] &\ge \bP^U_{\delta}\Big[\tau_{0} < S\wedge \tau_{2\delta} \Big]\\
&\ge \bP^U_{\delta}\Big[\tau_{0} < S \Big] - \bP^U_{\delta}\big[\tau_{0} > \tau_{2\delta}\big]\;.
\end{align*}
By the formulae~\cite[II.7.2.0.2 and II.7.2.2.2, p.542]{Handbook} we get the bounds
$$ \bP^U_{\delta}\Big[\tau_{0} < S\Big] = 1 - O(\delta a^{\frac14} e^{-2(\sqrt{a}-\delta)S})\;,$$
and
$$ \bP^U_{\delta}\big[\tau_{0} > \tau_{2\delta} \big] = \frac{\int_0^{\sqrt{2}\delta \sqrt{\sqrt{a}-\delta}} e^{v^2} dv }{\int_0^{2\sqrt{2}\delta \sqrt{\sqrt{a}-\delta}} e^{v^2} dv}\lesssim e^{- \rho a^{1/2} \delta^2}\;,$$
for some $\rho > 0$. Thus we get the asserted lower bound for the first term.

We turn to the second term. By Girsanov's Theorem, we get
\begin{align*}
&\bbP_{-\sqrt{a}+\delta}\Big[S \wedge \tau_{+} \le \tau_{-} < S' \wedge \tau_{++}\Big]\\
=\;& I(a) \bE_{-\sqrt{a}+\delta}\Big[e^{2\int_0^{\tau_{-}} H(s) ds}\;;\;  S \wedge \tau_{+} \le \tau_{-} < S' \wedge \tau_{++} \Big]\\
\le\;& I(a) \bE_{-\sqrt{a}+\delta}\Big[e^{2\int_0^{\tau_{-}} H(s) ds}\;;\;  S \wedge \tau_{+} \le \tau_{-} < \tau_{++} \Big]\;. 
\end{align*}
We bound the expectation at the last line by the sum of
$$ A := \bE_{-\sqrt{a}+\delta}\Big[e^{2\int_0^{\tau_{-}} H(s) ds}\;;\;  S \le \tau_{-} < \tau_{+} \Big]\;,$$
and
$$ B := \bE_{-\sqrt{a}+\delta}\Big[e^{2\int_0^{\tau_{-}} H(s) ds}\;;\;  \tau_{+} \le \tau_{-} < \tau_{++} \Big]\;,$$
and we bound these two terms separately. We start with $A$:
\begin{align*}
A &\le e^{2(-\sqrt{a}+2\delta) S} \bP_{-\sqrt{a}+\delta}\big[ S \le \tau_{-} < \tau_{+} \big]\\
&\le e^{2(-\sqrt{a}+2\delta) S} \Big(1-\bP_{-\sqrt{a}+\delta}\big[ \tau_{-} < S\wedge \tau_{+} \big) \Big]\;.
\end{align*}
A previous calculation showed that
$$ 1-\bP_{-\sqrt{a}+\delta}\big[ \tau_{-} < S\wedge \tau_{+} \big] = O(\delta a^{\frac14} e^{-2(\sqrt{a}-\delta)S}) + O(e^{-\rho a^{1/2} \delta^2}) = \cO\big((\ln a)^{2-2C}\big)\;,$$
so that we get the required bound for $A$. Regarding $B$, the proof is slightly more involved. Let us first introduce:
\begin{align*}
B_+&= \bE_{-\sqrt{a}+\delta}\Big[ e^{2\int_0^{\tau_{+}} H(s) ds} \;;\; \tau_{+} < \tau_{-}\Big]\;,\\
B_- &= \bE_{-\sqrt{a}+2\delta}\Big[ e^{2\int_0^{\tau_{-\sqrt{a}+\delta}} H(s) ds} \;;\; \tau_{-\sqrt{a}+\delta} < \tau_{++}\Big]\;.
\end{align*}
Applying the strong Markov property at time $\tau_+$ we get
\begin{align*}
B &= \bE_{-\sqrt{a}+\delta}\Big[\bE_{-\sqrt{a}+\delta}\Big[e^{2\int_0^{\tau_{-}} H(s) ds}\;;\;  \tau_{+} \le \tau_{-} < \tau_{++} \,\big| \, \cF_{\tau_+}\Big]\Big]\\
&= \bE_{-\sqrt{a}+\delta}\Big[e^{2\int_0^{\tau_{+}} H(s) ds} \;;\; \tau_+ \le \tau_- \;;\;\bE_{-\sqrt{a}+2\delta}\Big[e^{2\int_0^{\tau_{-}} H(s) ds}\;;\; \tau_{-} < \tau_{++} \Big]\Big]\\
&= B_+ \cdot \bE_{-\sqrt{a}+2\delta}\Big[e^{2\int_0^{\tau_{-}} H(s) ds}\;;\; \tau_{-} < \tau_{++} \Big]\;.
\end{align*}
Applying the strong Markov property at the first hitting time of $-\sqrt{a}+\delta$, we then obtain:
\begin{align*}
B &= B_+ B_- \cdot \bE_{-\sqrt{a}+\delta}\Big[e^{2\int_0^{\tau_{-}} H(s) ds}\;;\; \tau_{-} < \tau_{++} \Big]\\
&= B_+ B_- \cdot \big(B + D\big)\;,
\end{align*}
where
$$ D = \bE_{-\sqrt{a}+\delta}\Big[e^{2\int_0^{\tau_{-}} H(s) ds}\;;\; \tau_{-} < \tau_{+} \Big] \le 1\;.$$
As a consequence, we find
$$ B = \frac{B_+ B_-}{1-B_+ B_-} D \le \frac{B_+ B_-}{1-B_+ B_-}\;.$$
We claim that
\begin{equation}\label{Eq:ClaimBd}
B_+ B_- \lesssim e^{2t_a \sqrt{a}-\kappa \sqrt{a} \delta^2} \lesssim a^2 e^{-\rho \ln^4 a}\;.
\end{equation}
With this claim at hand, we deduce that $B$ is bounded by a negligible term compared to $A$, thus concluding the proof. We are left with the proof of the claim. We have
$$ B_+ \le \bP_{-\sqrt{a}+\delta}\big[\tau_{+} < \tau_{-}\big] = \bP^U_\delta[\tau_{2\delta} < \tau_0] \lesssim e^{-\rho \sqrt{a} \delta^2}\;.$$
To bound $B_-$, we argue as follows. We set
\begin{align*}
E &:= \sup_{x\in (-\sqrt{a}+\delta,\sqrt{a}-\delta/2)} \bE_{x}\Big[ e^{2\int_0^{\tau_{-\sqrt{a}+\delta}} H(s) ds} \;;\; \tau_{-\sqrt{a}+\delta} < \tau_{++}\Big]\;.
\end{align*}
By considering the two complementary events $\tau_{-\sqrt{a}+\delta} < t_a$ and $t_a \le \tau_{-\sqrt{a}+\delta}$, and by applying the Markov property at time $t_a$ in the second case, we get
$$ E \le e^{2t_a \sqrt{a}} + E\, e^{2t_a \sqrt{a}} \sup_{x\in (-\sqrt{a}+\delta,\sqrt{a}-\delta/2)} \bP_x[t_a < \tau_{-\sqrt{a}+\delta} < \tau_{++}]\;.$$
The proof of Lemma \ref{Lemma:Girsanov} ensures that, if $H$ starts from $\sqrt{a}-\delta/2$, then we have $\tau_{-\sqrt{a}+ \delta/2}< t_a$ on the event 
$\{\sup_{[0,t_a]} |B(t)| \le c_0 \ln a/a^{1/4}\}$, and therefore $\tau_{-\sqrt{a}+\delta}< t_a$ on the same event. By monotonicity, this remains true if $H$ starts from any point in $(-\sqrt{a}+\delta,\sqrt{a}-\delta/2)$. Recall that $\bP[\sup_{[0,t_a]} |B(t)| > c_0 \ln a/a^{1/4}] \lesssim a^{-c_0^2/2}$. Consequently, choosing $c_0$ large enough we get the crude bound
\begin{align*}
B_- \le E \le \frac{e^{2t_a \sqrt{a}}}{1-e^{2t_a \sqrt{a}} \bP[\sup_{[0,t_a]} |B(t)| > c_0 \ln a/a^{1/4}]} \lesssim a^2\;,
\end{align*}
uniformly over all $a$ large enough which concludes the proof of the claim.
\end{proof}

\begin{proof}[Proof of Proposition \ref{Prop:BoundsCross1}-(i) and (ii)]
Statement (i) is a consequence of Lemma \ref{Lemma:DiffXY}. Regarding (ii), we argue as follows.

After its first hitting time of $\sqrt a -\delta$, we decompose the path $X_a$ into two types of "bridges":\begin{itemize}
\item those that start at $\sqrt a - \delta$, hit $\sqrt a$ before $-\sqrt a$ and are stopped at their next hitting time of $\sqrt a -\delta$,
\item those that start at $\sqrt a - \delta$, hit $-\sqrt a$ before $\sqrt a$, and are stopped at their next hitting time of $\sqrt a -\delta$ (possibly after an explosion).
\end{itemize} 
We are interested in the first bridge hitting $-\sqrt{a}$. It follows the conditional law $\P_{\sqrt{a}- \delta}( \;\cdot\; | \tau_{-\sqrt{a}} < \tau_{\sqrt{a}})$ up to its ending time. We first show that it does not hit $\sqrt{a} - \delta/2$ before its first hitting time of $-\sqrt a$ with large probability. Indeed if we bound $\P_{\sqrt a - \delta}[\tau_{\sqrt a-\delta/2} < \tau_{-\sqrt a}]$ by $1$, then we get 
\begin{align*}
\P_{\sqrt{a} - \delta}\big[\tau_{-\sqrt{a}} < \tau_{\sqrt{a} - \delta/2} \,\big| \tau_{-\sqrt{a}} < \tau_{\sqrt{a}}\big] &\geq 1 -  \frac{\P_{\sqrt{a} - \delta /2}[\tau_{-\sqrt{a}} < \tau_{\sqrt{a}}]}{\P_{\sqrt{a} - \delta}[\tau_{-\sqrt{a}} < \tau_{\sqrt{a}}]}.
\end{align*}
Using the scale function \eqref{Eq:ScaleHit} and the fact that there exists $c>0$ such that $V(y+\delta/2) < V(y) - c\sqrt a \delta^2$ uniformly over all $y\in[\sqrt a - \delta,\sqrt a -\delta/2]$, we deduce that
\begin{align*}
\P_{\sqrt{a} - \delta}\big[\tau_{-\sqrt{a}} < \tau_{\sqrt{a} - \delta/2} \,\big| \tau_{-\sqrt{a}} < \tau_{\sqrt{a}}\big]  \ge 1- \exp(-c \,(\ln a)^4).
\end{align*}
Therefore, combining Lemma \ref{Lemma:Girsanov}, Lemma \ref{Lemma:Girsanov2}, Lemma \ref{Lemma:Girsanov3} and this last inequality, we obtain that there exists $C >0$ such that with probability larger than $1- 1/\ln a$, the trajectory of the first bridge hitting $-\sqrt a$ stays below $\sqrt{a} - \delta/2$ and:
\begin{itemize}
\item satisfies the complement of the event $E(C)$,
\item hits $0$ before time $(3/8) t_a + C \ln\ln a /\sqrt a$ but after time $(3/8) t_a - C \ln\ln a /\sqrt a$, and hits $-\sqrt{a} + \delta$ before time $(3/4) t_a + 2C \ln\ln a / \sqrt a$ but after time $(3/4) t_a - 2C \ln\ln a / \sqrt a$,
\item stays below $-\sqrt{a} + 2 \delta$ on the time interval $[\tau_{-\sqrt{a} + \delta}, \tau_{-\sqrt{a}}]$ and $\tau_{-\sqrt{a}} - \tau_{-\sqrt{a} + \delta} < C \ln\ln a/\sqrt{a}$.
\end{itemize}
Therefore, for $X_a$ in the time interval $[\iota_a,\theta_a]$, we deduce that: 
\begin{itemize}
\item The first portion of the trajectory from $\iota_a$ until its first hitting time of $\sqrt{a} - \delta$ stays in the interval $[\sqrt{a} - \delta, \sqrt{a}]$ while the function $t \mapsto -\sqrt{a}\tanh(-\sqrt{a}(t-\upsilon_a))$ is in $[\sqrt{a} - (\ln a)^{4C}/a^{1/4}, \sqrt{a}]$ for $t \leq \upsilon_a -(3/8)t_a + C\ln\ln a/\sqrt{a}$. Since $\upsilon_a \geq (3/8)t_a - C\ln\ln a/\sqrt{a}$, we deduce that these two curves are at distance smaller than $C \sqrt{a}/\ln a$ of each other (in fact, they are much closer).
\item The next portion after its first hitting time of $\sqrt{a} - \delta$ stays below $\sqrt{a} - \delta/2$ and takes a time $3/8\, t_a + O(\ln \ln a /  \sqrt a)$ to reach $0$ and then a time of the same order to reach $-\sqrt a + \delta$, and stays at a distance of order at most $\sqrt a / \ln a$ from $\sqrt a \tanh(-\sqrt a(t - \tau_0))$ where $\tau_0$ is the first hitting time of $0$ after $\iota_a$. This implies that the first and the last hitting times of $0$ of this portion of the trajectory are very close to each other (at distance of order at most $1/(\sqrt{a} \ln a)$): hence, we can replace $\tau_0$ by $\upsilon_a$ in the hyperbolic tangent without modifying the order of magnitude of the bound. 
\item The last portion of the trajectory stays below $-\sqrt a + 2\delta$ and takes at most $C \ln \ln a / \sqrt a$ to hit $-\sqrt a$.
\end{itemize}
This yields (ii).
\end{proof}

\subsection{Coupling of $X_a$ and $X_{a + \eps}$}\label{subsec:coupling}

The next lemma controls the difference between $X_{a}$ and $X_{a+\eps}$, and yields in particular the proof of Proposition \ref{Prop:BoundsCross2}.
\begin{lemma}[Difference between $X_a$ and $X_{a+\eps}$]\label{Lemma:DiffDiff}
Take $\eps \in (0,1]$. There exists $C >0$, such that with probability at least $1-O(1/\ln a)$, the following holds true:
\begin{align*}
&\sup_{t\in [\upsilon_a-(1/16)t_a,\upsilon_a+(1/16) t_a]} |X_a(t)-X_{a+\eps}(t)| < 1\;, \\
&\sup_{t\in [\upsilon_a+(1/16)t_a, \theta_a - (1/16) t_a]} X_{a+\eps}(t) <  -\sqrt{a} +C \,a^{3/7}\;, \\
&\sup_{t\in [\theta_a - (1/16) t_a,\theta_a]} X_{a+\eps}(t) \le \sqrt{a}-1\;.
\end{align*}
\end{lemma}

\begin{remark}
Those bounds are not optimal because we overestimate the difference between $X_a$ and $X_{a + \eps}$ at the initial time where $X_a$ starts its descent to $- \sqrt{a}$. We indeed expect that when $\eps$ is of order $1/\sqrt{a}$, the difference between $X_a$ and $X_{a+ \eps}$ is of order $1/a^{1/4}$ around time $\theta_a$.
\end{remark}
\begin{proof}
Recall the decomposition of the process $X_a$ from the previous proof and let $\sigma_a$ be the starting time of the bridge that starts at $\sqrt a - \delta$ and hits $-\sqrt a$ before $\sqrt a$. We introduce the process $Z(t) := X_{a+\eps}(\sigma_a+t) - X_a(\sigma_a+t)$, which solves
$$ dZ(t) = - Z(t) (X_a(\sigma_a+t) + X_{a+\eps}(\sigma_a+t)) dt + \eps dt\;.$$
Since $X_a(\sigma_a+t) \le X_{a+\eps}(\sigma_a+t)$ until $X_a$ explodes, we deduce that
$$ dZ(t) \le -2 X_a(\sigma_a+t) Z(t) dt + \eps dt\;,$$
or, written in its integrated form
$$ Z(t) \le Z(0) e^{-2\int_0^t X_a(\sigma_a+s) ds} + \eps \int_0^t e^{-2\int_s^t X_a(\sigma_a + r) dr} ds\;.$$
By Lemma \ref{Lemma:Girsanov}, we know that with probability at least $1-O(a^{-c})$, the process $X_a(\sigma_a+t)$ is bounded from below by $Z_-(t)-M$ until it hits $-\sqrt a + \delta$, where $Z_-$ was introduced in the proof of Lemma \ref{Lemma:Girsanov}. Since $e^{2 M t} \le 2$ for all $t\in [0, t_a]$ and all $a$ large enough, we have the bound
$$ Z(t) \le  2Z(0) e^{-2\int_0^t Z_-(s) ds} + 2\eps \int_0^t e^{-2\int_s^t Z_-(r) dr} ds\;.$$
A simple integration yields for all $t\ge 0$
$$ \int_0^t Z_-(s) ds = \frac1{(1-\kappa)^2}\ln \frac{\cosh \sqrt{a_-}\,A_-}{\cosh \sqrt{a_-}(t-A_-)}\;.$$
where $a_- := a (1+\kappa) (1- \kappa)^2$.

Similarly, we get
\begin{align*}
\int_0^t e^{-2\int_s^t Z_-(r) dr} ds &= (\cosh \sqrt{a_-}(t-A_-))^{2/(1-\kappa)^2} \int_0^t \frac{1}{\big(\cosh \sqrt{a_-}(s-A_-)\big)^{2/(1-\kappa)^2}} ds\\
&\le (\cosh \sqrt{a_-}(t-A_-))^{2/(1-\kappa)^2} \int_0^t \frac{1}{\big(\cosh \sqrt{a_-}(s-A_-)\big)^{2}} ds\\
&\leq \frac1{\sqrt{a_-}} (\cosh(\sqrt{a_-}(t-A_-)))^{2/(1-\kappa)^2}  \Big(\tanh \sqrt{a_-}(t-A_-) + \tanh \sqrt{a_-} A_- \Big)\;.
\end{align*}
Fix $\rho \in (0,3/4)$. For all $t\in [\rho t_a, (\frac34 -\rho) t_a]$ we have
\begin{align}
e^{-2\int_0^t Z_-(s) ds} = \Big(\frac{\cosh \sqrt{a_-}(t-A_-)}{\cosh \sqrt{a_-} A_-}\Big)^{\frac{2}{(1-\kappa)^2}} \lesssim e^{-2\sqrt{a} \frac{(1+\kappa)^{1/2}}{1-\kappa}\big( t \wedge (2A_- -t) \big)} \le a^{-2\rho+o(1)}\;. \label{eq:diffcoupling1}
\end{align}
Additionally, we have for all $t\in [\rho t_a, (\frac34 -\rho) t_a]$
\begin{align}
\eps \int_0^t e^{-2\int_s^t Z_-(r) dr} ds \le \eps \frac2{\sqrt{a_-}} e^{\frac{2}{(1-\kappa)^2}\sqrt{a_-} |t-A_-|} \le \eps a^{\frac14-2\rho+o(1)}\;.\label{eq:diffcoupling2}
\end{align}

From \eqref{eq:diffcoupling1} and \eqref{eq:diffcoupling2}, we deduce that
\begin{align*}
\forall t\in [\rho t_a, (\frac34 -\rho) t_a], \quad Z(t) \leq 2Z(0) a^{-2\rho+o(1)} + \eps a^{\frac14-2\rho+o(1)}.
\end{align*}

By Lemmas \ref{Lemma:Girsanov}, \ref{Lemma:Girsanov2} and \ref{Lemma:Girsanov3}, we know that $\upsilon_a - \sigma_a$ and $\theta_a - \upsilon_a$ are of order $3/8 t_a + O(\ln \ln a / \sqrt a)$ with probability $1-O(1/\ln a)$. Furthermore, $X_{a+\eps}(\sigma_a) < 10 \sqrt a$ with a huge probability so that $Z(0) < 10 \sqrt a$. We thus easily get the first bound of the statement of the lemma.

Notice that we also deduce that there exists $C$ such that for all $a$ large enough 
\begin{align*}
\sup_{t \in [\upsilon_a + (1/16) t_a, \theta_a - (1/16) t_a]} X_{a+\eps}(t) <  -\sqrt{a} +C \,a^{3/7},
\end{align*}
which gives the second bound. 

We turn to the third bound, for which we control $X_{a+\eps}$ on the time interval $[\theta_a - (1/16) t_a,\theta_a]$. We have:
$$ dZ(t) = -2X_a(t) Z(t)dt - Z(t)^2dt + \eps dt \le 2\sqrt{a} Z(t) dt - Z(t)^2 dt + \eps dt = \Big(a+\eps - (Z(t)-\sqrt{a})^2\Big)dt\;,$$
for all $t\le \theta_a$. Notice that $X_a = -\sqrt a + o(1)$ on this time interval. Consequently, for all $t\in [\theta_a - \frac1{16} t_a,\theta_a]$ we have the trivial bound:
\begin{align*}
X_{a+\eps}(t) \le \sqrt{a+\eps} \tanh\big(\sqrt{a+\eps}(t-A_\eps)\big) + o(1)\;,
\end{align*}
where $A_\eps$ is such that the r.h.s. coincides with $X_{a+\eps}(\theta_a - \frac1{16} t_a)$ at time $\theta_a - \frac1{16} t_a$. As $X_{a + \eps}(\theta_a -  \frac1{16} t_a) = -\sqrt{a} + O(a^{3/7})$, we deduce the third bound of the statement.
\end{proof}

\begin{lemma}\label{Lemma:Ordering}
Take $\eps \in (a^{-2/3},1)$. Assume that $X_a(t)$ and $X_{a+\eps}(t)$ lie in $[\sqrt{a}/2,3\sqrt{a}/2]$ for all $t\in[0,2t_a]$. If $X_{a+\eps}(0) < X_a(0)$, then by time $2t_a$ and for all $a$ large enough, $X_{a+\eps}$ passes above $X_a$.
\end{lemma}
\begin{proof}
Assume that $X_{a+\eps}$ is below $X_a$ up to time $2t_a$, then the difference $Z(t) = X_a(t) - X_{a+\eps}(t)$ is positive and satisfies for all $t \in [0,2t_a]$:
\begin{align*}
Z(t) &=  Z(0) \exp(- \int_{0}^t (X_a+X_{a+\eps})(s) ds) -\eps \int_{0}^{t} \exp(- \int_{v}^t (X_a+X_{a+\eps})(s) ds) dv\\
&\leq Z(0) \exp(- \sqrt{a} t)-\eps \int_{0}^{t} \exp( - 3 \sqrt{a}(t-v))dv \\
&\leq  Z(0) \exp(- \sqrt{a} t) - \frac{\eps}{3 \sqrt{a}}(1-e^{-3 \sqrt{a} t})\;.
\end{align*}
By assumption $Z(0) \leq \sqrt{a}$. At time $t=2t_a$, we thus find
\begin{align*}
Z(t) &\leq a^{-\frac32} - \frac{\eps}{3 \sqrt{a}}(1-o(1))\;.
\end{align*}
We get $Z(t) < 0$ for all $a$ large enough, thus raising a contradiction.
\end{proof}

\subsection{Bounds after time $\theta_a$}\label{subsec:boundsaftertheta}

\begin{lemma}\label{Lemma:ExploandHit}
As $a\rightarrow +\infty$, we have $\P_{-\sqrt{a}}\big( \tau_{-\sqrt{a}-\delta} < \tau_{-\sqrt{a}+\delta} \big) \rightarrow 1/2$. Furthermore, there exists $C>0$ such that for all $a$ large enough and all $x \in [-\sqrt{a} - \delta, -\sqrt{a} + \delta]$,
\begin{equation}\label{Eq:BndToto}
\E_{x}\big[ \tau_{-\sqrt{a}+\delta} \wedge \tau_{-\sqrt{a}-\delta}\big] \le C \frac{\ln \ln a}{\sqrt a}\;.
\end{equation}
Consequently,
\begin{align*}
&\P_{-\sqrt a}\big[ \tau_{-\sqrt{a}-\delta}  > \frac{(\ln \ln a)^2}{4 \sqrt a}  \,\big|\, \tau_{-\sqrt{a}-\delta} < \tau_{-\sqrt{a}+\delta}\big] \le \frac{12C}{\ln \ln a}\;,\\
\mbox{and}\quad &\P_{-\sqrt a}\big[ \tau_{-\sqrt{a}+\delta}  > \frac{(\ln \ln a)^2}{4 \sqrt a}  \,\big|\, \tau_{-\sqrt{a}+\delta} < \tau_{-\sqrt{a}-\delta}\big] \le \frac{12C}{\ln \ln a}\;.
\end{align*}
\end{lemma}

\begin{proof}
Recall the identity \eqref{Eq:ScaleHit}. Since $V(y) = V(-\sqrt a) - (y+\sqrt a)^2 \sqrt a + (y+\sqrt a)^3 / 3$, a simple computation yields the following asymptotics:
\begin{align}
&S(-\sqrt{a}+\delta)-S(-\sqrt a) = \int_{-\sqrt a}^{-\sqrt{a}+\delta} \exp(2 V_a(y)) dy \sim \frac{\sqrt{\pi}}{2 \sqrt{2}} \;a^{-1/4} \exp(\frac{4}{3}a^{3/2})\;, \notag\\
&S(-\sqrt{a} + \delta)-S(-\sqrt a - \delta) = \int_{-\sqrt a -\delta}^{-\sqrt{a}+\delta} \exp(2 V_a(y)) dy \sim \frac{\sqrt{\pi}}{\sqrt{2}} \;a^{-1/4} \exp(\frac{4}{3}a^{3/2})\;. \label{equivdiffS}
\end{align}
Therefore $\P_{-\sqrt{a}}[\tau_{-\sqrt{a}-\delta} < \tau_{-\sqrt{a}+\delta} \big] \to 1/2$. We now estimate the expectation of the time it takes to exit the interval $(-\sqrt{a} - \delta,-\sqrt{a} + \delta)$ when starting at $x \in [-\sqrt{a}- \delta,-\sqrt{a}+ \delta]$. Applying~\cite[Th VII.3.6]{RevuzYor}, we find:
\begin{align*}
\E_{x}[\tau_{-\sqrt{a} - \delta} \wedge \tau_{-\sqrt{a}+\delta}] 
&=2 \,\P_{x}[\tau_{-\sqrt{a}-\delta} < \tau_{-\sqrt{a}+\delta} ] \int_{-\sqrt{a}-\delta}^x \frac{S(y)- S(-\sqrt{a}-\delta)}{S'(y)} dy \\
&\quad + 2\, \P_{x}[ \tau_{-\sqrt{a} + \delta}  < \tau_{-\sqrt{a}-\delta}]  \int_{x}^{-\sqrt{a} + \delta}\frac{S(-\sqrt{a} + \delta) - S(y)}{S'(y)} dy\;.
\end{align*}
Let us bound the first term on the r.h.s., the second term can be bounded in the same way.
Using the change of variable $y = -\sqrt{a} + u/a^{1/4}$, and neglecting the cubic terms in the potential we get:
\begin{align*}
&\int_{-\sqrt{a}-\delta}^{x}\frac{S(y)- S(-\sqrt{a}-\delta)}{S'(y)} dy \le \frac{2}{\sqrt{a}} \int_{-a^{1/4} \delta}^{(x + \sqrt{a}) a^{1/4}} \exp(2 u^2) \int_{-a^{1/4} \delta}^{u} \exp(-2 v^2) dv du\;.
\end{align*}
If $(x+\sqrt{a})a^{1/4}$ is bounded from above, the simple inequality $\int_y^\infty \exp(- 2 u^2) du \leq \exp(- 2 y^2)/(4y)$ that holds for $y>0$ allows to prove that the term is of order $\ln(\delta a^{1/4}) / \sqrt a$. If $(x+\sqrt{a})a^{1/4} \to +\infty$, the inequality 
\begin{align*}
\P_x[\tau_{-\sqrt{a}-\delta} < \tau_{-\sqrt{a}+\delta} ] \leq C_0 \int_{(x+\sqrt{a})a^{1/4}}^{\delta a^{1/4}} \exp(- 2 y^2) dy
\end{align*}
allows to get the same bound. This yields \eqref{Eq:BndToto}. Thanks to Markov's inequality and since $\P_{-\sqrt a}\big(\tau_{-\sqrt{a} \pm \delta} < \tau_{-\sqrt{a}\mp\delta}\big) > 1/3$ for all $a$ large enough, we get:
\begin{align*}
\P_{-\sqrt a}\Big[\tau_{-\sqrt{a}\pm \delta} > \frac{(\ln\ln a)^2}{4\sqrt{a}} \, \big|\, \tau_{-\sqrt{a} \pm \delta} < \tau_{-\sqrt{a}\mp \delta}\Big]&\le \frac{\P_{-\sqrt a}\big[\tau_{-\sqrt{a}- \delta} \wedge \tau_{-\sqrt{a}+\delta} >  (\ln\ln a)^2 / (4\sqrt{a})\big]}{\P_{-\sqrt a}\big[\tau_{-\sqrt{a} \pm \delta} < \tau_{-\sqrt{a}\mp\delta}\big]}\\
&\le \frac{12C}{\ln \ln a}\;,
\end{align*}
as required.
\end{proof}

We need a last lemma that controls the time needed by the diffusion $X_a$ to return to $\sqrt a$ when it starts from $-\sqrt{a}+ \delta$.
\begin{lemma}\label{lem:return_sqrta}
There exists $C >0$ such that for all $a$ large enough, we have:
\begin{align*}
\P_{-\sqrt{a}+ \delta}\Big[\tau_{\sqrt{a}} \leq \frac{3}{4} t_a + C \frac{\ln\ln a}{\sqrt{a}},\; \tau_{\sqrt{a}}< \tau_{-\sqrt{a}} \Big] \geq 1- \cO\big(1/\ln a\big).
\end{align*}
\end{lemma}

\begin{proof}
It suffices to introduce $Z := X - B$ and adapt the proof of Lemmas \ref{Lemma:Girsanov}, \ref{Lemma:Girsanov2} and \ref{Lemma:Girsanov3}.
\end{proof}

This readily implies the last event (iii) of Proposition \ref{Prop:BoundsCross1}:
\begin{proof}[Proof of (iii) of Proposition \ref{Prop:BoundsCross1}]
After time $\theta_a$, the process $X_a$ has the law of the diffusion starting from $-\sqrt a$. If it exits the interval $[-\sqrt a-\delta,-\sqrt a + \delta]$ through $-\sqrt a-\delta$ then by Lemma \ref{Lemma:ExploandHit} it does so in a time smaller than $(\ln\ln a)^2/(4\sqrt a)$ with probability at least $1-O(1/\ln\ln a)$, and by Lemma \ref{Lemma:Explo} it stays below $-\sqrt a + 1$ and goes to $-\infty$ within a time $(3/8) t_a$, with probability at least $1-\exp(- (\ln \ln a)^2)$.\\
On the other hand, if it exits through $-\sqrt a + \delta$, then by Lemma \ref{lem:return_sqrta} it reaches $\sqrt a$ without hitting $-\sqrt a$ within a time $(3/4)t_a+ C \ln\ln a /\sqrt a$ with probability at least $1-O(1/\ln a)$.
\end{proof}
%
%
\begin{proof}[Proof of Proposition \ref{Prop:BoundsCross3}]
The r.v.~$\theta \wedge L$ is a stopping time in the filtration $\cF_t,t\ge 0$ of the underlying Brownian motion $B$. By the strong Markov property, the process $(B(t+\theta\wedge L)-B(\theta\wedge L),t\ge 0)$ is a standard Brownian motion, independent from $\cF_{\theta\wedge L}$. Hence, conditionally given $\theta\wedge L$, the process $(\hat{X}_a(t),t\in [0,L-\theta\wedge L])$ has the law of the time-reversed diffusion stopped at the deterministic time $L-\theta\wedge L$. The two bounds of the statement follow from the same type of arguments as those presented in the proofs of Lemmas \ref{Lemma:ControlY} and \ref{Lemma:DiffXY}. In particular, we introduce a stationary diffusion $\hat{Y}$ (with the same law as the stationary diffusion $-Y$) driven by $\hat B$. We define the event
$$\mathcal{A} := \big\{\hat{Y}((L-t)\vee 0) \leq -\sqrt a + h_a/2\mbox{ for all }t\in[\theta,\theta + 11 t_a]\big\}\;.$$
From the estimates on the invariant measure collected in Lemma \ref{Lemma:InvMeas} and using a comparison with a reflected Brownian motion (as in the proof of Proposition \ref{Propo:DiffClose}), we deduce that there exists $\rho>0$ such that $\P[\cA] > 1-1/\ln a$ for all $a$ large enough. Denote by $\mathcal{B}$ the event on which $\hat{X}_a$ and $\hat{Y}$ remain at a distance $h_a/2$ from each other on the interval $[(3/8)t_a,\tau]$ where $\tau$ is the first hitting time of $3\sqrt a$ by $\hat Y$ and both of them explode after time $\tau$ within a time $O(1/\sqrt{a})$. The proof of Lemma \ref{Lemma:DiffXY} ensures that $\P[\cB] \ge 1 - 1/\ln a$.\\
Define now the event $\cC$ on which
$$ \fint_{\theta\wedge L}^{t} \hat{Y}(L-s) ds \le -\sqrt a + h_a/2\;,\quad \forall t\in [\theta\wedge L,L]\;.$$
We claim that there exists $c>0$ such that $\P[\cC \, |\, \mathcal{F}_{\theta\wedge L}] > 1- \exp(-c(\ln a)^2)$, and consequently $\P[\cC] > 1-\exp(-c(\ln a)^2)$. Indeed, if we introduce the events
$$ E_k := \Big\{ \fint_{\theta\wedge L}^{\theta\wedge L + 2^{-k} (L-\theta\wedge L)} \un_{\{-\hat{Y}(L-s) \notin I_a(1/4)\}} ds \ge \frac1{a} \Big\}\;,\quad k\ge 1\;,$$
then, by the stationarity of $\hat{Y}$, we deduce that $\P[E_k \,|\, \mathcal{F}_{\theta\wedge L}] \le a\, e^{-(\ln a)^2/16}$ so that the proof of Lemma \ref{Lemma:ControlY} carries through \textit{mutatis mutandis} and yields the asserted bound on the conditional probability of $\cC$.\\
Let $\cD$ be the event on which $\hat{Y}$ explodes to $+\infty$ within a time of order $1/\sqrt a$ once it has hit $3\sqrt a$. By Lemma \ref{Lemma:Explo}, $\P(\cD) > 1-\exp(-(\ln\ln a)^2)$.\\
Therefore, on the event $\cA\cap \cB \cap \cC \cap \cD$, if $\hat \zeta_a(1) > L- \theta - 10 t_a$ and $\theta < L$ then $\hat Y$ hasn't reached $3\sqrt a$ by time $L-\theta-10t_a$: indeed, if it had then by $\cD$ it would explode before time $L-\theta$ and this would contradict $\cA$. Using $\cA$, this in turn ensures that $\tau > L-\theta$. Since $\hat{X}_a$ remains below $\hat Y$ up to the first explosion time of the latter, the bound of event $\cA$ yields:
$$\hat{X}_a(L-t) \le -\sqrt{a}+h_a\;,\quad \forall t\in [\theta,\theta+10\, t_a]\;,$$
so that $\hat{X}_a$ does not explode before time $L-\theta$. Moreover the bound of event $\cC$ combined with the condition of event $\cB$ yields the second bound of the proposition.
\end{proof}

\subsection{Proof of parts \ref{(b)-(i)}, \ref{(b)-(ii)} and \ref{(b)-(iii)}  of Proposition \ref{Propo:cE}}\label{subsec:proofexcursions1}


Let us fix $\eps >0$. Let $a \in M_L$. Recall that by Proposition \ref{Prop:PoissonExplo}, the explosion times of $X_{a}$ (resp. $X_a^j$), rescaled by $m(a_L)=L$, converge to a Poisson point process on $\R_+$ (resp. on $[t_j^n,+\infty)$) of intensity bounded from above by $e^{4/\eps}$ as $a \geq a_L - 1/(\eps \sqrt{a_L})$. 
Consequently, with probability $1-O(2^{-n})$ for all $L$ large enough, $X_a$ does not explode on $[(t_j^n - 2\cdot 2^{-2n} L)\vee 0, (t_j^n + 2\cdot 2^{-2n} L)\wedge L]$, $X^j_{a}$ explodes at most once per interval $[t_j^n, t_{j+3}^n]$. This yields \ref{(b)-(i)} and \ref{(b)-(iii)} for all $a \in M_L$ with probability $1-O(2^{-n})$ for all $L$ large enough (since there are of order $1/\eps^2$ points in $M_L$). 

It is easy to generalize \cite[Thm 3.3]{AllezDumazTW} to show that the first hitting time of $-\sqrt{a}$, rescaled by $m(a)/2$, of the diffusion $X_a$ starting from $\sqrt{a}$ converges to an exponential r.v.~of parameter $1$ when $a \to \infty$ (its Laplace transform satisfies a similar fixed point equation as in~\cite[Proposition 3.3]{AllezDumazTW} replacing $-\infty$ in the first integral by $-\sqrt{a}$ and taking $y \geq -\sqrt{a}$. The same proof then carries through, noting that we have to divide $m(a)$ by $2$ so that the recursive integrals $R_n(y,a)$ converges to $1$). As a consequence, we have the following counterpart of Proposition \ref{Prop:PoissonExplo}: the first hitting times of $-\sqrt a$ of the successive excursions of $X_a$, rescaled by $m(a)$, converge as $a\to\infty$ to a Poisson point process on $\R_+$ of intensity $2$. Consequently the first hitting times of $-\sqrt a$ of the successive excursions of $X_a$ (resp. $X_{a}^j$), rescaled by $m(a_L)=L$, converge as $L\to\infty$ to a Poisson point process on $\R_+$ (resp. on $[t_j^n,\infty)$) of intensity bounded from above by $2e^{4/\eps}$. Therefore for $j=0$ and $j=2^n-1$,$X_a$ and $X_{a}^j$ do not hit $-\sqrt a$ on $[t_j^n,t_{j+1}^n]$ with probability at least $1-O(2^{-n})$. For any other given $j$, $X_a$ and $X_{a}^j$ hit at most once $-\sqrt a$ with probability at least $1-O(2^{-2n})$ so that, taking a union bound over all such $j$, we obtain a probability at least $1-O(2^{-n})$. Hence \ref{(b)-(ii)} is satisfied with probability at least $1-O(2^{-n})$ for all $L$ large enough.

\subsection{Proof of part \ref{(b)-(iv)} of Proposition \ref{Propo:cE}}\label{subsec:proofexcursions2}
Let $a_\ell$ be the smallest point in $M_L$. In the next lemma, we show that for $a>a_\ell$ in $M_L$, if $X_a^j$ explodes then $X_{a_\ell}^j$ explodes roughly at the same time, and this proves \ref{(b)-(iv)}.
\begin{lemma}\label{Lemma:Xaa0}
Fix $n\ge 1$ and take $a \in M_L$. The following holds for all $j\in\{0,\ldots,2^n-1\}$ with probability at least $1- O(1/\ln\ln a_L)$. If $X_{a_\ell}^j$ makes at most one excursion to $-\sqrt a_\ell$ on $[t_j^n,t_{j+1}^n]$, then:\begin{itemize}
\item[(A)] on $[t_j^n,t_{j+1}^n]$ the number of excursions of $X_a^j$ to $-\sqrt a$ is smaller than or equal to the number of excursions of $X_{a_\ell}^j$ to $-\sqrt a_\ell$,
\item[(B)] if $X_a^j$ hits $-\sqrt a$ on $[t_j^n,t_{j+1}^n]$ and then explodes without coming back to $\sqrt a$, then so does $X_{a_\ell}^j$ and their explosion times lie at a distance at most $(\ln\ln a_L)^2/(3\sqrt{a_L})$ from each other. 
\end{itemize}
\end{lemma}
\begin{proof}
Notice that the difference $a-a_\ell$ belongs to $(a^{-2/3}_L,1)$. We are going to describe the behavior of the two trajectories $X^j_a, X^j_{a_\ell}$ by combining several estimates obtained in previous lemmas. The bounds on the probabilities obtained in these lemmas ensure that all what follows happens with probability greater than $1- O(1/\ln\ln a_L)$. The diffusion $X^j_a$ remains above $X^j_{a_\ell}$ until the first explosion time of $X^j_{a_\ell}$, and cannot reach $-\sqrt{a}$ before the first hitting time of $-\sqrt{a_\ell}$ by $X^j_{a_\ell}$. Let us denote this first hitting time $\theta^j_{a_\ell}$. At time $\theta^j_{a_\ell}$, by the estimates of Proposition \ref{Prop:BoundsCross2} the diffusion $X^j_{a}$ lies in $[-\sqrt a_\ell, \sqrt a_\ell-1] \subset [-\sqrt a, \sqrt a_\ell -1]$. For convenience, let us distinguish two cases: the diffusion is either in the interval $[-\sqrt{a}, -\sqrt{a}+\delta)$ or in the interval $[-\sqrt{a} + \delta, \sqrt{a_\ell} - 1]$, where $\delta = \ln^2(a) / a^{1/4}$.

In the first case, thanks to Lemma \ref{Lemma:ExploandHit}, it takes a time smaller than $(\ln\ln a_L)^2/(4\sqrt{a_L})$ to exit the interval $[-\sqrt{a} - \delta, \sqrt{a} + \delta]$ with probability greater than $1- O(1/\ln\ln a)$. If it exits through $-\sqrt{a}- \delta$, then it takes a time $(3/8)t_L +O(\ln \ln a_L/\sqrt{a_L})$ to explode to $-\infty$ (without hitting $-\sqrt a$) by Lemma \ref{Lemma:Explo}.
From the ordering of the diffusions, this implies that $X_{a_\ell}^j$ explodes as well and that their explosion times are at distance at most $(\ln\ln a_L)^2/(3\sqrt{a_L})$ from one another. If it exits through $-\sqrt{a} + \delta$, then we get to the second case.

In the second case, thanks to Lemma \ref{lem:return_sqrta}, it takes a time smaller than $(3/4) t_L +O(\ln\ln a_L/\sqrt{a_L})$ for the diffusion to reach $\sqrt{a}$ without hitting $-\sqrt a$. Moreover, after this return time, it stays in the region $[\sqrt{a_\ell}/2, (3/2)\sqrt{a_\ell}]$ during a time greater than $5 t_L$ (notice that this happens with a huge probability). Similarly, the diffusion $X^j_{a_\ell}$ hits $\sqrt{a_\ell}$ (with or without explosion) before time $\theta_{a}^j + (3/4) t_L + o(t_L)$ and stays in the interval $[\sqrt{a}/2,(3/2)\sqrt{a_\ell}]$ during a time greater than $5 t_L$. We can therefore apply Lemma \ref{Lemma:Ordering} which states that $X^j_{a}$ passes above $X^j_{a_\ell}$ before time $\theta_{a_\ell}^j + 5 t_L$, thus preventing $X^j _a$ to make a second excursion to $-\sqrt a$ before $X^j_{a_\ell}$ makes itself a second excursion to $-\sqrt a_\ell$.

Therefore, with probability $1-O(1/\ln\ln a_L)$ if $X^j_{a_\ell}$ makes only one excursion to $-\sqrt a_\ell$ on $[t_j^n,t_{j+1}^n]$ then (A) and (B) are satisfied. Since there are $2^n$ different values $j$, the statement follows.
\end{proof}

\section{Neumann boundary conditions}\label{Sec:Neumann}

When the operator is endowed with Neumann boundary conditions, we have $\varphi_k'(0)=\varphi_k'(L)=0$ so that necessarily $\varphi_k(0),\varphi_k(L) \ne 0$ (otherwise, $\varphi_k$ would be identically zero). Consequently, the corresponding Riccati transforms $\Chi_{k}^{(N)}$ start and end at $0$. Therefore, we need to start our diffusions from $0$ as well. To avoid confusions, we let $X_a^{(N)}$ and $\hat{X}_a^{(N)}$ be the analogues of $X_a$ and $\hat{X}_a$ but starting from $0$. Notice that we have the following almost sure equivalences:\begin{itemize}
\item $a > -\lambda_1^{(N)}$ if and only if $X_a^{(N)}$ does not explode on $[0,L]$ and $X_a^{(N)}(L) > 0$,
\item $-\lambda_k^{(N)} \ge a > -\lambda_{k+1}^{(N)}$ if and only if either $\Big[X_a^{(N)}$ explodes $k$ times on $[0,L]$ and $X_a^{(N)}(L) > 0\Big]$ or $\Big[X_a^{(N)}$ explodes $k-1$ times on $[0,L]$ and $X_a^{(N)}(L) \le 0\Big]$.
\end{itemize}

\medskip

The strategy of proof is exactly the same as in the case of Dirichlet boundary conditions: we decompose the interval $[0,L]$ into subintervals of size $2^{-n}L$, and, within these subintervals, we still consider the diffusions $X_a^{j}$ that start from $+\infty$ (and not from $0$!) at time $j2^{-n}L$. The only changes in the proof consist in dealing with the boundary conditions of $X_a^{(N)}$ and $\hat{X}_a^{(N)}$. One needs to show that with large probability we have:
\begin{enumerate}
\item The diffusion $X_a^{(N)}$ reaches a neighborhood of $\sqrt a$ very quickly:
$$ \tau_{x}(X_a^{(N)}) = (3/8)  t_a(1+o(1))\;,$$
where $x = \sqrt a - \frac{\ln a}{a^{1/4}}$ and
$$ \sup_{t\in (0,\frac38 t_a]} \Big|X_a^{(N)}(t)-\sqrt{a} \tanh (\sqrt a t)\Big| \le 1\;.$$
\item The diffusion $X_a^{(N)}$ synchronizes with $X_a^{j}$ for every $j=0,\ldots,2^n-1$: the content of Proposition \ref{Propo:DiffClose} remains true upon replacing $X_a$ by $X_a^{(N)}$ and $t_0$ by $j2^{-n}L$.
\item If $X_a^{(N)}$ does not explode more than $k$ times, then $X_a^{(N)}(L) > 0$.
\end{enumerate}
The proofs of the two first estimates are essentially the same as those for Dirichlet boundary conditions. The proof of the third estimate is a simple consequence  of the synchronization with the stationary diffusion proved in Lemma \ref{Lemma:DiffXY} for Dirichlet boundary conditions, which can be adapted to the Neumann case, and of the fact that, for a stationary diffusion $Y$ the probability that $Y(L) > \sqrt a / 2$ is overwhelming.\\
This being given, one considers a more restrictive event $\cE(n,\eps)$ than previously: one also imposes that for all $a\in M_L$, $X_a^{(N)}$ satisfies the above bounds and \ref{(b)-(ii)}, and similarly for $\hat{X}_a^{(N)}$. This event $\cE(n,\eps)$ has a probability of order at least $1-O(\epsilon)$.\\
We now work on $\cE(n,\eps)$. We know that for every $a\in M_L$, the sum over $j$ of the number of explosions of $X_a^{j}$ coincides with the number of explosions of $X_a$. Thanks to the additional properties collected above, this is also true for $X_a^{(N)}$. Since in addition $X_a^{(N)}(L) > 0$, the r.v. $a_{i}$ and $a_{i}'$, which were defined in Section \ref{Sec:ProofsTh}, satisfy $a_i \le -\lambda_i^{(N)} < a'_i$. Since $a_{i}'-a_{i} = \epsilon/\sqrt a_L$, this immediately implies that $(\lambda_{i}^{(N)}-\lambda_{i}) \sqrt a_L$ goes to $0$ in probability. Then, the same monotonicity arguments as for the Dirichlet case allow to bound $\Chi_{i}^{(N)}$ using the diffusions $X_{a_i}^{(N)}$, $X_{a_i'}^{(N)}$ and their time reversals. In particular, the location of the maximum of $\varphi_i^{(N)}$ is still very close to the point $\upsilon_*$, defined as the last zero of $X_{a_i}^{j_*}$ where the time-interval $[j_*2^{-n}L, (j_*+1)2^{-n}L]$ corresponds to the additional explosion of $X_{a_i}^{(N)}$. This concludes the proof of Theorem \ref{Th:Neumann}.

\bibliographystyle{Martin}
\bibliography{library}

\end{document}